\documentclass[reqno,twoside,english,11pt]{article}

\usepackage{amsmath,amsfonts,calrsfs,fullpage,amssymb,xcolor,verbatim,eucal,yfonts,mathrsfs}

\usepackage{MnSymbol}
\usepackage[english,french]{babel}

\usepackage{enumerate}
\usepackage[shortlabels]{enumitem}

\usepackage{mathtools}

\usepackage{color}
\usepackage{babel}
\usepackage{mathrsfs}
\usepackage{amsthm}
\usepackage[unicode=true,bookmarks=true,bookmarksnumbered=true,bookmarksopen=false,breaklinks=false,pdfborder={0 0 1},backref=false,colorlinks=true]{hyperref}
\hypersetup{
	pdfauthor={M Xie}}

\newtheorem{Theorem}{Theorem}[section]
\newtheorem{Definition}[Theorem]{Definition}

\newtheorem{Remark}[Theorem]{Remark}

\newtheorem{Hypothesis}{Hypothesis}

\numberwithin{equation}{section}
\theoremstyle{plain}
\newtheorem{thm}{\protect\theoremname}[section]
\theoremstyle{plain}
\newtheorem{prop}[thm]{\protect\propositionname}
\theoremstyle{plain}
\newtheorem{lem}[thm]{\protect\lemmaname}
\theoremstyle{remark}

\theoremstyle{plain}
\newtheorem{cor}[thm]{\protect\corollaryname}
\theoremstyle{remark}
\newtheorem*{rem*}{\protect\remarkname}

\def\r2{\mathbb{R}^2}
\def\le{\left}
\def\r{\right}

\def\d{\delta}

\def\e{\epsilon}

\def\si{\sigma}

\def\L{{\mathcal L}}
\def\O{{\mathcal O}}

\def\F{{\mathcal F}}

\def\s0t{\sup_{t \in [0,T]}}

\def\ds{\displaystyle}

\def\beq{\begin{equation}}
\def\eeq{\end{equation}}
\def\barr{\begin{array}}
\def\earr{\end{array}}
\def\vs{\vspace{.01mm}   \\}
\def\rd{\reals\,^{d}}

\newcommand{\E}{\mathbb E}

\newcommand{\norm}[1]{\left\| #1\right\|}

\newcommand{\inner}[1]{\langle #1\rangle}
\newcommand{\Inner}[1]{\big\langle #1\big\rangle}

\newcommand{\abs}[1]{\lvert #1\rvert}
\newcommand{\dInner}[1]{\llangle #1\rrangle}

\def\E{{\mathbb{E}}}
\def\P{{\mathbb{P}}}

\DeclareMathOperator*{\esssup}{ess\,sup}

\usepackage{babel}
\allowdisplaybreaks[4]

\date{}

  \providecommand{\corollaryname}{Corollary}
  
  \providecommand{\lemmaname}{Lemma}
\providecommand{\theoremname}{Theorem}

\theoremstyle{plain}

\makeatother

\providecommand{\corollaryname}{Corollary}
\providecommand{\lemmaname}{Lemma}
\providecommand{\propositionname}{Proposition}
\providecommand{\remarkname}{Remark}
\providecommand{\theoremname}{Theorem}

\begin{document}
\global\long\def\divg{{\rm div}\,}%

\global\long\def\curl{{\rm curl}\,}%

\global\long\def\rt{\mathbb{R}^{3}}%

\global\long\def\rd{\mathbb{R}^{d}}%

\global\long\def\rtwo{\mathbb{R}^{2}}%

\global\long\def\e{\epsilon}%

\title{Well-posedness and invariant measure for quasilinear parabolic SPDE on a bounded domain }

\author{ Mengzi Xie\thanks{Department of Mathematics, University of Maryland, mxie2019@umd.edu.}}

\maketitle

\selectlanguage{english}

\date{}

\begin{abstract}
	We study quasilinear parabolic stochastic partial differential equations with general multiplicative noise on a bounded domain in $\mathbb{R}^{d}$, with homogeneous Dirichlet boundary condition. We establish the existence and uniqueness of solutions in a $L^{1}$ setting, and we prove a comparison result and an $L^{1}$-contraction property for the solutions. In addition, we show the existence of an invariant measure in case of non-degenerate diffusion. Finally, we show the uniqueness and ergodicity of the invariant measure in $L^{1}$, in case of bounded diffusion and additive noise.

\end{abstract}

\section{Introduction}
\label{sec_introduction}

In this article we consider the following quasilinear stochastic parabolic equation on a bounded smooth domain $\mathcal{O}\subset \mathbb{R}^d$, $d\geq1$, endowed with the homogeneous Dirichlet boundary condition, 
\begin{equation}\label{limit_para}
	\le\{\begin{array}{l}
		\ds{\partial_{t}u(t,x)=\text{div}\big(b(u(t,x))\nabla u(t,x)\big)+\sigma(u(t,x))\partial_{t}w(t,x), }\\[10pt]
		\ds{u(0,x)=u_{0}(x),\ \ \ u(t,\cdot)|_{\partial\mathcal{O}}=0 },
	\end{array}\r.
\end{equation}
where $u_{0}\in L^{1}(\O)$. The diffusion function $b\in C^{1}(\mathbb{R})$ is positive and can be degenerate. The coefficient $\sigma$ is assumed to be Lipschitz continuous and the noise $w(t)$ is a cylindrical Wiener process.
 
The study of quasilinear parabolic equations has many important applications, including, for instance, two or three phase flows in porous media (see, e.g., \cite{chavent1986}). The addition of a stochastic perturbation is often used to account for numerical, empirical or physical uncertainties. In the deterministic case, the well-posedness theory for entropy solutions of quasilinear anisotropic degenerate parabolic equations has been widely studied (e.g., \cite{ben2004}, \cite{benilan1996mild}-\cite{carrillo1999}, \cite{chen2003} and \cite{liwang2012}). In the literature we just cited it is essential that the solutions $u$ possess the regularity $\nabla \mathfrak{b}(u)\in L^{2}$, where $\mathfrak{b}(u):=\int_{0}^{u}b(\xi)d\xi$. This excludes the possibility of considering general $L^{1}$ initial data, since in this case we cannot expect a solution to be more than $L^{1}$. To address this issue, Chen and Perthame (see \cite{chen2003}) introduced the notion of kinetic solution and provided a well-posedness theory for general degenerate parabolic-hyperbolic equations with $L^{1}(\mathbb{R}^{d})$ initial data. An alternative approach was given by Bendahmane and H.Karlsen (ref. \cite{ben2004}) based on a notion of renormalized entropy solution and the technique of doubling variables, which was introduced by Kruzhkov (see \cite{kruzkov1970}) and developed by Carrillo (see \cite{carrillo1994}, \cite{carrillo1999}).

In the stochastic case, a particular example of equation \eqref{limit_para} is given by the following stochastic porous media equation 
\begin{equation}\label{porous_media}
	du=\Delta\big[\text{sgn}(u)\abs{u}^{\theta}\big]dt+\sigma(u)dw,
\end{equation}
with $\theta>2$. Stochastic (generalized) porous media equations have attracted a lot of interest in recent years (see, e.g., \cite{barbu2006weak}, \cite{barbu2008nonnegative}, \cite{barbu2009existence}, \cite{daprato2006Strong} and the references therein). Existence and uniqueness of solutions for \eqref{porous_media} for a general class of monotone functions $b$ have been established, as well as the corresponding results about the ergodicity (ref. \cite{barbu2006weak}, \cite{barbu2010Invariant}, \cite{daprato2003Invariant}, \cite{daprato2006Strong} and the references therein). 
All of these results rely on an $H^{-1}$ approach, which consists in treating $\text{div}(b(\cdot)\nabla\cdot)$ as a monotone operator in $H^{-1}$. In fact, unlike in the deterministic case, there are few results dealing with the well-posedness of \eqref{limit_para} in a $L^{1}$ setting. The reason is that, especially in the case when the diffusion function $b$ is degenerate, the usual proof of existence of mild solutions relying on the Crandall-Liggett theory of $m$-accretive operators in $L^{1}$ cannot be applied to the stochastic case (ref. \cite{carrillo1999entropy}, \cite{crandall1972} and \cite{crandall1971}). In \cite{debussche16} Debussche, Hofmanova and Vovelle have studied the following Cauchy problem for a quasilinear degenerate parabolic-hyperbolic SPDE on the d-dimensional torus $\mathbb{T}^{d}$:
\begin{equation}\label{intro1}
	\partial_{t}u+\text{div}(A(u))=\text{div}\big(b(u)\nabla u\big)dt+\sigma(u)\partial_{t}w,\ \ \ u(0)=u_{0},
\end{equation}
where the flux function $A:\mathbb{R}\to\mathbb{R}^{d}$ is of class $C^{2}$, the diffusion  $b$ is defined in $\mathbb{R}$ with values in symmetric and positive semidefinite matrices, and its square-root matrix  is assumed to be locally $\gamma$-H{\"o}lder continuous for some $\gamma>1/2$. They have established the existence and uniqueness of a kinetic solution for \eqref{intro1}, with initial data  $u_{0}\in \bigcap_{p\geq1}L^{p}(\mathbb{T}^{d})$, under some restrictive assumptions on $A$ and $b$. More precisely, boundedness of function $b$ and polynomial growth of $A''$ had to be assumed.  More recently in \cite{gess18}, by introducing a generalized definition of kinetic solution and imposing a nondegeneracy condition for the symbol $\L$ associated to the kinetic form of \eqref{intro1}
\begin{equation*}
	\mathcal{L}(iu,in,\xi):=i\big(u+a(\xi)\cdot n\big)+n^{\ast}b(\xi)n,\ \ \ a:=\nabla A,
\end{equation*}
Gess and Hofmanova obtained new regularity results  based on averaging techiques, which allowed them to generalize the previous results in \cite{debussche16} to the case when $u_{0}\in L^{1}(\mathbb{T}^{d})$, without assuming growth conditions on the nonlinearities $A$ and $b$. In particular, they established a well-posedness theory for \eqref{intro1} in a full $L^{1}$ setting on $\mathbb{T}^{d}$. Moreover, in both \cite{debussche16} and \cite{gess18}, a comparison result and the so-called $L^{1}$-contraction property for kinetic solutions were proved on $\mathbb{T}^{d}$. Namely, it was shown that if $u_{1}$, $u_{2}$ are kinetic solutions to \eqref{intro1}, with initial data $u_{1,0},u_{2,0}\in L^{1}(\mathbb{T}^{d})$, respectively, then
\begin{equation}
	\label{contraction_intro}
	\esssup_{t\in[0,T]}\ \E\norm{(u_{1}(t)-u_{2}(t))^{+}}_{L^{1}(\mathbb{T}^{d})}\leq \norm{(u_{1,0}-u_{2,0})^{+}}_{L^{1}(\mathbb{T}^{d})}.
\end{equation}
Concerning equation \eqref{limit_para} on bounded domain $\O\subset \mathbb{R}^{d}$, there are only few results about its well-posedness in $L^{p}$, for $p\geq2$. In \cite{fridnew}, the Cauchy problem \eqref{intro1} has been studied in $\O\subset\mathbb{R}^{d}$, with a non-homogeneous boundary condition. Under conditions similar to whose assumed as in \cite{debussche16} for nonlinearities $A$ and $b$, and under the nondegeneracy condition assumed in \cite{gess18} for the symbol $\mathcal{L}$, the existence and uniqueness of bounded solutions for \eqref{intro1} in $L^{\infty}$ were was proved, for every initial condition $u_{0}$ in $L^{\infty}(\O)$ such that $u_{\min}\leq u_{0}\leq u_{\max}$. Finally, under non-degenerate assumption on $b\in C^{1}(\mathbb{R})$
\begin{equation}\label{particular}
	0<b_{0}\leq b(r)\leq b_{1},
\end{equation}
Cerrai and Xie \cite{cerraixie2023} proved the existence and uniqueness of a solution for \eqref{limit_para} in
\begin{equation*}
	L^{2}(\Omega;C([0,T];L^{2}(\O))\bigcap L^{2}([0,T];W^{1,2}_{0}(\O))
\end{equation*} 
for every initial condition $u_{0}\in L^{2}(\O)$.

However, to the best of our knowledge, there are currently no results about the well-posedness for \eqref{limit_para} in a $L^{1}$ setting on a bounded domain $\O\subset\mathbb{R}^{d}$. As in the periodic case, the $L^{1}$ framework for equation like \eqref{limit_para} offers several advantages. For instance, contration properties in $L^{1}$ norm are sometimes better than those in $H^{-1}$ norm, and moreover, they would allow to study invariant measures in $L^{1}$ (see, e.g., \cite{chen2019}).

In this paper we aim to study the well-posedness of \eqref{limit_para} in a $L^{1}$ setting. The notion of kinetic solution was introduced by Lions, Perthame and Tadmor (see \cite{lions1994}) for deterministic first-order scalar conservation laws, and was then extended to the case of SPDEs by Debussche and Vovelle (see \cite{debussche14}, \cite{debussche15}). As in the deterministic case, when the initial data $u_{0}$ are regular enough (for example, when $u_{0}\in L^{\infty}(\O)$), every (weak) solution is expected to have the regularity $b(u)\nabla u\in L^{2}(\O)$, and thus it is also a kinetic solution. But with the only assumption that the intial data $u_{0}$ belongs to $L^{1}(\O)$, we cannot expect even $\sqrt{b(u)}\nabla u$ to be square-integrable, so the usual kinetic solution (see, e.g., \cite{debussche16}) may not be well-defined because of the lack of integrability. For this reason, we adapt the defintion of kinetic solution in \cite{gess18} to our situation and introduce a notion of generalized kinetic solution, which is weaker than usual kinetic solution defined in \cite{fridnew} (see also \cite{debussche16}) and is suitable in the $L^{1}$ setting.

  Concerning the diffusion function $b:\mathbb{R}\to\mathbb{R}$, we assume that it is continuously differentiable and $\sqrt{b}$ is locally $\gamma$-H{\"o}lder continuous for some $\gamma>1/2$. Then, under some suitable assumptions on $\sigma$, we prove a comparison result and an $L^{1}$-contraction property for generalized kinetic solutions. Namely, if $u_{1},u_{2}$ are generalized kinetic solutions to \eqref{limit_para} with initial data $u_{1,0},u_{2,0}\in L^{1}(\O)$, respectively, then
  \begin{equation}
  	\esssup_{t\in[0,T]}\ \E\norm{(u_{1}(t)-u_{2}(t))^{+}}_{L^{1}(\O)}\leq \norm{(u_{1,0}-u_{2,0})^{+}}_{L^{1}(\O)}.
  \end{equation} 
  In particular, this result implies that the generalized kinetic solution to \eqref{limit_para} is unique, and so does either the kinetic or the weak solution. Our proof for this contractive property is based on a modification of the method of doubling variables introduced in \cite{gess18} for the torus $\mathbb{T}^{d}$. Since we are here considering Dirichlet boundary conditions, a localized version is needed and we have to make a use of the technique of boundary layer sequences.
  
  To obtain the existence of generalized kinetic solutions to \eqref{limit_para} with initial data in $L^{1}(\O)$, we start with the well-posedness for \eqref{limit_para} with initial data in $L^{p}(\O)$, for $p\geq2$. In the present paper we introduce two different types of conditions on the diffusion function $b$. One of them is a non-degenerate condition, in the sense that there exist some $c>0$, $b_{0}>0$ and $\theta\geq1$ such that
\begin{equation}\label{non_deg_intr}
	b_{0}\leq b(r)\leq c\big(1+\abs{r}^{\theta-1}\big).
\end{equation} 
Under this condition, we prove that given any initial condition $u_{0}\in L^{2\theta}(\O)$, there exists a unique weak solution $u\in L^{2\theta}(\Omega;L^{\infty}([0,T];L^{2\theta}(\O)))\cap L^{2}(\Omega;L^{2}([0,T];W^{1,2}_{0}(\O)))$ to \eqref{limit_para}, which has almost surely continuous trajectories in $L^{2\theta}(\O)$. As we can see, this result is a generalization of the results established in \cite[Proposition A.3]{cerraixie2023}, where the well-posedness for \eqref{limit_para} was proved in $L^{2}(\O)$ under condition \eqref{particular}, with $\theta=1$ in condition \eqref{non_deg_intr}. The other condition on $b$ that we provide is a degenerate one. Namely, we assume that there exist $\theta_{2}\geq \theta_{1}>2$ and $c_{1},c_{2}>0$ such that
\begin{equation*}
	c_{1}\abs{r}^{\theta_{1}-1}\leq b(r)\leq c_{2}\big(1+\abs{r}^{\theta_{2}-1}\big).
\end{equation*}  
This, in particular, allows us to consider stochastic (generalized) porous media equations (see, e.g., \cite{daprato2003Invariant}). Under this condition, we prove that given any initial condition $u_{0}\in L^{2\theta_{2}}(\O)$, there exists a unique kinetic solution $u\in L^{2\theta_{2}}(\Omega;L^{\infty}([0,T];L^{2\theta_{2}}(\O)))$ to \eqref{limit_para}, which has almost surely continuous trajectories in $L^{2\theta_{2}}(\O)$. Our approach is based on an approximation for \eqref{limit_para} where $b(u)$ is replaced by $b(u)+\tau$, for $\tau>0$ small, so that the results obtained under the non-degenerate condition can be applied to the approximating problem. The main difficulty lies in the fact that, there is no possibility to obtain uniform bounds in $W^{1,2}_{0}(\O)$ for the solutions of the approximating problems because of the degeneracy on $b$. To overcome this difficulty, our strategy is to first reduce the regurity problem to the periodic case by a localization, which was introduced in \cite{frid22} (also see \cite{fridnew}), and then apply the stochastic averaging lemma established by Gess and Hofmanova in \cite{gess18}. This will allow us to construct a kinetic solution by a compactness argument. Then, based on the $L^{1}(\O)$ contractive property of generalized kinetic solutions and the well-posedness that we established in $L^{p}(\O)$ for $p\geq2$, we prove that there is a unique generalized kinetic solution $u\in L^{1}(\Omega;L^{\infty}([0,T];L^{1}(\O)))$ to \eqref{limit_para} with initial data $u_{0}\in L^{1}(\O)$. Moreover, we show that $u$ has almost surely continuous trajectories in $L^{1}(\O)$.

After we have established the well-posedness for \eqref{limit_para} in $L^{1}(\O)$, we study the existence and uniqueness of invariant measures for $P_{t}$, the transition semigroup associated with \eqref{limit_para} in $L^{1}(\O)$. To this purpose, the existence and the uniqueness of an invariant meausure in $L^{2}(\O)$ under condition \eqref{particular} has been established in \cite{cerraixie2023} through an $H^{-1}$ approach. More specifically, it has been proven that if $u_{1},u_{2}$ are two generalized kinetic solutions to \eqref{limit_para}, with initial conditions $u_{1,0},u_{2,0}$, there exists some positive constant $\lambda>0$ such that
\begin{equation}\label{contraction_dual}
	\E\norm{u_{1}(t)-u_{2}(t)}_{H^{-1}}\leq e^{-\lambda t}\norm{u_{1,0}-u_{2,0}}_{H^{-1}},\ \ \ \ t\geq0.
\end{equation}
 In contrast with this strongly contractive property in $H^{-1}$, we prove that 
\begin{equation}\label{contraction_intr}
	\E\norm{u_{1}(t)-u_{2}(t)}_{L^{1}(\O)}\leq \norm{u_{1,0}-u_{2,0}}_{L^{1}(\O)},\ \ \ \ t\geq 0,
\end{equation}
and this, in particular, 
implies the Feller property of the Markovian semigroup $P_{t}$ in $L^{1}(\O)$. In the present paper, by the Krylov-Bogoliubov approach, we prove that if the diffusion $b$ satisfies the non-degenerate condition \eqref{non_deg_intr} and the coefficient $\sigma$ has sub-linear growth, then $P_{t}$ admits at least one invariant measure on $L^{1}(\O)$ supported in $W^{1,2}_{0}(\O)$. The study of the uniqueness and the ergodicity is relatively more challenging, because in contrast to the case in $H^{-1}$ where the stronger contraction \eqref{contraction_dual} holds, $L^{1}$-contraction \eqref{contraction_intr} does not guarantee the uniquness of invariant measures in $L^{1}(\O)$. Even worse is that, the classical approach (see, e.g., \cite{cerraiholtz2020}) by an argument with the (asymptotically) strong Feller property for Markovian semigroups seems not possible in the $L^{1}$ setting. For this reason we use the coupling method (see \cite{doeblin1937}) and an argument of smallness of noise, which was used in \cite{debussche15} (also see \cite{chen2019}). Then, relying on the fact that in the case of additive noise, $L^{1}$-contraction \eqref{contraction_intr} holds in the almost sure sense, we show that if the noise in \eqref{limit_para} is additive and the diffusion $b$ is bounded, then any pairs of generalized kinetic solutions will be arbitrarily close to each other after a long enough time almost surely (see Theorem \ref{uniqueness_inv}). In particular, this implies the uniqueness of the invariant measure for $P_{t}$ on $L^{1}(\O)$.

{\em Organization of the paper:} In Section \ref{sec_assumption}, we introduce all notations and assumptions. In Section \ref{sec_main_result} we first introduce the notion of weak and kinetic solution to \eqref{limit_para} in the $L^{p}$ setting, and give a definition of generalized kinetic solution which is suitable for the $L^{1}$ setting. Then we state the main result of the paper. In Section \ref{L1}, we introduce a localized version of doubling of variables (see Lemma \ref{contraction}). Then, we apply it to prove a comparison result and an $L^{1}$ contraction property for solutions to \eqref{limit_para}. In Section \ref{sec_wellposedness_Lp}, we prove the well-posedness of equation \eqref{limit_para} in $L^{p}$, for $p\geq2$. We also provide some energy estimates, which are needed in our construction of solutions in $L^{1}(\O)$. In Section \ref{sec_proof_main_thm}, we establish the well-posedness of equation \eqref{limit_para} given initial data in $L^{1}(\O)$. In Section \ref{sec_inv}, we study the existence and the uniqueness of the invariant measure for the transition semigroup associated to \eqref{limit_para} in $L^{1}(\O)$. Finally, in Appendix \ref{sec_appendix} we give the proof of Lemma \ref{contraction}.

\section{Assumptions and notations}\label{sec_assumption}

Throughout the present paper $\mathcal{O}$ is an open bounded domain in $\mathbb{R}^d$, with smooth boundary. We denote by $H$ the Hilbert space $L^2(\mathcal{O})$ and by $\langle \cdot , \cdot \rangle_H$ the corresponding  inner product. We use $\dInner{\cdot,\cdot}$ to denote the duality between the space of distribution over $\mathcal{O}\times\mathbb{R}$ and $C_{c}^{\infty}(\mathcal{O}\times\mathbb{R})$.For every $s>0$ and $p\geq1$, $W^{-s,p}_{0}(\O)$ is the completion of $C_c^\infty(\mathcal{O})$ with respect to the norm on $W^{s,p}(\O)$, and $W^{-s,p}(\O)$ is the dual space to $W^{s,p}_{0}(\O)$. It is well-known that for every $s\in\mathbb{R}$ and $p\geq1$, $W^{s,p}(\O)$ is a complete separable metric space, and the embeddings
\begin{equation*}
	W^{s_{1},p_{2}}(\O)\subset W^{s_{2},p_{2}}(\O),\ \ \ \ s_{1},s_{2}\in\mathbb{R},\ \ p_{1},p_{2}\in[1,+\infty),
\end{equation*}
are compact provided $s_{1}>s_{2}$ and $1/p_{1}-s_{1}/d<1/p_{2}-s_{2}/d$.

 Given the domain $\mathcal{O}$, we denote by $\{e_i\}_{i\in\mathbb{N}}\subset H^1:=W^{1,2}_{0}(\O)$ the complete orthonormal basis of $H$ which diagonalizes the Laplacian $A:=\Delta$, endowed with Dirichlet boundary conditions on $\partial\mathcal{O}$. Moreover, we denote by $\{-\alpha_i\}_{i\in \mathbb{N}}$ the corresponding sequence of eigenvalues, i.e.
\[A e_i=-\alpha_i e_i,\ \ \ \  i\in \mathbb{N}.\]
Next, for every $\d>0$ we denote $H^{\delta}:=W^{\delta,2}_{0}(\O)$, the completion of $C^{\infty}_{c}(\O)$ with respect to the norm  
\[\Vert u\Vert_{H^\delta}^2:=\sum_{i=1}^\infty \alpha_i^\delta \langle u,e_i\rangle_H^2,\]
and denote by $H^{-\delta}$ the dual space to $H^{\d}$. This implies, in particular, that $H^{1}$ is the completion of $C^{\infty}_{c}(\O)$ with respect to the norm $\norm{u}_{H^{1}}^{2}=\norm{\nabla u}_{H}^{2}$, and 
\begin{equation*}
	\norm{u}_{H}\leq \frac{1}{\sqrt{\alpha_{1}}}\norm{u}_{H^{1}},\ \ \ \ u\in H^{1}.
\end{equation*}

\subsection{About the boundary $\partial\O$}
\label{sec_boundary}

In order to take advantage of the fact that the boundary $\partial\mathcal{O}$ is locally the graph of a $C^{2}$ function, we introduce a system $\mathfrak{B}$ of balls (ref. \cite{frid17}, \cite{mascia2002}) as follows. An open ball $B=B(x_{0},r)$, with center at an arbitrary $x_{0}\in \partial\mathcal{O}$ and with some radius $r>0$, belongs to $\mathfrak{B}$, provided that, in some local coordinates $x=(x',x_{d})$, there exists some $h\in C^{2}(\mathbb{R}^{d-1})$ such that
\begin{equation*}
	B\cap \mathcal{O}=\Big\{x\in B:x_{d}<h(x')\Big\},\ \ \ \ B\cap\partial\mathcal{O}=\Big\{x\in B: x_{d}=h(x')\Big\}.
\end{equation*}
Note that we can choose a finite open covering $\{B_{i}\}_{i=0}^{N}$ of $\overline{\O}$ such that $B_{i}\in\mathfrak{B}$, $1\leq i\leq N$, and $B_{0}\subset\subset \O$.

Now, we recall the definition of a boundary layer sequence (ref. \cite{frid17}, \cite{frid22}, \cite{mascia2002}).
\begin{Definition}
	We call $\{\zeta_{\delta}\}$ an $\mathcal{O}$-boundary layer sequence if, for each $\delta>0$, $\zeta_{\delta}\in C(\overline{\mathcal{O}})\bigcap C^{2}(\O)$, $0\leq \zeta_{\delta}\leq1$, $\zeta_{\delta}(x)\to1$ for every $x\in\mathcal{O}$, as $\delta\to0$, and $\zeta_{\delta}=0$ on $\partial\mathcal{O}$.
\end{Definition}

The name of boundary-layersequnce for $\zeta_{\delta}$ is justified by the fact that the behavior of $\zeta_{\delta}$ as $\delta$ tends to zero gives an account of formation of boundary layers; indeed, we have $-\nabla\zeta_{\delta}$ converging to the outward normal $\nu$ of the boundary, and for any $\varphi\in (W^{1,2}(\O))^{d}$ we have 
\begin{equation}
	\label{boundary_laryer_property}
	\lim_{\delta\to0}\int_{\O}\varphi\cdot\nabla\zeta_{\delta}dx=-\lim_{\delta\to0}\int_{\O}\text{div}(\varphi)\zeta_{\delta}dx=-\int_{\O}\text{div}(\varphi)dx=-\int_{\partial\O}\varphi\cdot\nu\ dS.
\end{equation}

Here is an example for the $\O$-boundary layer sequence (ref. \cite[Section 5]{mascia2002}). Let $\zeta_{\delta}$ be the unique solution of 
\begin{equation*}
		\le\{\begin{array}{l}
		\ds{-\delta^{2}\Delta\zeta_{\delta}+\zeta_{\delta}=1,\ \ \ \ x\in\O }\\[10pt]
		\ds{\zeta_{\delta}=0,\ \ \ \ x\in\partial\O }.
	\end{array}\r.
\end{equation*}
Then $0\leq\zeta_{\delta}\leq 1$ is an $\O$-boundary layer sequence. Moreover, we have $\Delta_{x}\zeta_{\delta}\leq 0$.


\subsection{About the stochastic term}
\label{sec_stochastic_term}

Let $w(t)$ be a cylindrical Wiener process, defined on a complete stochastic basis  $(\Omega,\mathcal{F},(\mathcal{F}_t)_{t\geq 0},\mathbb{P})$. This means that $w(t)$ can be formally written as
\[
w(t)=\sum_{i=1}^\infty  \beta_i(t)\tilde{e}_{i},
\]
where $\{\beta_i\}_{i\in \mathbb{N}}$ is a sequence of independent standard Brownian motions on $(\Omega,\mathcal{F},(\mathcal{F}_t)_{t\geq 0},\mathbb{P})$,   $\{\tilde{e}_i\}_{i\in\,\mathbb{N}}$ is the complete orthonormal system in a separable in a Hilbert space $\mathcal{U}$. Moreover, if $\mathcal{U}_{0}$ is any Hilbert space containing $\mathcal{U}$ such that the embedding of $\mathcal{U}$ into $\mathcal{U}_{0}$ is Hilbert-Schmidt, we have that 
\begin{equation}
	\label{contb}w \in\,C([0,T];\mathcal{U}_{0}),\ \ \ \P\text{-a.s.}
\end{equation}

Next, we recall that for every two separable Hilbert spaces $E$ and $F$, $L_2(E,F)$ denotes the space of Hilbert-Schmidt operators from $E$ into $F$. $L_2(E,F)$ is a Hilbert space, endowed with the inner product
\[\langle A,B\rangle_{L_2(E,F)}=\mbox{Tr}_E\,[A^\star B]=\mbox{Tr}_F[B A^\star].\]
As well known, $L_2(E,F) \subset L(E,F)$ and
\begin{equation}
	\label{sm5}
	\Vert A\Vert_{L(E,F)}\leq 	\Vert A\Vert_{L_2(E,F)}.
\end{equation}

\bigskip

\begin{Hypothesis}\label{Hypothesis1}
	For each $h\in H$, the mapping $\si(h):\mathcal{U}\to H$ is defined by 
	\[[\sigma(h)\tilde{e}_i](x) = \sigma_i(x,h(x)), \ \ \ \ x \in\,\mathcal{O},\ \ \ \ i \in\,\mathbb{N},\] 
	for some mapping $\sigma_i\in C(\overline{\O}\times\mathbb{R})$. We assume that there exists a sequence $(\lambda_{i})_{i\in\mathbb{N}}$ of positive numbers satisfying $D:=4\sum_{i=1}^{\infty}\lambda_{i}^{2}<\infty$ such that 
	\begin{equation}
		\label{sgfine}
		\abs{\sigma_{i}(x,0)}+\abs{\nabla_{x}\sigma_{i}(x,\xi)}+\abs{\partial_{\xi}\sigma_{i}(x,\xi)}\leq \lambda_{i},\ \ \ \  x\in\O,\ \ \ \xi\in\mathbb{R}.
	\end{equation}
	Moreover, we assume that 
	\begin{equation}
		\label{sigma_boundary}
		\sigma_{i}(x,\xi)=0,\ \ \ \  x\in\partial{\O},\ \ \ \xi\in\mathbb{R}.
	\end{equation}
	
\end{Hypothesis}

\begin{Remark} Condition \eqref{sgfine} implies that for all $x,y\in\O$, and $\xi,\zeta\in\mathbb{R}$
	\begin{equation}
		\label{sgfine1}	
		\sum_{i=1}^\infty \vert \sigma_i(x,\xi) - \sigma_i(y,\zeta)\vert^2 \leq D\,\Big(\abs{x-y}^{2}+\vert \xi-\zeta\vert^{2}\Big),
	\end{equation}
	and 
	\begin{equation}
		\label{sgfine2}
		\Sigma^{2}(x,\xi):=\sum_{i=1}^{\infty}\abs{\sigma_{i}(x,\xi)}^{2}\leq D(1+\abs{\xi}^{2}),\ \ \ \  x\in\O,\ \ \ \xi\in\mathbb{R}.
	\end{equation}
	In particular, the mapping $\si:H\to L_{2}(\mathcal{U},H)$ is Lipschitz continuous, and 
	\begin{equation}
		\norm{\si(h)}_{L_{2}(\mathcal{U},H)}\leq c(1+\norm{h}_{H}),\ \ \  h\in H.
	\end{equation} Hence, given a predictable process $u\in L^{2}(\Omega\times[0,T];H)$, the stochastic integral
	\begin{equation*}
		\int_{0}^{t}\sigma(u(s))dw(s),\ \ \ t\in[0,T],
	\end{equation*}
	is a well-defined process taking values in $H$. 
\end{Remark}

\subsection{About the diffusion term}
\label{sec_diffusion}

We assume that the diffusion $b$ satisfies the following hypothesis.
\begin{Hypothesis}\label{Hypothesis2}
	The mapping $b$ belongs to $C^1(\mathbb{R})$, and $\sqrt{b}$ is locally $\gamma$-H{\"o}lder continuous for some $\gamma>1/2$, that is, for all $R>0$ there is a constant $C=C(R)$ such that
	\begin{equation}
		\label{holder_cont}
		\big\lvert\sqrt{b(r_{1})}-\sqrt{b(r_{2})}\big\rvert\leq C(R)\abs{r_{1}-r_{2}}^{\gamma},\ \ \ \forall r_{1},r_{2}\in\mathbb{R},\ \ \abs{r_{1}},\abs{r_{2}}\leq R.
	\end{equation}
\end{Hypothesis}

We shall also introduce two different types of conditions on $b$ as follows. 
\begin{Hypothesis}\label{Hypothesis3} We assume that $b\in C^{1}(\mathbb{R})$ and one of the following two conditions is satisfied:
	\begin{enumerate}[A]
		\item{(Non-degenerate condition).}\label{nondeg_condition} There exist some $\theta\geq 1$ and $b_{0},c>0$ such that
		\begin{equation}
			\label{nondeg_growth}
			b_{0}\leq b(r)\leq c(1+\abs{r}^{\theta-1}),\ \ \ \ \ r\in\mathbb{R}.
		\end{equation}
		
		\smallskip
		
		\item{(Degenerate condition).}\label{deg_condition}  There exist some $\theta_{2}\geq \theta_{1}>2$ and constants $c_{1},c_{2}>0$ such that
		\begin{equation}
			\label{deg_growth}
			c_{1}\abs{r}^{\theta_{1}-1}\leq b(r)\leq c_{2}(1+\abs{r}^{\theta_{2}-1}),\ \ \ \ \ r\in\mathbb{R}.
		\end{equation}
	\end{enumerate}
\end{Hypothesis}

In what follows, we shall denote
\[\mathfrak{b}(r)=\int_0^r b(\xi)\,d\xi,\  \ \ \ r \in\,\mathbb{R}.\]
One can immediately see that $\mathfrak{b}:\mathbb{R}\to\mathbb{R}$ is monotone and has at most polynomial growth.

\begin{Remark}
	\begin{enumerate}
		\item It is easy to check that, if the diffusion $b$ satisfies Hypothesis \ref{Hypothesis3}\ref{nondeg_condition}, then $\sqrt{b}$ is locally Lipschitz continuous, and thus Hypothesis \ref{Hypothesis2} is satisfied.
		
		\item A typical example of the diffusion $b$ satisfying both Hypotheses \ref{Hypothesis2} and \ref{Hypothesis3}\ref{deg_condition} is 
		\begin{equation*}
			b(r)=c\abs{r}^{\theta-1},\ \ \ \ \theta>2,\ \ \ c>0.
		\end{equation*}
		In this case, it can be verified that $\sqrt{b}$ is locally $\gamma$-H{\"o}lder continuous on $\mathbb{R}$, where 
		\begin{equation*}
			\gamma=1,\ \ \ \text{if}\ \ \theta\in[3,+\infty);\ \ \ \text{and}\ \ \ \gamma\in\big(1/2,(\theta-1)/2\big),\ \ \ \text{if}\ \ \theta\in(2,3).
		\end{equation*} 
	\end{enumerate}
\end{Remark}

Finally, we introduce a non-degeneracy condition on the symbol $\L$ (see \cite[Section 2.2]{gess18}) assiciated to the kinetic form of \eqref{limit_para}
\begin{equation*}
	\L(iu,in,\xi):=iu+b(\xi)n^{2},\ \ \ \ u\in\mathbb{R},\ \ n\in\mathbb{Z},\ \ \xi\in\mathbb{R}.
\end{equation*}
For $J,\delta>0$ and $\eta\in C_{b}^{\infty}(\mathbb{R})$ nonnegative, let 
\begin{equation*}
	\Omega_{\L}^{\eta}(u,n;\delta):=\big\{ \xi\in\text{supp}(\eta):\abs{\L(iu,in,\xi)}\leq \delta \big\},
\end{equation*}
\begin{equation*}
	\omega_{\L}^{\eta}(J;\delta):=\sup_{u\in\mathbb{R},n\in\mathbb{Z},\abs{n}\sim J}\abs{\Omega_{\L}^{\eta}(u,n;\delta) }
\end{equation*}
and $\L_{\xi}:=\partial_{\xi}\L$. Here we employ the notation $x\lesssim y$ if there exists a constant $C$ independent of the variables under consideration such that $x\leq Cy$ and we write $x\sim y$ if $x\lesssim y$ and $y\lesssim x$. By $x\lesssim_{z}y$ we mean that the corresponding proportional constant depends on $z$. As introduced in \cite{gess18}, a non-degeneracy condition on the symbol $\L$ is satisfied if there exists $\alpha\in(0,1)$, $\beta>0$ and a measurable map $\vartheta\in L_{\text{loc}}^{\infty}(\mathbb{R};[1,\infty))$ such that 
\begin{equation}\label{symbol_condition}
	\begin{array}{ll}
		\ds{\quad\quad\quad\quad\quad\quad\quad\omega_{\L}^{\eta}(J;\delta)\lesssim_{\eta}\Big(\frac{\delta}{J^{\beta}}\Big)^{\alpha}, }\\
		\vs
		\ds{\sup_{u\in\mathbb{R},n\in\mathbb{Z},\abs{n}\sim J}\sup_{\xi\in\text{supp}(\eta)}\frac{\abs{\L_{\xi}(iu,in,\xi)}}{\vartheta(\xi)}\lesssim_{\eta}J^{\beta}, \ \ \ \forall \delta>0,\ \ J\gtrsim 1 }.
	\end{array}
\end{equation} 
It can be immediately verified that (e.g., \cite[Section 2.4]{gess18}), if the mapping $b:\mathbb{R}\to\mathbb{R}$ satisfies Hypothesis \ref{Hypothesis3}\ref{deg_condition}, then for every $J,\delta>0$ and $\eta\in C^{\infty}_{c}(\mathbb{R})$, condition \eqref{symbol_condition} is satisfied with
\begin{equation*}
	\alpha=\frac{1}{\theta_{1}-1},\ \ \ \ \beta=2,\ \ \ \ \text{and}\ \ \ \ \vartheta(\xi):=1+\abs{\xi}^{\theta_{1}-2},\ \ \ \xi\in\mathbb{R}.
\end{equation*}

\section{The main result }
\label{sec_main_result}

\subsection{Weak and kinetic solution}

We start with the definition of weak solution to equation \eqref{limit_para}. 
\begin{Definition}
	[Weak solution] Let $u_{0}\in L^{p}(\O)$, for some $p\geq2$. We call an $(\mathcal{F}_{t})_{t\geq0}$ adapted process 
	\begin{equation*}
		u\in  L^{p}(\Omega;L^{\infty}([0,T];L^{p}(\O))) \bigcap L^{2}(\Omega;L^{2}([0,T];H^{1}))
	\end{equation*}
	is said to be a weak solution of equation \eqref{limit_para} if 
	\begin{equation*}
		\mathfrak{b}(u)=\int_{0}^{u}b(\xi)d\xi\in L^{2}(\Omega;L^{2}([0,T];H^{1})),
	\end{equation*}
	and for every test function $\varphi\in C^{\infty}_{0}(\mathcal{O})$
	\begin{equation}\label{weak_test}
		\begin{array}{ll}
			\ds{\Inner{u(t),\varphi}_{H} =\Inner{u_{0},\varphi}_{H}-\int_{0}^{t}\Inner{\nabla \mathfrak{b}(u(s)),\nabla\varphi}_{H}ds +\int_{0}^{t}\Inner{\varphi,\sigma(u(s))dw^{Q}(s)}_{H},\ \ \ \P\text{-a.s.} }.
		\end{array}
	\end{equation}
\end{Definition}

The following proposition is a localized version of the generalized It{\^o} formula proven in \cite{debussche16}:
\begin{prop}[Generalized It{\^o} formula]
	Let $u$ be a weak solution of equation \eqref{limit_para} with initial datum $u_{0}$, then for every $\phi\in C^{1}_{c}(\O)$ and every $\psi\in C^{2}(\mathbb{R})$ with $\psi''\in L^{\infty}(\mathbb{R})$
	\begin{equation}\label{gen_ito}
		\begin{array}{ll}
			&\ds{\int_{\O}\psi(u(t,x))\phi(x)dx =\int_{\O}\psi(u_{0}(x))\phi(x)dx-\int_{0}^{t}\int_{\O}\nabla \mathfrak{b}(u(s,x))\cdot \nabla\big(\psi'(u(s,x))\phi(x)\big)dxds }\\
			\vs
			&\ds{\quad\quad\quad\quad\quad +\frac{1}{2}\int_{0}^{t}\int_{\O}\psi''(u(s,x))\Sigma^{2}(x,u(s,x))\phi(x)dxds }\\
			\vs
			&\ds{\quad\quad\quad\quad\quad\quad\quad + \sum_{k=1}^{\infty}\int_{0}^{t}\int_{\O}\psi'(u(s,x))\sigma_{k}(x,u(s,x))\phi(x)dxd\beta_{k}(s),\ \ \ \P\text{-a.s.} }.
		\end{array}
	\end{equation}
\end{prop}

\begin{Remark}\label{remark_gen_ito}
In particular, if $\psi'(u)\in L^{2}(\Omega\times[0,T];H^{1})$, then by standard Sobolev spaces arguments (see, e.g., \cite[Section 5.5]{evans2010}), the test function $\phi$ in \eqref{gen_ito} can be chosen to be in $C^{1}(\overline{\O})$.
\end{Remark}

Next, we introduce the definitions of kinetic measure and kinetic solution to \eqref{limit_para} (see \cite{gess18} and \cite{fridnew}):
\begin{Definition}[Kinetic measure] \label{def_kinetic_measure}
	A mapping $m$ from $\Omega$ to $\mathcal{M}^{+}_{b}([0,T]\times\mathcal{O}\times\mathbb{R})$, the set of nonnegative bounded measure over $[0,T]\times\mathcal{O}\times\mathbb{R}$, is said to be a kinetic measure provided:
	\begin{enumerate}
		\item $m$ is measurable in the sense that, for each $\psi\in C_{c}([0,T]\times\O\times\mathbb{R})$ the mapping $m(\psi):\Omega\to\mathbb{R}$ is measurable.
		
		\smallskip
		
		\item For all $\psi\in C_{c}(\mathcal{O}\times\mathbb{R})$, the process
		\begin{equation*}
			\int_{[0,t]\times\mathcal{O}\times\mathbb{R}}\psi(x,\xi)dm(s,x,\xi)\in L^{2}(\Omega\times[0,T])
		\end{equation*}
		is predictable.
		
		\smallskip
		
		\item $m$ vanishes for large $\xi$ in the sense that, if $B_{R}^{c}=\{\xi\in\mathbb{R}:\abs{\xi}\geq R\}$, then
		\begin{equation*}
			\lim_{R\to\infty}\E m([0,T]\times\O\times B_{R}^{c})=0.
		\end{equation*}
	\end{enumerate}
\end{Definition}

\begin{Definition}[Kinetic solution]\label{def_kinetic_solution}
	 Let $u_{0}\in L^{p}(\mathcal{O})$, for some $p\geq2$. Assume that
	\begin{equation}
		u\in L^{p}(\Omega\times[0,T];L^{p}(\mathcal{O}))\bigcap L^{p}(\Omega;L^{\infty}([0,T];L^{p}(\O)))
	\end{equation}
	satisfies the following conditions:
	\begin{enumerate}
		\item $\text{div}\int_{0}^{u}\sqrt{b(\xi)}d\xi \in L^{2}(\Omega\times[0,T]\times{\mathcal{O}})$.
		
		\smallskip
		
		\item For all $\phi\in C_{b}(\mathbb{R})$, the following chain rule formula holds true:
		\begin{equation}\label{chain_rule_kin}
			\text{div}\int_{0}^{u}\sqrt{b(\xi)}\phi(\xi)d\xi=\phi(u)\cdot\text{div}\int_{0}^{u}\sqrt{b(\xi)}d\xi\ \ \ \text{in}\ \ \ \mathcal{D}'(\overline{\O})\ \text{a.e.}\ (\omega,t).
		\end{equation}	
	\end{enumerate}
	Let  $n_{1}:\Omega\to\mathcal{M}^{+}_{b}([0,T]\times\mathcal{O}\times\mathbb{R})$ be defined as follows: for all $\varphi\in C_{c}([0,T]\times\mathcal{O}\times\mathbb{R})$,
	\begin{equation}\label{kinetic_measure}
		n_{1}(\varphi)=\int_{0}^{T}\int_{\mathcal{O}}\int_{\mathbb{R}}\varphi(t,x,\xi)\Big\lvert\nabla_{x}\int_{0}^{u}\sqrt{b(r)}dr\Big\rvert^{2}d\delta_{u(t,x)=\xi}(\xi) dxdt.
	\end{equation}
	Then $u$ is called a kinetic solution of equation \eqref{limit_para} with initial datum $u_{0}$ provided there exists a kinetic measure $m\geq n_{1}$, $\P$-a.s., such that the pair $\big(h=\mathbf{1}_{u>\xi},m\big)$ satisfies, for all $\varphi\in C^{\infty}_{c}([0,T)\times\mathcal{O}\times\mathbb{R})$, $\P$-a.s.,
	\begin{equation}\label{kinetic_cond}
		\begin{array}{ll}
			&\ds{-\int_{0}^{T}\dInner{h(t),\partial_{t}\varphi(t)}dt=\dInner{h(0),\varphi(0)}+\int_{0}^{T}\dInner{h(t),b(\xi)\Delta_{x}\varphi(t)}dt}\\
			\vs
			&\ds{\quad\quad\quad  +\sum_{k=1}^{\infty}\int_{0}^{T}\int_{\mathcal{O}}\varphi(t,x,u(t,x))\sigma_{k}(x,u(t,x))dxd\beta_{k}(t)}\\
			\vs
			&\ds{\quad\quad\quad +\frac{1}{2}\int_{0}^{T}\int_{\mathcal{O}}\partial_{\xi}\varphi(t,x,u(t,x))\Sigma^{2}(x,u(t,x))dxdt-m(\partial_{\xi}\varphi) },
		\end{array}
	\end{equation}
	where $\dInner{\cdot,\cdot}$ is the duality between the space of distribution over $\mathcal{O}\times\mathbb{R}$ and $C_{c}^{\infty}(\mathcal{O}\times\mathbb{R})$.
\end{Definition}

\begin{Remark}
	We emphasize that a kinetic solution is a class of equivalence in $L^{p}(\Omega\times[0,T];L^{p}(\O))$ so not necessarily a stochastic process in the usual sense. Nevertheless, it will be seen later in Lemma \ref{continuity_Lp} that, in this class of equivalence, there exists a representative which has almost surely continuous trajectories in $L^{p}(\O)$ and, therefore, it can be regarded as a stochastic process.
\end{Remark}

\begin{Remark}\label{remark_weak_solution}
	By $h=\mathbf{1}_{u>\xi}$ we understand a real funtion of four variables. In the stochastic case, the equation \eqref{kinetic_cond} is a weak form of the so-called kinetic formulation of \eqref{limit_para}, which can be formally written as
	\begin{equation}
		\partial_{t}\mathbf{1}_{u>\xi}-b(\xi)\Delta_{x}\mathbf{1}_{u>\xi}=\sum_{i=1}^{\infty}\delta_{u=\xi}\sigma_{i}(x,\xi)d\beta_{i}+\partial_{\xi}\Big(m-\frac{1}{2}\delta_{u=\xi}\Sigma^{2}(x,\xi)\Big)
	\end{equation}
	in the sense of distributions over $[0,T)\times\O\times\mathbb{R}$, where we denote by $\delta_{u=\xi}$ the Dirac measure centered at $u(t,x)$. 	This choice is reasonable since for any $u$ being a weak solution to equation \eqref{limit_para} that belongs to $ L^{p}(\Omega;L^{\infty}([0,T];L^{p}(\O))) \bigcap L^{2}(\Omega;L^{2}([0,T];H^{1}))$, by using \eqref{gen_ito} it can be easily verified that, the pair $\big(h=\mathbf{1}_{u>\xi},n_{1}\big)$ satisfies \eqref{kinetic_cond}, and consequently $u$ is also a kinetic solution. 
\end{Remark}

\begin{Remark}
	We should also mention that, if we let $u$ be a kinetic solution to \eqref{limit_para} and consider $h=\mathbf{1}_{u>\xi}$, then $\nu:=\delta_{u=\xi}=-\partial_{\xi}h$ is a Young measure on $\Omega\times[0,T]\times\O$ (see \cite[Section 2.2]{debussche16}). Therefore \eqref{kinetic_cond} can also be rewritten as follows: for all $\varphi\in C^{\infty}_{c}([0,T)\times\O\times\mathbb{R})$, $\P$-a.s.,
	\begin{equation*}
		\begin{array}{ll}
			&\ds{-\int_{0}^{T}\dInner{h(t),\partial_{t}\varphi(t)}dt=\dInner{h(0),\varphi(0)}+\int_{0}^{T}\dInner{h(t),b(\xi)\Delta_{x}\varphi(t)}dt}\\
			\vs
			&\ds{\quad\quad\quad  +\sum_{k=1}^{\infty}\int_{0}^{T}\int_{\mathcal{O}}\int_{\mathbb{R}}\varphi(t,x,\xi)\sigma_{k}(x,\xi)dxd\nu_{t,x}(\xi)d\beta_{k}(t)}\\
			\vs
			&\ds{\quad\quad\quad +\frac{1}{2}\int_{0}^{T}\int_{\mathcal{O}}\int_{\mathbb{R}}\partial_{\xi}\varphi(t,x,\xi)\Sigma^{2}(x,\xi)d\nu_{t,x}(\xi)dxdt-m(\partial_{\xi}\varphi) }.
		\end{array}
	\end{equation*}
\end{Remark}

\subsection{Generalized kinetic solution}

Now let us introduce the definition of generalized kinetic solution as well as the related definitions. It is a generalization of the definition of kinetic solution in Definition \ref{def_kinetic_solution}, which is suited for establishing the well-posedness of \eqref{limit_para} in $L^{1}(\O)$.

\begin{Definition}[Generalized kinetic measure] \label{def_gen_kinetic_measure}
	 A mapping $m$ from $\Omega$ to $\mathcal{M}^{+}([0,T]\times\mathcal{O}\times\mathbb{R})$, the set of nonnegative Radon measure over $[0,T]\times\mathcal{O}\times\mathbb{R}$, is said to be a generalized kinetic measure provided:
	\begin{enumerate}
		\item For all $\psi\in C_{c}([0,T]\times\mathcal{O}\times\mathbb{R})$, the process
		\begin{equation*}
			\int_{[0,t]\times\mathcal{O}\times\mathbb{R}}\psi(s,x,\xi)dm(s,x,\xi)\in L^{2}(\Omega\times[0,T])
		\end{equation*}
		is predictable.
		
		\smallskip
		
		\item $m$ vanishes for large $\xi$ in the sense that
		\begin{equation*}
			\lim_{n\to+\infty}\frac{1}{2^{l}}\E m(A_{2^{l}})=0,
		\end{equation*}
		where $A_{2^{l}}=[0,T]\times\mathcal{O}\times\big\{\xi\in\mathbb{R}:2^{l}\leq\abs{\xi}\leq 2^{l+1}\big\}$.
	\end{enumerate}
\end{Definition}

\bigskip

\begin{Definition}[Generalized kinetic solution] \label{def_gen_kinetic_solution}
	 Let $u_{0}\in L^{1}(\mathcal{O})$. Assume that $u\in L^{1}(\Omega\times[0,T];L^{1}(\mathcal{O}))$ satisfies 
	\begin{equation}
		\E\esssup_{t\in[0,T]}\norm{u(t)}_{L^{1}(\O)}^{p}\leq C_{T,p},\ \ \ \forall p\geq 1,
	\end{equation}
	and the following conditions:
	\begin{enumerate}
		\item For all $\phi\in C_{c}(\mathbb{R})$, $\phi\geq0$,
		\begin{equation}\label{div_L2}
			\text{div}\int_{0}^{u}\sqrt{b(\xi)}\phi(\xi)d\xi \in L^{2}(\Omega\times[0,T]\times{\mathcal{O}}).
		\end{equation}
		
		\item For all $\phi_{1},\phi_{2}\in C_{c}(\mathbb{R})$, $\phi_{1},\phi_{2}\geq0$, the following chain rule formula holds true in $L^{2}(\Omega\times[0,T]\times\mathcal{O})$
		\begin{equation}\label{chain_rule}
			\text{div}\int_{0}^{u}\sqrt{b(\xi)}\phi_{1}(\xi)\phi_{2}(\xi)d\xi=\phi_{1}(u)\cdot\text{div}\int_{0}^{u}\sqrt{b(\xi)}\phi_{2}(\xi)d\xi.
		\end{equation}
		\end{enumerate}
		Let $\phi\in C^{\infty}_{c}(\mathbb{R})$, $\phi\geq0$, and let $n^{\phi}:\Omega\to\mathcal{M}^{+}([0,T]\times\mathcal{O}\times\mathbb{R})$ be defined as follows: for all $\varphi\in C^{\infty}_{c}([0,T]\times\mathcal{O})$,
		\begin{equation}\label{gen_kinetic_measure}
			n^{\phi}(\varphi)=\int_{0}^{T}\int_{\mathcal{O}}\varphi(t,x)\Big\lvert\nabla_{x}\int_{0}^{u}\sqrt{\phi(r)b(r)}dr\Big\rvert^{2} dxdt.
		\end{equation}
		Then $u$ is called a generalized kinetic solution of equation \eqref{limit_para} with initial datum $u_{0}$ provided there exists a generalized kinetic measure $m$ such that, for all $\varphi\in C^{\infty}_{c}([0,T]\times\mathcal{O})$ and $\phi\in C^{\infty}_{c}(\mathbb{R})$, $\varphi,\phi\geq0$, it holds $m(\varphi\phi)\geq n^{\phi}(\varphi)$, $\P$-a.s., and, in addition, the pair $\big(h=\mathbf{1}_{u>\xi},m\big)$ satisfies equation \eqref{kinetic_cond} for all $\varphi\in C^{\infty}_{c}([0,T)\times\mathcal{O}\times\mathbb{R})$.
\end{Definition}

\begin{Remark}
	Note that the generalized kinetic measure $m$ is generally not a finite measure and the decay assumption from Definition \ref{def_gen_kinetic_measure} is weaker than the one in Definition \ref{def_kinetic_measure}. The chain rule \eqref{chain_rule} in Definition \ref{def_gen_kinetic_solution} is weaker than the corresponding version \eqref{chain_rule_kin} in Definition \ref{def_kinetic_solution}, so the parabolic dissipation $n^{\phi}$ does not necessarily define a measure with respect to the variable $\xi$. 
\end{Remark}

\begin{Remark}
	Let $u\in L^{p}(\Omega;L^{p}([0,T]\times\O))$ for some $p\geq2$. Then $u$ is a generalized kinetic solution to \eqref{limit_para} if and only if $u$ is a kinetic solution.
\end{Remark}

\begin{Remark}
	We also emphasize that a generalized kinetic solution is a class of equivalence in $L^{1}(\Omega\times[0,T];L^{1}(\O))$.
\end{Remark}

\subsection{Statement of the main result}

We are now ready to state our main result of this paper.

\begin{thm}\label{main_thm} Assume Hypotheses \ref{Hypothesis1} and \ref{Hypothesis2}. If $u_{1},u_{2}$ are generalized kinetic solutions to \eqref{limit_para} with initial data $u_{1,0},u_{2,0}\in L^{1}(\O)$, respectively, then it holds that
	\begin{equation}\label{L1_contraction_kinetic}
		\esssup_{t\in[0,T]}\E\norm{(u_{1}(t)-u_{2}(t))^{+}}_{L^{1}(\mathcal{O})}\leq \norm{(u_{1,0}-u_{2,0})^{+}}_{L^{1}(\mathcal{O})}.
	\end{equation}
	Assume in addition that either Hypothesis \ref{Hypothesis3}\ref{nondeg_condition} or \ref{Hypothesis3}\ref{deg_condition} holds. Then for every $u_{0}\in L^{1}(\mathcal{O})$ and every $T>0$, there exists a unique generalized kinetic solution $u\in C([0,T];L^{1}(\mathcal{O}))$, $\P$-a.s. to equation \eqref{limit_para}. Moreover, for every $p\geq1$, there exists $C_{T,p}>0$ such that
	\begin{equation}\label{energy_gen_kinetic_L1}
		\E\ \esssup_{t\in[0,T]}\norm{u(t)}_{L^{1}(\mathcal{O})}^{p}\leq C_{T,p}\Big(1+\norm{u_{0}}_{L^{1}(\mathcal{O})}^{p}\Big).
	\end{equation}
	
\end{thm}

\begin{Remark}\label{main_remark} Under Hypotheses \ref{Hypothesis1} and \ref{Hypothesis3}\ref{nondeg_condition}, adding a drift term $f(u)$ on the right-hand side of equation \eqref{limit_para} will not cause any additional difficulties in our proof for the existence of a generalized kinetic solution, provided that $f\in C^{1}(\mathbb{R})$ is non-increasing, Lipschitz continuous and satisfies 
		\begin{equation*}
			\liminf_{r\to+\infty}f(r)\leq 0.
		\end{equation*}
\end{Remark}

To study the ergodic behavior of \eqref{limit_para}, let us define $P_{t}$ the transition semigroup associated to equation \eqref{limit_para} by 
 \begin{equation*}
 	P_{t}\varphi(z):=\E\varphi(u^{z}(t)),\ \ \ t\geq0,\ \ \ z\in L^{1}(\mathcal{O}),
 \end{equation*}
 for every $\varphi\in B_{b}(L^{1}(\mathcal{O}))$, where $u^{z}$ is the unique generalized kinetic solution with initial datum $z$.
 
 \begin{thm}
 	\label{ergodic_thm}
 	Assume Hypothesis \ref{Hypothesis1} and \ref{Hypothesis3}\ref{nondeg_condition}. If we assume that there exist some $c>0$, $\lambda<\sqrt{2\alpha_{1}b_{0}}$ and $\alpha\in(0,1)$ such that
 	\begin{equation}
 		\label{polynomial_sigma}
 		\norm{\sigma(h)}_{L_{2}(\mathcal{U},H)}\leq \lambda\norm{h}_{H}+c\big(1+\norm{h}_{H}^{\alpha}\big),\ \ \ \ h\in H,
 	\end{equation} 
 	then the transition semigroup $P_{t}$, $t\geq0$, admits at least one invariant measure $\nu$ in $L^{1}(\O)$, with support in $H^{1}$. Assume in addition that the noise in \eqref{limit_para} is additive and the diffusion $b$ is bounded. Then the invariant measure of $P_{t}$ is unique.  
 \end{thm}

\section{Comparison principle and $L^{1}$-contraction}
\label{L1}

In this section we prove a comparison result and an $L^{1}$-contraction property for generalized kinetic solutions to equation \eqref{limit_para}.

 Analogously to \cite[Proof of Proposition 10]{debussche14}, we can prove the existence of left- and right-continuous representatives for generalized kinetic solutions (also see \cite{gess18}). 

\begin{lem}
	Let $u$ be a generalized kinetic solution of equation \eqref{limit_para}. Then $h=\mathbf{1}_{u>\xi}$ admits representatives $h^{-}$ and $h^{+}$ which are almost surely left and right continuous, respectively, at all points  $t_{\ast}\in[0,T]$ in the sense of distributions over $\mathcal{O}\times\mathbb{R}$. More precisely, for all $t_{\ast}\in [0,T]$ there exist kinetic functions $h_{\ast}^{\pm}$ on $\Omega\times\mathcal{O}\times\mathbb{R}$ such that setting $h^{\pm}(t_{\ast})=h_{\ast}^{\pm}$ yields $h^{\pm}=h$ almost everywhere and 
	\begin{equation*}
		\lim_{\epsilon\to0}\dInner{h^{\pm}(t_{\ast}\pm\epsilon),\psi}=\dInner{h^{\pm}(t_{\ast}),\psi},\ \ \ \forall\psi\in C^{2}_{c}(\mathcal{O}\times\mathbb{R}),\ \ \P\text{-a.s.},
	\end{equation*}  
	where the zero set does not depend on $\psi$ nor $t_{\ast}$. Moreover, there is a countable set $S\subset[0,T]$ such that $\P$-a.s. for all $t_{\ast}\in[0,T]\setminus S$ we have $h^{+}(t_{\ast})=h^{-}(t_{\ast})$.
\end{lem}

\begin{Remark}
	For all $t\in[0,T]$ and $\psi\in C^{2}_{c}(\mathcal{O}\times\mathbb{R})$, by taking $\varphi(s,x,\xi)=\eta(s)\psi(x,\xi)$ in \eqref{kinetic_cond}, where 
	\begin{equation*}
		\eta(s)=
		\le\{\begin{array}{l}
			\ds{1,\ \ \ \ s\leq t,}\\
			\vs
			\ds{1-\frac{s-t}{\epsilon},\ \ \ \ t\leq s\leq t+\epsilon,}\\
			\vs
			\ds{0,\ \ \ \ t+\epsilon\leq s},
		\end{array}\r.
	\end{equation*}
	we obtain at the limit when $\epsilon\to0$ that $\mathbb{P}$-a.s.,
	\begin{equation}\label{kinetic_test}
		\begin{array}{ll}
			&\ds{-\dInner{h^{+}(t),\psi }+\dInner{h(0),\psi}+\int_{0}^{t}\Inner{h(s),b(\xi)\Delta_{x}\psi}ds+\sum_{i=1}^{\infty}\int_{0}^{t}\int_{\mathcal{O}}\int_{\mathbb{R}}\psi(x,\xi)\sigma_{i}(x,\xi)d\nu_{s,x}(\xi)dxd\beta_{i}(s) }\\
			\vs
			&\ds{\quad\quad\quad\quad\quad\quad\quad\quad\quad=-\frac{1}{2}\int_{0}^{t}\int_{\mathcal{O}}\int_{\mathbb{R}}\partial_{\xi}\psi(x,\xi)\Sigma^{2}(x,\xi)d\nu_{s,x}(\xi)dxds+m(\partial_{\xi}\psi)([0,t]) }.
		\end{array}
	\end{equation}
\end{Remark}

Let $\rho$ and $\psi$ be the standard mollifier on $\mathbb{R}^{d}$ and $\mathbb{R}$, and let $K\in C^{\infty}(\mathbb{R})$ be such that $0\leq K(\eta)\leq1$, $K\equiv1$ if $\abs{\eta}\leq1$, $K\equiv0$ if $\abs{\eta}\geq2$, and $\abs{K'(\eta)}\leq1$. Define 
\begin{equation}\label{mollifier}
	\rho_{\epsilon}(x)=\frac{1}{\epsilon^{d}}\rho\Big(\frac{x}{\epsilon}\Big),\ \ \ \ \psi_{\delta}(\xi)=\frac{1}{\delta}\psi\Big(\frac{\xi}{\delta}\Big),\ \ \ \ K_{l}(\eta)=K\Big(\frac{\eta}{2^{l}}\Big),
\end{equation}  
for $\epsilon,\delta>0$ and $l\in\mathbb{N}$. Then $\abs{K_{l}'(\eta)}\leq 2^{-l}\mathbf{1}_{2^{l}\leq \abs{\eta}\leq 2^{l+1}}$. Throughout this section, we use the following convention: When no integral bounds are specified, we integrate with respect to $(x,y,\xi,\zeta,\eta)
\in\O^{2}\times\mathbb{R}^{3}$. In addition, we only specify the kinetic and Young measures, but omit the Lebesgue measure (also with respect to the time variable).

The next lemma relies on the doubling of variables technique (see \cite{debussche14}, \cite{debussche16}, \cite{fridnew}). Since we are now working on a bounded domain $\mathcal{O}\subset\mathbb{R}^{d}$, it is necessary to establish a localized version which is suitable in the $L^{1}$-setting. The proof is given in Section \ref{sec_appendix}.

\begin{lem}[Doubling of variables]\label{contraction}
	Let $u_{1},u_{2}$ be generalized kinetic solutions of equation \eqref{limit_para} with initial data $u_{1,0},u_{2,0}\in L^{1}(\O)$. Denote $h_{1}=\mathbf{1}_{u_{1}>\xi}$, $h_{2}=\mathbf{1}_{u_{2}>\xi}$ with the corresponding Young measures $\nu^{1}=\delta_{u_{1}}$, $\nu^{2}=\delta_{u_{2}}$ and generalized kinetic measures $m_{1},m_{2}$, respectively. Then for all $t\in[0,T]$ and nonnegative $\phi\in C^{\infty}_{c}(\mathcal{O}^{2})$ we have
	\begin{equation}
			\E\int\Big(h_{1}^{\pm}(t)\overline{h}_{2}^{\pm}(t)-h_{1,0}\overline{h}_{2,0}\Big)\phi(x,y)K_{l}(\eta)\psi_{\delta}(\eta-\xi)\psi_{\delta}(\eta-\zeta)\leq \sum_{i=1}^{6}J_{i}(\delta,l),\ \ \ \delta>0,\ \ l\in\mathbb{N}.
	\end{equation}
	Here we denoted by $\overline{h}$ the conjugate function $\overline{h}=1-h$, and   
	\begin{equation*}
		J_{1}(\delta,l)=
		\E\int_{0}^{t}\int h_{1}(s)\overline{h}_{2}(s)\Big(b(\xi)\nabla_{x}^{2}+2\sqrt{b(\xi)b(\zeta)}\nabla_{x,y}^{2}+b(\zeta)\nabla_{y}^{2}\Big)\phi(x,y)K_{l}(\eta)\psi_{\delta}(\eta-\xi)\psi_{\delta}(\eta-\zeta),
	\end{equation*}

	\begin{equation*}
		J_{2}(\delta,l)=\frac{1}{2}\E\sum_{k=1}^{\infty}\int_{0}^{t}\int \phi(x,y)K_{l}(\eta)\psi_{\delta}(\eta-\xi)\psi_{\delta}(\eta-\zeta)\big\lvert \sigma_{k}(x,\xi)-\sigma_{k}(y,\zeta)\big\rvert^{2}d\nu_{x,s}^{1}(\xi)d\nu_{y,s}^{2}(\zeta),
	\end{equation*}
	
	\begin{equation*}
		J_{3}(\delta,l)=-\frac{1}{2}\E\int_{0}^{t}\int h_{1}(s)\phi(x,y)K'_{l}(\eta)\psi_{\delta}(\eta-\xi)\psi_{\delta}(\eta-\zeta)\Sigma^{2}(y,\zeta)d\nu_{y,s}^{2}(\zeta),
	\end{equation*}
	
	\begin{equation*}
		J_{4}(\delta,l)=\frac{1}{2}\E\int_{0}^{t}\int\overline{h}_{2}(s)\phi(x,y)K'_{l}(\eta)\psi_{\delta}(\eta-\xi)\psi_{\delta}(\eta-\zeta)\Sigma^{2}(x,\xi)d\nu_{x,s}^{1}(\xi),
	\end{equation*}
	
	\begin{equation*}
		J_{5}(\delta,l)=-\E\int_{0}^{t}\int\overline{h}_{2}^{+}(s)\phi(x,y)K'_{l}(\eta)\psi_{\delta}(\eta-\xi)\psi_{\delta}(\eta-\zeta) dm_{1}(x,s,\xi),
	\end{equation*}
	and
	\begin{equation*}
		J_{6}(\delta,l)=\E\int_{0}^{t}\int h_{1}^{-}(s)\phi(x,y)K'_{l}(\eta)\psi_{\delta}(\eta-\xi)\psi_{\delta}(\eta-\zeta) dm_{2}(y,s,\zeta),
	\end{equation*}
	where $\nabla_{x,y}^{2}\phi(x,y):=\sum_{i=1}^{d}\partial_{x_{i}y_{i}}^{2}\phi(x,y)$.
\end{lem}

\bigskip

\begin{prop}\label{uniqueness_prelim}
	Assume Hypotheses \ref{Hypothesis1} and \ref{Hypothesis2}. Let $u_{1},u_{2}$ be kinetic solutions of equation \eqref{limit_para} with initial datum $u_{1,0}$, $u_{2,0}$, respectively. Denote $h_{1}=\mathbf{1}_{u_{1}>\xi}$, $h_{2}=\mathbf{1}_{u_{2}>\xi}$, then for every $t\in[0,T]$ and nonnegative $\varphi\in C^{\infty}_{c}(\mathcal{O})$
	\begin{equation}\label{goal1}
		\begin{array}{ll}
			\ds{\E\int_{\mathcal{O}}\int_{\mathbb{R}}h_{1}^{\pm}(t)\overline{h}_{2}^{\pm}(t)\varphi(x)d\xi dx }
			&\ds{\leq \int_{\mathcal{O}}\int_{\mathbb{R}}h_{1,0}\overline{h}_{2,0}\varphi(x)d\xi dx }\\
			\vs
			&\ds{\quad +\E\int_{0}^{t}\int_{\mathcal{O}}\int_{\mathbb{R}}h_{1}(s)\overline{h}_{2}(s)b(\xi)\Delta_{x}\varphi(x)d\xi dxds }.
		\end{array}
	\end{equation}
\end{prop}

\begin{proof}
	Let $0\leq \varphi\in C^{\infty}_{c}(\mathcal{O})$. By applying Lemma \ref{contraction} to $\phi(x,y)=\varphi(x)\rho_{\epsilon}(x-y)$, where $\rho_{\epsilon}$ is defined in \eqref{mollifier}, we have for every $l\in\mathbb{N}$,
	\begin{equation}\label{base}
		\begin{array}{ll}
			&\ds{\E\int_{\mathcal{O}}\int_{\mathbb{R}}h^{\pm}_{1}(t,x,\xi)\overline{h}^{\pm}_{2}(t,x,\xi)\varphi(x)K_{l}(\xi)d\xi dx }\\
			\vs
			&\ds{= \E\int h_{1}^{\pm}(t,x,\xi)\overline{h}^{\pm}_{2}(t,y,\zeta)\varphi(x)\rho_{\epsilon}(x-y)K_{l}(\eta)\psi_{\delta}(\eta-\xi)\psi_{\delta}(\eta-\zeta) +\Lambda_{t}(\epsilon,\delta,l) }\\
			\vs
			&\ds{\leq \int h_{1,0}(x,\xi)\overline{h}_{2,0}(y,\zeta)\varphi(x)\rho_{\epsilon}(x-y)K_{l}(\eta)\psi_{\delta}(\eta-\xi)\psi_{\delta}(\eta-\zeta)+\sum_{i=1}^{6}J_{i}(\epsilon,\delta,l)+\Lambda_{t}(\epsilon,\delta,l)  }
		\end{array}
	\end{equation}
	where $J_{i}(\epsilon,\delta,l)$, $1\leq i\leq 6$, are introduced in Proposition \ref{contraction}, with $\phi(x,y)$ replaced by $\varphi(x)\rho_{\epsilon}(x-y)$,   and 
	\begin{equation*}
		\lim_{\epsilon,\delta\to0}\Lambda_{t}(\epsilon,\delta,l)=0,\ \ \  l\in\mathbb{N},\ \ \ t\geq 0.
	\end{equation*}
	Now we aim to estimate $J_{i}(\epsilon,\delta,l)$, for $1\leq i\leq 6$. First, note that $\nabla_{x}\rho_{\epsilon}(x-y)=-\nabla_{y}\rho_{\epsilon}(x-y)$, we can write
	\begin{equation*}
		\begin{array}{ll}
			\ds{J_{1}(\epsilon,\delta,l)}
			&\ds{=\E\int_{0}^{t}\int h_{1}\overline{h}_{2}b(\xi)\Delta_{x}\varphi(x)\rho_{\epsilon}(x-y)K_{l}(\eta)\psi_{\delta}(\eta-\xi)\psi_{\delta}(\eta-\zeta) }\\
			\vs
			&\ds{\quad+ \E\int_{0}^{t}\int h_{1}\overline{h}_{2}\big(\sqrt{b(\xi)}-\sqrt{b(\zeta)}\big)^{2}\varphi(x)\Delta_{x}\rho_{\epsilon}(x-y)K_{l}(\eta)\psi_{\delta}(\eta-\xi)\psi_{\delta}(\eta-\zeta) }\\
			\vs
			&\ds{\quad +2\E\int_{0}^{t}\int h_{1}\overline{h}_{2}\sqrt{b(\xi)}\big(\sqrt{b(\xi)}-\sqrt{b(\zeta)}\big)\nabla_{x}\varphi(x)\cdot\nabla_{x}\rho_{\epsilon}(x-y)K_{l}(\eta)\psi_{\delta}(\eta-\xi)\psi_{\delta}(\eta-\zeta) }.
		\end{array}
	\end{equation*}
	Then, by the locally $\gamma$-H{\"o}lder continuity of $\sqrt{b}$ on $\mathbb{R}$, we have 
	\begin{equation}\label{J1}
		\begin{array}{ll}
			\ds{J_{1}(\epsilon,\delta,l)}
			&\ds{\leq \E\int_{0}^{t}\int h_{1}\overline{h}_{2}b(\xi)\Delta_{x}\varphi(x)\phi_{\epsilon}(x-y)K_{l}(\eta)\psi_{\delta}(\eta-\xi)\psi_{\delta}(\eta-\zeta) }\\
			\vs
			&\ds{\quad +t\cdot C_{l}\Big(\norm{\varphi}_{\infty}\epsilon^{-2}\delta^{2\gamma}+\norm{\nabla\varphi}_{\infty}\epsilon^{-1}\delta^{\gamma}\Big) }.
		\end{array}
	\end{equation}
	Due to \eqref{sgfine1}, we have
	\begin{equation}\label{J2}
		\begin{array}{ll}
		\ds{J_{2}(\epsilon,\delta,l)}
		&\ds{\leq C\E\int_{0}^{t}\int \varphi(x)\rho_{\epsilon}(x-y)\abs{x-y}^{2}\psi_{\delta}(\eta-\zeta)\psi_{\delta}(\eta-\zeta)d\nu_{x,s}^{1}(\xi)d\nu_{y,s}^{2}(\zeta) }\\
		\vs
		&\ds{\quad +C\E\int_{0}^{t}\int \varphi(x)\rho_{\epsilon}(x-y)\psi_{\delta}(\eta-\zeta)\psi_{\delta}(\eta-\zeta)\abs{\xi-\zeta}^{2}d\nu_{x,s}^{1}(\xi)d\nu_{y,s}^{2}(\zeta) }\\
		\vs
		&\ds{\leq ct\norm{\varphi}_{\infty}\delta^{-1}\epsilon^{2}+ct\norm{\varphi}_{\infty}\delta }.
		\end{array}
	\end{equation}
	Moreover, thanks to \eqref{sgfine2} we have
	\begin{equation*}
		\begin{array}{ll}
			\ds{J_{3}(\epsilon,\delta,l)}
			&\ds{\leq \frac{C\norm{\varphi}_{\infty}}{2^{l}}\E\int_{0}^{t}\int_{\mathcal{O}}\int_{\mathbb{R}}\mathbf{1}_{2^{l}-\delta\leq\abs{\zeta}\leq 2^{l+1}+\delta}\Sigma^{2}(y,\zeta)d\nu_{y,s}^{2}(\zeta)dyds}\\
			\vs
			&\ds{\leq \frac{C\norm{\varphi}_{\infty}}{2^{l}}\E\int_{0}^{t}\int_{\mathcal{O}}\int_{\mathbb{R}}\mathbf{1}_{2^{l}-\delta\leq\abs{\zeta}\leq 2^{l+1}+\delta}d\nu_{y,s}^{2}(\zeta)dyds }\\
			\vs
			&\ds{\quad \quad +\frac{C\norm{\varphi}_{\infty}}{2^{l}}\E\int_{0}^{t}\int_{\mathcal{O}}\int_{\mathbb{R}}\mathbf{1}_{2^{l}-\delta\leq\abs{\zeta}\leq 2^{l+1}+\delta}\abs{\zeta}^{2}d\nu_{y,s}^{2}(\zeta)dyds},
		\end{array}
	\end{equation*}
	and this implies that 
	\begin{equation}\label{J3}
		\limsup_{\delta\to0}\ J_{3}(\epsilon,\delta,l)\leq \frac{c\norm{\varphi}_{\infty}}{2^{l}}t+c\norm{\varphi}_{\infty}\E\int_{0}^{t}\int_{\O}\mathbf{1}_{2^{l}\leq \abs{u_{2}(s,y)}}\abs{u_{2}(s,y)}dyds.
	\end{equation}
	Similarly, we obtain that
	\begin{equation}\label{J4}
		\limsup_{\delta\to0}\ J_{4}(\epsilon,\delta,l)\leq \frac{c\norm{\varphi}_{\infty}}{2^{l}}t+c\norm{\varphi}_{\infty}\E\int_{0}^{t}\int_{\O}\mathbf{1}_{2^{l}\leq \abs{u_{1}(s,x)}}\abs{u_{1}(s,x)}dyds.
	\end{equation}
	Furthermore, it is easy to see that
	\begin{equation}\label{J5}
		\begin{array}{ll}
			\ds{J_{5}(\epsilon,\delta,l)\leq \frac{1}{2^{l}}\E\int_{0}^{t}\int_{\mathcal{O}}\int_{\mathbb{R}}\mathbf{1}_{2^{l}-\delta\leq\abs{\xi}\leq 2^{l+1}+\delta}dm_{1}(x,s,\xi) }
		\end{array}
	\end{equation}
	and
	\begin{equation}\label{J6}
		\begin{array}{ll}
			\ds{J_{6}(\epsilon,\delta,l)\leq \frac{1}{2^{l}}\E\int_{0}^{t}\int_{\mathcal{O}}\int_{\mathbb{R}}\mathbf{1}_{2^{l}-\delta\leq\abs{\zeta}\leq 2^{l+1}+\delta}dm_{2}(y,s,\zeta) }.
		\end{array}
	\end{equation}
	Now taking $\epsilon=\delta^{\alpha}$ with $\alpha\in(1/2,\gamma)$ and taking the limit as $\delta\to0$ in \eqref{base}, thanks to \eqref{J1} and \eqref{J2}, we have for every $t\geq0$ and every $l\in\mathbb{N}$
	\begin{equation*}
		\begin{array}{ll}
			&\ds{\E\int_{\mathcal{O}}\int_{\mathbb{R}}h^{\pm}_{1}(t,x,\xi)\overline{h}^{\pm}_{2}(t,x,\xi)\varphi(x)K_{l}(\xi)d\xi dx}\\
			\vs
			&\ds{\leq \int_{\mathcal{O}}\int_{\mathbb{R}}h_{1,0}(x,\eta)\overline{h}_{2,0}(x,\eta)\varphi(x)K_{l}(\eta)d\eta dx +\E\int_{0}^{t}\int_{\mathcal{O}}\int_{\mathbb{R}}h_{1}\overline{h}_{2}b(\eta)\Delta_{x}\varphi(x)K_{l}(\eta)d\eta dxds }\\
			\vs
			&\ds{\quad+\limsup_{\delta\to0}\Big(J_{3}(\delta^{\alpha},\delta,l)+J_{4}(\delta^{\alpha},\delta,l)+J_{5}(\delta^{\alpha},\delta,l)+J_{6}(\delta^{\alpha},\delta,l)\Big) }.
		\end{array}
	\end{equation*}
	Finally, by taking the limit as $l\to+\infty$, due to \eqref{J3}-\eqref{J6} we obtain \eqref{goal1}.
	
\end{proof}

Concerning the behavior near the boundary $\partial\O$, we recall the system $\mathfrak{B}$ of balls introduced in Section \ref{sec_boundary}, and prove the following result.

\begin{prop}\label{uniqueness_further}
	Assume Hypothesis \ref{Hypothesis1} and \ref{Hypothesis2}. Let ball $B\in\mathfrak{B}$ and let $u_{1},u_{2}$ be generalized kinetic solutions to \eqref{limit_para} with initial data $u_{1,0},u_{2,0}$, respectively. Denote $h_{1}=\mathbf{1}_{u_{1}>\xi}$, $h_{2}=\mathbf{1}_{u_{2}>\xi}$, then for every $t\in[0,T]$ and nonnegative $\varphi\in C^{\infty}_{c}(B)$,
	\begin{equation}
		\begin{array}{ll}
			\ds{\E\int_{\mathcal{O}}\int_{\mathbb{R}}h_{1}^{\pm}(t)\overline{h}_{2}^{\pm}(t)\varphi(x)d\xi dx }
			&\ds{\leq \int_{\mathcal{O}}\int_{\mathbb{R}}h_{1,0}\overline{h}_{2,0}\varphi(x)d\xi dx }\\
			\vs
			&\ds{\quad +\E\int_{0}^{t}\int_{\mathcal{O}}\int_{\mathbb{R}}h_{1}(s)\overline{h}_{2}(s)b(\xi)\Delta_{x}\varphi(x)d\xi dxds }.
		\end{array}
	\end{equation}
\end{prop}

\begin{proof}
	Let $0\leq \phi\in C^{\infty}_{c}(B)$ and $\{\zeta_{\delta}\}_{\delta\in(0,1)}$ be an $\mathcal{O}$-boundary layer sequence such that $\Delta\zeta_{\delta}\leq 0$ (see Section \ref{sec_boundary}). Now, by approximation we may take $\varphi(x)=\phi(x)\zeta_{\delta}(x)$ as test functions in  \eqref{goal1}, and then by proceeding as in the proof of Proposition \ref{uniqueness_prelim}, we have for every $l\in\mathbb{N}$ and every $\delta\in(0,1)$ 
	\begin{equation}\label{base_boundary}
		\begin{array}{ll}
			&\ds{\E\int_{\mathcal{O}}\int_{\mathbb{R}}\Big(h_{1}^{\pm}(t,x,\xi)\overline{h}_{2}^{\pm}(t,x,\xi)-h_{1}(0,x,\xi)\overline{h}_{2}(0,x,\xi)\Big)\phi(x)\zeta_{\delta}(x)K_{l}(\xi)d\xi dx}\\
			\vs
			&\ds{\leq \E\int_{0}^{t}\int_{\mathcal{O}}\int_{\mathbb{R}}h_{1}(s,x,\xi)\overline{h}_{2}(s,x,\xi)b(\xi)\Delta_{x}\phi(x)\zeta_{\delta}(x)K_{l}(\xi)d\xi dxds}\\
			\vs
			&\ds{\quad+2\E\int_{0}^{t}\int_{\mathcal{O}}\int_{\mathbb{R}}h_{1}(s,x,\xi)\overline{h}_{2}(s,x,\xi)b(\xi)\nabla\phi(x)\cdot\nabla\zeta_{\delta}(x)K_{l}(\xi)d\xi dxds }\\
			\vs
			&\ds{\quad+\E\int_{0}^{t}\int_{\mathcal{O}}\int_{\mathbb{R}}h_{1}(s,x,\xi)\overline{h}_{2}(s,x,\xi)b(\xi)\phi(x)\Delta_{x}\zeta_{\delta}(x)K_{l}(\xi)d\xi dxds+\Gamma_{t}(l,\delta) },
		\end{array}
	\end{equation}
	where $\lim_{l\to\infty}\sup_{\delta\in(0,1)}\Gamma_{t}(\delta,l)=0$, $ t\geq0$. For every $l\in\mathbb{N}$, if we define 
	\begin{equation*}
		\Psi_{l}(s,x):=\int_{\mathbb{R}}h_{1}(s,x,\xi)\overline{h_{2}}(s,x,\xi)b(\xi)K_{l}(\xi)d\xi=\int_{u_{2}(s,x)}^{u_{1}(s,x)}b(\xi)K_{l}(\xi)d\xi,\ \ \ (s,x)\in[0,T)\times\overline{\mathcal{O}},
	\end{equation*}
	then it is clear that $\Psi_{l}\in L^{2}(\Omega;L^{2}([0,T];H))$. Since $\sqrt{b}K_{l}\in C_{c}(\mathbb{R})$, by \eqref{chain_rule} we have 
	\begin{equation*}
		\text{div}(\Psi_{l})=\text{div}\int_{u_{2}}^{u_{1}}b(\xi)K_{l}(\xi)d\xi\in L^{2}(\Omega\times[0,T]\times\O).
	\end{equation*}
	Moreover, due to the boundary condition $u_{1}(s,\cdot)|_{\partial\O}=u_{2}(s,\cdot)|_{\partial\O}=0$, we know $\Psi_{l}(s,x)=0$ for $(s,x)\in[0,T]\times\partial\O$. Hence, we have $\Psi_{l}\in L^{2}(\Omega;L^{2}([0,T];H^{1}))$, and this implies that $\P$-a.s.,
	\begin{equation*}
		\Psi_{l}(s,\cdot)\nabla\phi(\cdot)\in (H^{1})^{d},\ \ \ \ s\in[0,T].
	\end{equation*}
	Then it follows from \eqref{boundary_laryer_property} that, for every $s\in[0,T]$, as $\delta\to0$
	\begin{equation*}
		\begin{array}{ll}
			&\ds{\int_{\mathcal{O}}\int_{\mathbb{R}}h_{1}(s,x,\xi)\overline{h}_{2}(s,x,\xi)b(\xi)K_{l}(\xi)\nabla\phi(x)\cdot\nabla\zeta_{\delta}(x)d\xi dx}\\
			\vs
			&\ds{=\int_{\O}\Psi_{l}(s,x)\nabla\phi(x)\cdot\nabla\zeta_{\delta}(x)dx=-\int_{\O}\text{div}\big(\Psi_{l}(s,x)\nabla\phi(x)\big)\zeta_{\delta}(x)dx }\\
			\vs
			&\ds{\to-\int_{\O}\text{div}\big(\Psi_{l}(s,x)\nabla\phi(x)\big)dx=-\int_{\partial\O}\Psi_{l}(s,x)\nabla\phi(x)\cdot \nu(x) dS(x)=0,\ \ \ \P\text{-a.s.} }.
		\end{array}
	\end{equation*}
	Note that, due to the fact that $0\leq \zeta_{\delta}\leq 1$ and $\phi\in C^{\infty}_{c}(B)$, it follows that for each $l\in\mathbb{N}$
	\begin{equation*}
		\begin{array}{ll}
		&\ds{\sup_{\delta\in(0,1)}\E\int_{0}^{t}\Big\lvert\int_{\O}\Psi_{l}(s,x)\nabla\phi(x)\cdot\nabla\zeta_{\delta}(x)dx\Big\rvert^{2}ds }\\
		\vs
		&\ds{\leq c\ \sup_{\delta\in(0,1)}\E \int_{0}^{t}\int_{\mathcal{O}}\Big\lvert\text{div}\big(\Psi_{l}(s,x)\nabla\phi(x)\big)\zeta_{\delta}(x)\Big\rvert ^{2}dxds\leq C(l,\phi)},
		\end{array}
	\end{equation*}
	and then, by the Vitali convergence theorem we have for each $l\in\mathbb{N}$
	\begin{equation*}
		\lim_{\delta\to0}\E\int_{0}^{t}\Bigg\lvert\int_{\mathcal{O}}\int_{\mathbb{R}}h_{1}(s,x,\xi)\overline{h}_{2}(s,x,\xi)b(\xi)K_{l}(\xi)\nabla\phi(x)\cdot\nabla\zeta_{\delta}(x)d\xi dx\Bigg\rvert ds=0.
	\end{equation*}
	Finally, thanks to the property $\Delta_{x}\zeta_{\delta}\leq 0$, we have $\P$-a.s.,
	\begin{equation*}
		\int_{0}^{t}\int_{\mathcal{O}}\int_{\mathbb{R}}h_{1}(s,x,\xi)\overline{h}_{2}(s,x,\xi)b(\xi)K_{l}(\xi)\phi(x)\Delta_{x}\zeta_{\delta}(x)d\xi dxds\leq 0, \ \ \ \delta\in(0,1),\ \ \ l\in\mathbb{N}.
	\end{equation*}
	Therefore, by taking the limit as $\delta\to0$ and then as $l\to+\infty$ in \eqref{base_boundary}, we complete the proof.
	
\end{proof}

Now we are ready to prove the following comparison principle for generalized kinetic solution, which leads to the proof of uniqueness as well as continuous dependence on the initial condition.

\begin{prop}[Comparison principle]
	\label{L1_contraction}
	Assume Hypotheses \ref{Hypothesis1} and \ref{Hypothesis2}. Let $u$ be a generalized kinetic solution of equation \eqref{limit_para}. Then there exist $u^{+}$ and $u^{-}$, respectatives of $u$, such that, for all $t\in[0,T]$, $h^{\pm}(t,x,\xi)=\mathbf{1}_{u^{\pm}(t,x)>\xi}$ for a.e. $(\omega,x,\xi)$. Moreover, if $u_{1},u_{2}$ are generalized kinetic solutions of equation \eqref{limit_para} with initial data $u_{1,0}$, $u_{2,0}\in L^{1}(\O)$, respectively, then it holds that
	\begin{equation}\label{L1_contraction_eq}
		\sup_{t\in[0,T]}\E\norm{(u_{1}^{\pm}(t)-u_{2}^{\pm}(t))^{+}}_{L^{1}(\mathcal{O})}\leq \norm{(u_{1,0}-u_{2,0})^{+}}_{L^{1}(\mathcal{O})}.
	\end{equation}
\end{prop}

\begin{proof}
	
	Let $h_{1}=\mathbf{1}_{u_{1}>\xi}$ and $h_{2}=\mathbf{1}_{u_{2}>\xi}$. We choose a finite covering $\{B_{i}\}_{i=0}^{N}$ of $\overline{\mathcal{O}}$ such that $B_{i}\in\mathfrak{B}$, $1\leq i\leq N$, and $B_{0}\subset\subset\mathcal{O}$, and a partition of unity $\{p_{i}\}_{i=0}^{N}$ subordinated to $\{B_{i}\}_{i=0}^{N}$. Then we apply Proposition \ref{uniqueness_prelim} to $p_{0}\in C^{\infty}_{c}(\O)$, and apply Corollary \ref{uniqueness_further} to each $p_{i}\in C^{\infty}_{c}(B_{i})$, for $1\leq i\leq N$, and then by taking the sum of all corresponding inequalities and using the property that 
	\begin{equation*}
		\sum_{i=0}^{N}p_{i}(x)=1, \ \ \ x\in\O.
	\end{equation*}
	 we have for every $t\in[0,T]$,
	\begin{equation}
		\E\int_{\mathcal{O}}\int_{\mathbb{R}}h_{1}^{\pm}(t,x,\xi)\overline{h}_{2}^{\pm}(t,x,\xi)d\xi dx\leq \int_{\mathcal{O}}\int_{\mathbb{R}}h_{1,0}(x,\xi)\overline{h}_{2,0}(x,\xi)d\xi dx.
	\end{equation}
	The conclusion of \eqref{L1_contraction_eq} now follows from this proceeds along the lines in \cite[Proof of Theorem 3.3]{debussche16}.

\end{proof}

\section{Well-posedness of equation \eqref{limit_para} in $L^{p}(\O)$, $p\geq2$ }
\label{sec_wellposedness_Lp}

In this section we shall establish the well-posedness for equation \eqref{limit_para} in $L^{p}(\O)$, for $p\geq2$. First, by following the approach in \cite[Proof of Corollary 3.4]{hofmanova2013} and \cite[Proof of Corollary 16]{debussche14}, we can prove the continuity of trajectories of kinetic solutions in $L^{p}(\O)$.
\begin{lem}[Continuity in time]\label{continuity_Lp}
	Assume Hypotheses \ref{Hypothesis1} and \ref{Hypothesis2}. Let $u$ be a kinetic solution to equation \eqref{limit_para} with initial datum $u_{0}\in L^{p}(\O)$ for some $p\geq2$. Then there exists a representative of $u$ which has almost surely continuous trajectories in $L^{p}(\O)$.
\end{lem}
Then, combining Lemma \ref{continuity_Lp} together with Proposition \ref{L1_contraction}, we obtain the following result.
\begin{cor}\label{uniqueness_kinetic}
	Assume Hypotheses \ref{Hypothesis1} and \ref{Hypothesis2}. If $u_{1},u_{2}$ are kinetic solutions of equation \eqref{limit_para} with initial datum $u_{1,0},u_{2,0}\in L^{p}(\O)$ for some $p\geq2$, then we have for all $t\in[0,T]$
	\begin{equation*}
		\E\norm{u_{1}(t)-u_{2}(t)}_{L^{1}(\O)}\leq \norm{u_{1,0}-u_{2,0}}_{L^{1}(\O)}.
	\end{equation*}
	This, in particular, implies the uniqueness of kinetic solution to \eqref{limit_para}.
\end{cor}

We remark that since every weak solution to \eqref{limit_para} is also a kinetic solution, the results above also remain true for weak solutions to \eqref{limit_para}. 

Now, based on different types of conditions on the diffusion $b$, the well-posedness for \eqref{limit_para} is established in two subsections. More specifically, in Section \ref{sec_nondegenerate} we prove the existence of a weak solution to \eqref{limit_para} when $b$ is non-degenerate; and in Section \ref{sec_degenerate} we prove the existence of a kinetic solution to \eqref{limit_para} in the case when $b$ is degenerate.  

\subsection{The non-degenerate case}
\label{sec_nondegenerate}

In this subsection we will assume Hypothesis \ref{Hypothesis3}\ref{nondeg_condition} holds. Given any initial data $u_{0}\in L^{2\theta}(\O)$, the well-posdeness for \eqref{limit_para} is given in the following theorem.

\begin{thm}\label{weak_nondeg} Assume Hypotheses \ref{Hypothesis1} and \ref{Hypothesis3}\ref{nondeg_condition}. Then for every $u_{0}\in L^{2\theta}(\O)$ and every $T>0$, there exists a unique weak solution 
	\begin{equation*}
		u\in L^{2\theta}(\Omega;L^{\infty}([0,T];L^{2\theta}(\O)))\bigcap L^{2}(\Omega;L^{2}([0,T];H^{1})),
	\end{equation*}
	of equation \eqref{limit_para}, which has $\P$-almost surely continuous trajectories in $L^{2\theta}(\O)$. Moreover, for all $p\in[1,2\theta]$ and $q\in[1,+\infty)$ there is a constant $C_{T,p,q}>0$ such that
	\begin{equation}\label{energy_Lp}
		\begin{array}{ll}
		&\ds{\E\sup_{t\in[0,T]}\norm{u(t)}_{L^{p}(\O)}^{pq}+\E\Big(\int_{0}^{T}\int_{\O}\big(1+\abs{u(s,x)}^{2}\big)^{\frac{p}{2}-1}b(u(s,x))\abs{\nabla u(s,x)}^{2}dxds\Big)^{q}}\\
		\vs
		&\ds{\quad\quad\leq C_{T,p,q}\Big(1+\norm{u_{0}}_{L^{p}(\O)}^{pq}\Big)}.
		\end{array}
	\end{equation}
\end{thm}

\begin{Remark}\label{L1_almost_surely_weak_solution}
	Under Hypotheses \ref{Hypothesis1} and \ref{Hypothesis3}\ref{nondeg_condition}, if we assume that the noise in \eqref{limit_para} is additive, then following our approach in Section \ref{L1} and Section \ref{sec_appendix} one can show that the $L^{1}$-contraction property holds for weak solutions in the almost sure sense. Namely, if $u_{1},u_{2}$ are weak solutions to equation \eqref{limit_para} with initial data $u_{1,0},u_{2,0}\in L^{2\theta}(\O)$, respectively, then it holds that 
	\begin{equation*}
		\sup_{t\in[0,T]}\norm{u_{1}(t)-u_{2}(t)}_{L^{1}(\O)}\leq \norm{u_{1,0}-u_{2,0}}_{L^{1}(\O)},\ \ \ \P\text{-a.s.}.
	\end{equation*}
\end{Remark}

\subsubsection{The approximation problem}

In order to construct an approximation problem for \eqref{limit_para}, we define
\begin{equation*}
	\tilde{\mathfrak{b}}(r):=\mathfrak{b}(r)-b_{0}r=\int_{0}^{r}b(\xi)d\xi-b_{0}r,\ \ \ \ \ r \in\,\mathbb{R}.
\end{equation*}
Then since Hypothesis \ref{Hypothesis3}\ref{nondeg_condition} implies that $\mathfrak{b}\in C^{2}(\mathbb{R})$ has at most polynomial growth of $\theta$, and for all $r_{1},r_{2}\in\mathbb{R}$
\begin{equation*}
	(\mathfrak{b}(r_{1})-\mathfrak{b}(r_{2}))(r_{1}-r_{2})\geq b_{0}\abs{r_{1}-r_{2}}^{2},
\end{equation*}
 we have $\tilde{\mathfrak{b}}\in C^{2}(\mathbb{R})$ is monotone, so that we can introduce its Yosida approximation by setting
\begin{equation*}
	J_{\epsilon}(r):=(I+\epsilon \tilde{\mathfrak{b}})^{-1}(r),\ \ \ \  \tilde{\mathfrak{b}}_{\epsilon}(r)=\tilde{\mathfrak{b}}(J_{\epsilon}(r))=\frac{1}{\epsilon}\big(r-J_{\epsilon}(r)\big), \ \ \ r\in\mathbb{R},\ \ \ \epsilon>0,
\end{equation*}
and then, we define
\begin{equation*}
	\mathfrak{b}_{\epsilon}(r):=\tilde{\mathfrak{b}}_{\epsilon}(r)+b_{0}r,\ \ \ \ \ b_{\epsilon}(r):=\mathfrak{b}_{\epsilon}'(r),\ \ \ r\in\mathbb{R}.
\end{equation*} 
It can be easily verified that 
\begin{equation*}
	J_{\epsilon}'(r)=\frac{1}{1+\epsilon\big[b(J_{\epsilon}(r))-b_{0}\big]},\ \ \ \ \ b_{\epsilon}(r)=\mathfrak{b}_{\epsilon}'(r)=b_{0}+\frac{b(J_{\epsilon}(r))-b_{0}}{1+\epsilon\big[b(J_{\epsilon}(r))-b_{0}\big]},
\end{equation*}
and moreover, the following lemma holds. The proof is standard and we omit it here.
\begin{lem}
	\label{lem_B_approx}
	 Assume Hypothesis \ref{Hypothesis3}\ref{nondeg_condition}.
	
	\begin{enumerate}
		\item For every $\epsilon>0$, $J_{\epsilon}:\mathbb{R}\to\mathbb{R}$ is a contraction. Namely, for all $r_{1},r_{2}\in\mathbb{R}$
		\begin{equation}\label{Lip_J}
			\abs{J_{\epsilon}(r_{1})-J_{\epsilon}(r_{2})}\leq \abs{r_{1}-r_{2}}.
		\end{equation}
		
		\smallskip
		
		\item For every $\epsilon>0$, $\tilde{\mathfrak{b}}_{\epsilon}\in C^{2}(\mathbb{R})$ is monotone and Lipschitz continuous with Lipschitz constant $2/\epsilon$. Moreover, there exists some constant $c>0$ such that
		\begin{equation}\label{uni_B_approx}
			\abs{\tilde{\mathfrak{b}}_{\epsilon}(r)}\leq \abs{\tilde{\mathfrak{b}}(r)}\leq c(1+\abs{r}^{\theta}),\ \ \  r\in\mathbb{R}.
		\end{equation}
		
		\smallskip
		
		\item For every $\epsilon>0$ we have
		\begin{equation}\label{conv_J}
			\abs{J_{\epsilon}(r)-r}\leq \epsilon\abs{\tilde{\mathfrak{b}}_{\epsilon}(r)},\ \ \ r\in\mathbb{R}.
		\end{equation}
	\end{enumerate}
\end{lem}

		By Lemma \ref{lem_B_approx} we have $b_{\epsilon}\in C^{1}(\mathbb{R})$ with
		\begin{equation}\label{boundedness_b_approx}
			b_{0}\leq b_{\epsilon}(r)\leq b_{0}+2/\epsilon,\ \ \ r\in\mathbb{R}.
		\end{equation}
		 Thanks to \eqref{uni_B_approx} we know there exists $c>0$ such that for every $\epsilon>0$
		\begin{equation*}
			\abs{\mathfrak{b}_{\epsilon}(r)}\leq c(1+\abs{r}^{\theta}),\ \ \ \ r\in\mathbb{R}.
		\end{equation*}
		Moreover, due to conditions \eqref{Lip_J} and \eqref{conv_J}, we have 
		\begin{equation*}
			\abs{\mathfrak{b}_{\epsilon}(r)-\mathfrak{b}(r)}\leq c\abs{J_{\epsilon}(r)-r}\big(1+\abs{J_{\epsilon}(r)}^{\theta-1}+\abs{r}^{\theta-1}\big)\leq C\epsilon\abs{\tilde{\mathfrak{b}}_{\epsilon}(r)}\big(1+\abs{r}^{\theta-1}\big),
		\end{equation*}
		and 
		\begin{equation*}
			\abs{\mathfrak{b}_{\epsilon}(r_{1})-\mathfrak{b}_{\epsilon}(r_{2})}\leq c\abs{r_{1}-r_{2}}\big(1+\abs{J_{\epsilon}(r_{1})}^{\theta-1}+\abs{J_{\epsilon}(r_{2})}^{\theta-1}\big)\leq C\abs{r_{1}-r_{2}}\big(1+\abs{r_{1}}^{\theta-1}+\abs{r_{2}}^{\theta-1}\big).
		\end{equation*}
		This implies that for every $u\in L^{2\theta}(\O)$
		\begin{equation}\label{polynomial_B}
			\norm{\mathfrak{b}_{\epsilon}(u)}_{H}\leq c\Big(1+\norm{u}_{L^{2\theta}(\O)}^{\theta}\Big),\ \ \ \epsilon>0,
		\end{equation}
		\begin{equation}\label{convergence_B}
			\norm{\mathfrak{b}_{\epsilon}(u)-\mathfrak{b}(u)}_{L^{1}(\O)}\leq c\epsilon\Big(1+\norm{u}_{L^{2\theta}(\O)}^{\theta}\Big)\Big(1+\norm{u}_{L^{2(\theta-1)}(\O)}^{\theta-1}\Big),\ \ \ \epsilon>0,
		\end{equation}
		and for every $u_{1},u_{2}\in L^{2\theta}(\O)$
		\begin{equation}
			\label{diff_B}
			\norm{\mathfrak{b}_{\epsilon}(u_{1})-\mathfrak{b}_{\epsilon}(u_{2})}_{L^{1}(\O)}\leq c\norm{u_{1}-u_{2}}_{H}\Big(1+\norm{u_{1}}_{L^{2(\theta-1)}(\O)}^{\theta-1}+\norm{u_{2}}_{L^{2(\theta-1)}(\O)}^{\theta-1}\Big),\ \ \ \epsilon>0.
		\end{equation}

Now we are ready to introduce the following approximation problem for \eqref{limit_para}.
\begin{equation}\label{weak_nondeg_eq_approx}
	\le\{\begin{array}{l}
		\ds{\partial_{t}u^{\epsilon}(t,x)=\text{div}\big(b_{\epsilon}(u^{\epsilon}(t,x))\nabla u^{\epsilon}(t,x)\big)+\sigma(u^{\epsilon}(t,x))\partial_{t}w(t,x), }\\[10pt]
		\ds{u^{\epsilon}(0,x)=u_{0},\ \ \ u^{\epsilon}(t,\cdot)|_{\partial\mathcal{O}}=0 },
	\end{array}\r.
\end{equation}

\begin{prop}\label{weak_nondeg_approx} 
	
	Assume Hypotheses \ref{Hypothesis1} and \ref{Hypothesis3}\ref{nondeg_condition}. Let $p\geq2$, for every $u_{0}\in L^{p}(\O)$ and every $\epsilon,T>0$ there exists a unique weak solution 
	\begin{equation*}
		u^{\epsilon}\in L^{2}(\Omega;C([0,T];H))\bigcap  L^{p}(\Omega;L^{\infty}([0,T];L^{p}(\O)))\bigcap L^{2}(\Omega;L^{2}([0,T];H^{1})).
	\end{equation*}
	of equation \eqref{weak_nondeg_eq_approx}, which has $\P$-almost surely continuous trajectories in $L^{p}(\O)$. Moreover, for all $p,q\in[1,+\infty)$ there is a constant $C_{T,p,q}>0$ such that
	\begin{equation}\label{approx_energy_Lp}
		\begin{array}{ll}
			&\ds{\sup_{\epsilon\in(0,1)}\Bigg[\E\sup_{t\in[0,T]}\norm{u^{\epsilon}(t)}_{L^{p}(\O)}^{pq}+\E\Big(\int_{0}^{T}\int_{\O}\big(1+\abs{u^{\epsilon}(s,x)}^{2}\big)^{\frac{p}{2}-1}b_{\epsilon}(u^{\epsilon}(s,x))\abs{\nabla u_{\epsilon}(s,x)}^{2}dxds\Big)^{q}\Bigg]}\\
			\vs
			&\ds{\quad\quad\leq C_{T,p,q}\Big(1+\norm{u_{0}}_{L^{p}(\O)}^{pq}\Big)}.
			\end{array}
	\end{equation}
\end{prop}

\begin{proof}
 Let $u_{0}\in L^{p}(\O)$, for some $p\geq2$. According to \cite[Proposition A.4]{cerraixie2023} and thanks to \eqref{boundedness_b_approx}, we know that for every $T,\epsilon>0$ there exists a unique weak solution
 \begin{equation*}
 	u^{\epsilon}\in L^{2}(\Omega;C([0,T];H)\bigcap L^{2}([0,T];H^{1}))
 \end{equation*}
 for equation \eqref{weak_nondeg_eq_approx}. Now we aim to obtain a uniform estimate for $u^{\epsilon}$ in $L^{p}(\O)$-norm with respect to $\epsilon$. Here we will only provide the proof of \eqref{approx_energy_Lp} for $p\geq2$. The case when $p\in[1,2)$ can be proved by following the approach in \cite[Proof of Proposition 4.4]{gess18} and we omit details here.
 
 Fix any $\epsilon>0$. Let $p\geq2$, we consider the following approximation for $\psi(\xi)=\abs{\xi}^{p}$, $\xi\in\mathbb{R}$, which was used in the proof of \cite[Proposition 5.1]{debussche16}:
 \begin{equation*}
 	\psi_{n}(\xi)=
 	\le\{\begin{array}{l}
 		\ds{\abs{\xi}^{p},\ \ \ \ \abs{\xi}\leq n,}\\
 		\vs
 		\ds{n^{p-2}\Bigg[\frac{p(p-2)}{2}\xi^{2}-p(p-2)n\abs{\xi}+\frac{(p-1)(p-2)}{2}n^{2}\Bigg],\ \ \ \ \abs{\xi}>n}.
 	\end{array}\r.
 \end{equation*}
 Then it follows that $\psi_{n}\in C^{2}(\mathbb{R})$ with $\psi_{n}''\in L^{\infty}(\mathbb{R})$. Moreover, notice that for every $n\in\mathbb{N}$
 \begin{equation*}
 	\psi_{n}'(u^{\epsilon})\in L^{2}(\Omega\times[0,T];H^{1}),
 \end{equation*}
 by Remark \ref{remark_gen_ito} we can take $\psi_{n}$ and $\phi\equiv1$ as test functions in \eqref{gen_ito} to obtain 
 \begin{equation}\label{base_Lp}
 	\begin{array}{ll}
 		&\ds{\int_{\O}\psi_{n}(u^{\epsilon}(t,x))dx =\int_{\O}\psi_{n}(u_{0}(x))dx-\int_{0}^{t}\int_{\O}\nabla \mathfrak{b}_{\epsilon}(u^{\epsilon}(s,x))\cdot \nabla\psi_{n}'(u^{\epsilon}(s,x))dxds }\\
 		\vs
 		&\ds{ \quad\quad+\frac{1}{2}\int_{0}^{t}\int_{\O}\psi_{n}''(u^{\epsilon}(s,x))\Sigma^{2}(x,u^{\epsilon}(s,x))dxds+ \sum_{k=1}^{\infty}\int_{0}^{t}\int_{\O}\psi_{n}'(u^{\epsilon}(s,x))\sigma_{k}(x,u^{\epsilon}(s,x))dxd\beta_{k}(s) }
 	\end{array}
 \end{equation}
 Since $b_{\epsilon}\geq b_{0}>0$ and $\phi_{n}''\geq0$, we have
 \begin{equation*}
 	\begin{array}{ll}
 		\ds{-\int_{0}^{t}\int_{\O}\nabla \mathfrak{b}_{\epsilon}(u^{\epsilon}(s,x))\cdot \nabla\psi_{n}'(u^{\epsilon}(s,x))dxds =-\int_{0}^{t}\int_{\O}\psi_{n}''(u^{\epsilon}(s,x))b_{\epsilon}(u^{\epsilon}(s,x))\abs{\nabla u^{\epsilon}(s,x)}^{2}dxds\leq 0 },
 	\end{array}
 \end{equation*}
 For any $q\geq1$, due to \eqref{sgfine2} and the property $\xi^{2}\psi_{n}''(\xi)\leq p(p-1)\psi_{n}(\xi)$, $\xi\in\mathbb{R}$, it follows that
 \begin{equation*}
 	\begin{array}{ll}
 		&\ds{\Bigg\lvert\int_{0}^{t}\int_{\O}\psi_{n}''(u^{\epsilon}(s,x))\Sigma^{2}(x,u^{\epsilon}(s,x))dxds\Bigg\rvert^{q} \leq c\Bigg\lvert\int_{0}^{t}\int_{\O}\psi_{n}''(u^{\epsilon}(s,x))\big(1+\abs{u^{\epsilon}(s,x)}^{2}\big)dxds\Bigg\rvert^{q} }\\
 		\vs
 		&\ds{\quad\quad\quad\quad\quad\quad\quad\quad\quad\quad\quad\quad\quad\quad\quad\quad\quad\quad\quad\leq C_{T,p,q}\Bigg(1+\int_{0}^{T}\Big(\int_{\O}\psi_{n}(u^{\epsilon}(s,x))dx\Big)^{q}ds\Bigg) }.
 	\end{array}
 \end{equation*}
Moreover, by proceeding as in the proof of \cite[Proposition 5.1]{debussche16}, we have 
 \begin{equation*}
 	\begin{array}{ll}
 		&\ds{\E\sup_{t\in[0,T]}\Bigg\lvert \sum_{k=1}^{\infty}\int_{0}^{t}\int_{\O}\psi_{n}'(u^{\epsilon}(s,x))\sigma_{k}(x,u^{\epsilon}(s,x))dxd\beta_{k}(s)\Bigg\rvert^{q}}\\
 		\vs
 		&\ds{\leq C_{p,q}\E\Bigg(\int_{0}^{T}\int_{\O}\psi_{n}(u^{\epsilon}(s))dx\Big(1+\int_{\O}\psi_{n}(u^{\epsilon}(s))dx\Big)ds\Bigg)^{\frac{q}{2}} }\\
 		\vs
 		&\ds{\leq \frac{1}{2}\E\sup_{t\in[0,T]}\Big(\int_{\O}\psi_{n}(u^{\epsilon}(t))dx\Big)^{q}+C_{T,p,q}\Bigg(1+\E\int_{0}^{T}\Big(\int_{\O}\psi_{n}(u^{\epsilon}(s))dx\Big)^{q}ds\Bigg)},
 	\end{array}.
 \end{equation*}
Then, taking the supremum, the $q$-th power and the expection on both sides of \eqref{base_Lp}, we obtain that 
\begin{equation*}
	\begin{array}{ll}
		&\ds{\E\sup_{t\in[0,T]}\Big(\int_{\O}\psi_{n}(u^{\epsilon}(t))dx\Big)^{q}+\E\Bigg(\int_{0}^{T}\int_{\O}\psi_{n}''(u^{\epsilon}(s,x))b_{\epsilon}(u^{\epsilon}(s,x))\abs{\nabla u^{\epsilon}(s,x)}^{2}dxds\Bigg)^{q} }\\
		\vs
		&\ds{\leq \Big(\int_{\O}\psi(u_{0}(x))dx\Big)^{q}+\frac{1}{2}\E\sup_{t\in[0,T]}\Big(\int_{\O}\psi_{n}(u^{\epsilon}(t))dx\Big)^{q}+C_{T,p,q}\Bigg(1+\E\int_{0}^{T}\Big(\int_{\O}\psi_{n}(u^{\epsilon}(s))dx\Big)^{q}ds\Bigg) },
	\end{array}
\end{equation*}
uniformly in $\epsilon>0$. Hence, applying the Gr{\"o}nwall lemma and then letting $n\to+\infty$ yields \eqref{approx_energy_Lp}.

\end{proof}

\subsubsection{Proof of Theorem \ref{weak_nondeg}}

	From Corollary \ref{uniqueness_kinetic}, we know there is at most one weak solution for equation \eqref{limit_para} in $L^{2\theta}(\Omega;L^{\infty}([0,T];L^{2\theta}(\O)))\bigcap L^{2}(\Omega;L^{2}([0,T];H^{1}))$. Hence, if we show that there exists a probabilistically weak solution for equation \eqref{limit_para}
	\begin{equation*}
		\big(\hat{\Omega},\hat{\F},\{\hat{\F}\}_{t},\hat{\P},\hat{w},\hat{u}\big)
	\end{equation*}  
	such that $\hat{u}\in L^{2\theta}(\hat{\Omega};L^{\infty}([0,T];L^{2\theta}(\O)))\bigcap L^{2}(\hat{\Omega};L^{2}([0,T];H^{1}))$, the existence and uniqueness of a weak solution for equation \eqref{limit_para} follows by applying the Gy{\"o}ngy-Krylov characterization of convergence in probability (see \cite{gyongy1996}).
	
	Fix $u_{0}\in L^{2\theta}(\O)$. According to Proposition \ref{weak_nondeg_approx}, for every $\epsilon>0$ there exists a unique weak solution $u^{\epsilon}$ to equation \eqref{weak_nondeg_eq_approx}, and
	\begin{equation}\label{uniform_bound1}
		\sup_{\epsilon\in(0,1)}\Bigg(\E\sup_{t\in[0,T]}\norm{u^{\epsilon}(t)}_{L^{2\theta}(\O)}^{2\theta}+\E\int_{0}^{T}\int_{\O}\big(1+\abs{u^{\epsilon}(s,x)}^{2}\big)^{\theta-1}b_{\epsilon}(u^{\epsilon}(s,x))\abs{\nabla u^{\epsilon}(s,x)}^{2}dxds\Bigg)<\infty.
	\end{equation}

 We will divide our proof into two steps as follows.\\

{\em Step 1:} Fix $u_{0}\in L^{2\theta}(\O)$. The family of probability measures $\{\L(u^{\epsilon})\}_{\epsilon\in(0,1)}$ is tight in 
\begin{equation*}
	L^{p}([0,T];H)\bigcap C([0,T];H^{-\delta}),
\end{equation*}
 for every $p>1$ and $\delta>0$.\\ 

	{\em Proof of Step 1:} 	 
	For every $h\in(0,T)$ and $t\in[0,T-h]$ we have 
	\begin{equation*}
		u^{\epsilon}(t+h)-u^{\epsilon}(t)=\int_{t}^{t+h}\Delta \mathfrak{b}_{\epsilon}(u^{\epsilon}(s))ds+\int_{t}^{t+h}\sigma(u^{\epsilon}(s))dw(s).
	\end{equation*}
	From \eqref{polynomial_B}, for every $\epsilon\in(0,1)$ and every $t\in[0,T-h]$ we have
	\begin{equation*}
		\E\int_{t}^{t+h}\norm{\Delta \mathfrak{b}_{\epsilon}\big(u^{\epsilon}(r)\big)}_{H^{-2}}dr\leq ch\ \E\sup_{r\in[0,T]}\norm{\mathfrak{b}_{\epsilon}(u^{\epsilon}(r))}_{H}\leq ch\Big(1+\E\sup_{t\in[0,T]}\norm{u^{\epsilon}(r)}_{L^{2\theta}(\O)}^{\theta}\Big),
	\end{equation*}
	and due to \eqref{sgfine2}, for any $a>2$ it follows that
	\begin{equation*}
		\E\norm{\int_{t}^{t+h}\sigma(u^{\epsilon}(r))dw(r)}_{H}^{a}\leq c\E\Big(\int_{t}^{t+h}\norm{\sigma(u^{\epsilon}(r))}_{L_{2}(\mathcal{U},H)}^{2}dr\Big)^{\frac{a}{2}}=  ch^{\frac{a}{2}}\Big(1+\E\sup_{r\in[0,T]}\norm{u^{\epsilon}(r)}_{H}^{a}\Big),
	\end{equation*}
	Hence, as a consequence of the Kolmogorov continuity theorem (see \cite[Theorem 3.3]{daprato2014}), we obtain that for any $\lambda\in(0,1/2)$,  
	\begin{equation}\label{tight2}
		\sup_{\epsilon\in(0,1)}\E\norm{u^{\epsilon}}_{C^{\lambda}([0,T];H^{-2})}<\infty,
	\end{equation}
	and since \eqref{uniform_bound1} implies, in particular, that 
	\begin{equation*}
		\sup_{\epsilon\in(0,1)}\E\norm{u^{\epsilon}}_{C([0,T];H)}<\infty,
	\end{equation*}
	due to \cite[Theorem 7]{simon1986} this implies that $\big\{\mathcal{L}(u^{\epsilon})\big\}_{\epsilon\in(0,1)}$ is tight in $C([0,T];H^{-\delta})$, for every $\delta>0$. Moreover, thanks again to \eqref{uniform_bound1}, by proceeding as in \cite[Proof of Theorem 5.1]{cerraixi}, we obtain that the family $\big\{\mathcal{L}(u^{\epsilon})\big\}_{\epsilon\in(0,1)}$ is tight in $L^{p}([0,T];H)$, for every $p>1$.
	
	In what follows, we define
	\begin{equation*}
		\mathcal{K}(T):=\Big[L^{2}([0,T];H)\bigcap C([0,T];H^{-1})\Big]\times C([0,T];\mathcal{U}_{0}),
	\end{equation*}
	where $\mathcal{U}_{0}$ is any Hilbert space such that the embedding $\mathcal{U}\hookrightarrow\mathcal{U}_{0}$ is Hilbert-Schmidt.\\
	
	{\em Step 2:} There exists a filtered probability space $\big(\hat{\Omega},\hat{\F},\{\hat{\F}\}_{t},\hat{\P}\big)$, a cylindrical Wiener process $\hat{w}$ associated with $\{\hat{\F}\}_{t}$ and a process 
\begin{equation*}
	\hat{u}\in L^{2\theta}(\hat{\Omega};L^{\infty}([0,T];L^{2\theta}(\O)))\bigcap L^{2}(\hat{\Omega};L^{2}([0,T];H^{1}))
\end{equation*}
such that $\mathfrak{b}(\hat{u})\in L^{2}(\hat{\Omega};L^{2}([0,T];H^{1}))$ and for every $\varphi\in C^{\infty}_{c}(\O)$
\begin{equation*}
	\Inner{\hat{u}(t),\varphi}=\Inner{u_{0},\varphi}_{H}+\int_{0}^{t}\Inner{\mathfrak{b}(\hat{u}(s)),\Delta\varphi}_{H}ds+\int_{0}^{t}\Inner{\varphi,\sigma(\hat{u}(s))d\hat{w}(s)}_{H},\ \ \ \hat{\P}\text{-a.s.}. 
\end{equation*}

	{\em Proof of Step 2:}
	  By the tightness of $\{\L(u^{\epsilon},w)\}_{\epsilon\in(0,1)}$ in $\mathcal{K}(T)$, there exists a sequence $\epsilon_{n}\downarrow0$ such that $\L(u^{\epsilon_{n}},w)$ is weakly convergent in $\mathcal{K}(T)$. Now let $u$ be any weak limit point for $\{u^{\epsilon_{n}}\}$, then due to Skorohod's Theorem, there exist a sequence of $\mathcal{K}(T)$-valued random variables $\mathcal{Y}_{n}=(\hat{u}_{n},\hat{w}_{n})$, $\mathcal{Y}=(\hat{u},\hat{w})$, and a probabilty space $\big(\hat{\Omega},\hat{\mathcal{F}},\{\hat{\mathcal{F}}_{t}\}_{t\in[0,T]},\hat{\P}\big)$ such that
	\begin{equation*}
	\L(\mathcal{Y})=\L(u,w),	\ \ \ \ \mathcal{L}(\mathcal{Y}_{n})=\mathcal{L}(u_{n},w),\ \ \ n\in\mathbb{N},
	\end{equation*} 
	and 
	\begin{equation}\label{convergence_almost_surely}
		\lim_{n\to\infty}\Big(\norm{\hat{u}_{n}-\hat{u}}_{L^{2}([0,T];H)}+\norm{\hat{u}_{n}-\hat{u}}_{C([0,T];H^{-1})}+\norm{\hat{w}_{n}-\hat{w}}_{C([0,T];\mathcal{U}_{0})}\Big)=0,\ \ \ \hat{\P}\text{-a.s.}.
	\end{equation}
	This, together with the uniform bound \eqref{uniform_bound1} and the fact that $b_{\epsilon}\geq b_{0}$, yields
	\begin{equation*}
		\hat{u}\in L^{2\theta}(\hat{\Omega};L^{\infty}([0,T];L^{2\theta}(\O)))\bigcap L^{2}(\hat{\Omega}; L^{2}([0,T];H^{1})),
	\end{equation*}
	with
	\begin{equation*}
		\hat{\E}\int_{0}^{T}\int_{\O}\big(1+\abs{\hat{u}(s,x)}^{2}\big)^{\theta-1}\abs{\nabla \hat{u}(s,x)}^{2}dxds<\infty,
	\end{equation*}
	which implies that $\mathfrak{b}(\hat{u})\in L^{2}(\hat{\Omega};L^{2}([0,T];H^{1}))$. And moreover, we may extract a subsequence $\{\hat{u}_{n}\}$ (not relabeled) such that
	\begin{equation}\label{convergence_Lp1}
		\lim_{n\to\infty}\ \hat{u}_{n}=\hat{u},\ \ \text{in}\ \  L^{2}(\hat{\Omega};L^{2}([0,T];L^{2}(\O))).
	\end{equation}
	Now, we denote $\mathfrak{b}_{n}:=\mathfrak{b}_{\epsilon_{n}}$, then for every $\varphi\in C^{\infty}_{0}(\mathcal{O})$,
	\begin{equation}\label{weak_test_approx}
	\Inner{\hat{u}_{n}(t),\varphi}_{H} =\Inner{u_{0},\varphi}_{H}+\int_{0}^{t}\Inner{\mathfrak{b}_{n}(\hat{u}_{n}(s)),\Delta\varphi}_{H}ds +\int_{0}^{t}\Inner{\varphi,\sigma(\hat{u}_{n}(s))d\hat{w}_{n}(s)}_{H},\ \ \ \hat{\P}\text{-a.s.}
	\end{equation}
	In order to take the limit as $n\to\infty$ of both sides in \eqref{weak_test_approx}, we only have to study the limit of the second term on the left-hand side, and the rest terms can be treated by using the general argument as in \cite[Proof of Theorem 7.1]{cerraixie}. We will show that 
	\begin{equation}
		\lim_{n\to\infty}\E\sup_{r\in[0,t]}\Big\lvert \int_{0}^{r}\Inner{\mathfrak{b}_{n}(\hat{u}_{n}(s))-\mathfrak{b}(\hat{u}(s)),\Delta\varphi}_{H}ds\Big\rvert=0,\ \ \ t\in[0,T].
	\end{equation}
	Indeed, thanks to \eqref{convergence_B} and \eqref{diff_B}, for every $t\in[0,T]$ we have
	\begin{equation*}
		\begin{array}{ll}
			&\ds{\sup_{r\in[0,t]}\Big\lvert \int_{0}^{r}\Inner{\mathfrak{b}_{n}(\hat{u}_{n}(s))-\mathfrak{b}(\hat{u}(s)),\Delta\varphi}_{H}ds\Big\rvert  }\\
			\vs
			&\ds{\leq \norm{\Delta\varphi}_{\infty}\Bigg(\int_{0}^{t}\norm{\mathfrak{b}_{n}(\hat{u}_{n}(s))-\mathfrak{b}_{n}(\hat{u}(s))}_{L^{1}(\O)}ds +\int_{0}^{t}\norm{\mathfrak{b}_{n}(\hat{u}(s))-\mathfrak{b}(\hat{u}(s))}_{L^{1}(\O)}ds\Bigg) }\\
			\vs
			&\ds{\leq c\norm{\Delta\varphi}_{\infty}\int_{0}^{t}\norm{\hat{u}_{n}(s)-\hat{u}(s)}_{H}\Big(1+\norm{\hat{u}_{n}(s)}_{L^{2(\theta-1)}(\O)}^{\theta-1}+\norm{\hat{u}(s)}_{L^{2(\theta-1)}(\O)}^{\theta-1}\Big)ds }\\
			\vs
			&\ds{\quad +c\epsilon_{n}\norm{\Delta\varphi}_{\infty}\int_{0}^{t}\Big(1+\norm{\hat{u}(s)}_{L^{2\theta}(\O)}^{\theta}\Big)\Big(1+\norm{\hat{u}(s)}_{L^{2(\theta-1)}(\O)}^{\theta-1}\Big)ds},
		\end{array}
	\end{equation*}
	and thus as $n\to\infty$, it follows from \eqref{uniform_bound1} and \eqref{convergence_Lp1} that
	\begin{equation*}
		\begin{array}{ll}
			&\ds{\E\sup_{r\in[0,t]}\Big\lvert \int_{0}^{r}\Inner{\mathfrak{b}_{n}(\hat{u}_{n}(s))-\mathfrak{b}(\hat{u}(s)),\Delta\varphi}_{H}ds\Big\rvert }\\
			\vs
			&\ds{\leq c\norm{\Delta\varphi}_{\infty}\Bigg(\E\int_{0}^{t}\norm{\hat{u}_{n}(s)-\hat{u}(s)}_{H}^{2}ds\Bigg)^{\frac{1}{2}}\Bigg(\E\int_{0}^{t}\Big(1+\norm{\hat{u}_{n}(s)}_{L^{2(\theta-1)}(\O)}^{2(\theta-1)}+\norm{\hat{u}(s)}_{L^{2(\theta-1)}(\O)}^{2(\theta-1)}\Big)ds\Bigg)^{\frac{1}{2}} }\\
			\vs
			&\ds{\quad +c\epsilon_{n}\norm{\Delta\varphi}_{\infty}\Bigg(\E\int_{0}^{t}\Big(1+\norm{\hat{u}(s)}_{L^{2\theta}(\O)}^{2\theta}\Big)ds\Bigg)^{\frac{1}{2}}\Bigg(\E\int_{0}^{t}\Big(1+\norm{\hat{u}(s)}_{L^{2(\theta-1)}(\O)}^{2(\theta-1)}\Big)ds\Bigg)^{\frac{1}{2}}\to0 }.
		\end{array}
	\end{equation*}
	Therefore, by proceeding as in \cite{cerraixie}, we may extract a subsequence and take the limit as $n\to\infty$ on both sides of \eqref{weak_test_approx} to obtain that $\hat{u}$ satisfies \eqref{weak_test}, with $w$ replaced by $\hat{w}$. Due to what we have seen above, there exists a unique weak solution $u$ of equation \eqref{limit_para}. Moreover, it can be verified that $u$ satisfies \eqref{energy_Lp} by using the same argument as in the proof of Proposition \ref{weak_nondeg_approx}. Finally, the fact that $u\in C([0,T];L^{2\theta}(\O))$, $\P$-a.s., is a direct consequence of Lemma \ref{continuity_Lp}.

\bigskip

\subsection{The degenerate case}
\label{sec_degenerate}

Throughout this subsection we assume Hypothesis \ref{Hypothesis3}\ref{deg_condition} holds. Given any initial data $u_{0}\in L^{2\theta_{2}}(\O)$, the well-posdeness for \eqref{limit_para} is given in the following theorem.

\begin{thm}\label{kin_deg}Assume Hypotheses \ref{Hypothesis1}, \ref{Hypothesis2} and \ref{Hypothesis3}\ref{deg_condition}. For every $u_{0}\in L^{2\theta_{2}}(\mathcal{O})$ and every $T>0$, there exists a unique kinetic solution $u\in C([0,T];L^{2\theta_{2}}(\mathcal{O}))$, $\P$-a.s., of equation \eqref{limit_para}. For every $p\in[1,2\theta_{2}]$ and $q\in[1+\infty)$, there is a constant $C_{T,p,q}>0$ such that
	\begin{equation}\label{energy_kinetic_Lp}
		\E\ \esssup_{t\in[0,T]}\norm{u(t)}_{L^{p}(\mathcal{O})}^{pq}\leq C_{T,p,q}\Big(1+\norm{u_{0}}_{L^{p}(\mathcal{O})}^{pq}\Big).
	\end{equation}
	Moreover, if $u_{1},u_{2}$ are kinetic solutions of equation \eqref{limit_para} with initial data $u_{1,0},u_{2,0}\in L^{2\theta_{2}}(\O)$, respectively, we have
	\begin{equation}\label{L1_contraction_kinetic_Lp}
		\E\norm{u_{1}(t)-u_{2}(t)}_{L^{1}(\mathcal{O})}\leq \norm{u_{1,0}-u_{2,0}}_{L^{1}(\mathcal{O})},\ \ \ t\in[0,T].
	\end{equation}
\end{thm}

\subsubsection{The approximation problem}

Fix $u_{0}\in L^{2\theta_{2}}(\O)$ and assume Hypotheses \ref{Hypothesis1}, \ref{Hypothesis2} and \ref{Hypothesis3}\ref{deg_condition}. For every $\tau\in(0,1)$, we consider the following approximation problem for \eqref{limit_para}:
\begin{equation}\label{kin_deg_approx_eq}
	\le\{\begin{array}{l}
		\ds{\partial_{t}u^{\tau}(t,x)=\text{div}\big(b(u^{\tau}(t,x))\nabla u^{\tau}(t,x)\big)+\tau\Delta u^{\tau}(t,x)+\sigma(u^{\tau}(t,x))\partial_{t}w(t,x), }\\[10pt]
		\ds{u^{\tau}(0,x)=u_{0},\ \ \ u^{\tau}(t,\cdot)|_{\partial\mathcal{O}}=0 },
	\end{array}\r.
\end{equation}
Denote $b^{\tau}(r):=b(r)+\tau$, $r\in\mathbb{R}$, then it is clear that $b^{\tau}$ satisfies Hypothesis \ref{Hypothesis3}\ref{nondeg_condition}. Hence, by Theorem \ref{weak_nondeg} we know that for every $\tau\in(0,1)$ there exists a unique weak solution $u^{\tau}\in L^{2\theta_{2}}(\Omega;L^{\infty}([0,T];L^{2\theta_{2}}(\O)))\bigcap L^{2}(\Omega;L^{2}([0,T];H^{1}))$ of equation \eqref{limit_para}, and moreover, for all $p\in[1,2\theta_{2}]$ and $q\in[1,+\infty)$ there exists some constant $C_{T,p,q}>0$ such that
 \begin{equation}\label{kin_energy_Lp_approx}
 	\begin{array}{ll}
 	&\ds{\sup_{\tau\in(0,1)}\Bigg[\E\sup_{t\in[0,T]}\norm{u^{\tau}(t)}_{L^{p}(\O)}^{pq}+\E\Big(\int_{0}^{T}\int_{\O}\big(1+\abs{u^{\tau}(s,x)}^{2}\big)^{\frac{p}{2}-1}b^{\tau}(u(s,x))\abs{\nabla u^{\tau}(s,x)}^{2}dxds\Big)^{q}\Bigg]}\\
 	\vs
 	&\ds{\quad\quad\quad\leq C_{T,p,q}\Big(1+\norm{u_{0}}_{L^{p}(\O)}^{pq}\Big)}.
 	\end{array}
 \end{equation}
Denoting $h^{\tau}(t,x,\xi):=\mathbf{1}_{u^{\tau}(t,x)>\xi}$, we know the following kinetic formulation holds in the sense of $\mathcal{D}'(\O\times\mathbb{R})$:
\begin{equation}\label{kinetic_form}
	dh^{\tau}-(b(\xi)+\tau)\Delta h^{\tau}dt=\sum_{i=1}^{\infty}\delta_{u^{\tau}=\xi}\sigma_{i}(x,\xi)d\beta_{i}(t)+\partial_{\xi}q_{\tau}dt,
\end{equation}
where 
\begin{equation*}
	q_{\tau}(t,x,\xi)=m_{\tau}-\frac{1}{2}\delta_{u^{\tau}=\xi}\Sigma^{2}(x,\xi),\ \ \ \ m_{\tau}=n_{1,\tau}+n_{2,\tau},
\end{equation*}
and 
\begin{equation*}
	dn_{1,\tau}(t,x,\xi)=b(\xi)\abs{\nabla u^{\tau}}^{2}d\delta_{u^{\tau}=\xi}dxdt,\ \ \ \ dn_{2,\tau}(t,x,\xi)=\tau\abs{\nabla u^{\tau}}^{2}d\delta_{u^{\tau}=\xi}dxdt.
\end{equation*}
Then, if we define $\chi^{\tau}(t,x,\xi):=h^{\tau}(t,x,\xi)-\mathbf{1}_{0>\xi}$, then $\chi^{\tau}=\chi^{\tau}(t,x,\xi)$ satisfies that
\begin{equation*}
	\partial_{t}\chi^{\tau}-(b(\xi)+\tau)\Delta\chi^{\tau}=\partial_{\xi}q_{\tau}-\sum_{i=1}^{\infty}(\partial_{\xi}\chi^{\tau})\sigma_{i}\dot{\beta}_{i}+\sum_{i=1}^{\infty}\delta_{0}\sigma_{i}\dot{\beta}_{i}.
\end{equation*}
Hence, for every $\varphi\in C_{c}^{\infty}([0,T)\times\O)$, we have $\varphi\chi^{\tau}$ satisfies 
\begin{equation*}
	\begin{array}{ll}
		&\ds{	\partial_{t}(\varphi\chi^{\tau})-(b(\xi)+\tau)\Delta(\varphi\chi^{\tau})=\partial_{\xi}(\varphi q_{\tau})-\sum_{i=1}^{\infty}(\partial_{\xi}(\varphi\chi^{\tau}))\sigma_{i}\dot{\beta}_{i}+\sum_{i=1}^{\infty}\varphi\delta_{0}\sigma_{i}\dot{\beta}_{i}}\\
		\vs
		&\ds{\quad\quad\quad\quad\quad\quad\quad\quad\quad\quad\quad\quad\quad 
		-(b(\xi)+\tau)\big(\chi^{\tau}\Delta\varphi+2\nabla\chi^{\tau}\cdot\nabla\varphi\big) }
	\end{array}
\end{equation*}

Now we denote by $\L^{\tau}$ the symbol assiciated to the kinetic form \eqref{kinetic_form}
\begin{equation*}
	\L^{\tau}(iu,in,\xi):=iu+(b(\xi)+\tau)n^{2}
\end{equation*}
For $J,\delta>0$, since for some constant $C>0$,
\begin{equation*}
	\big\{\xi\in\mathbb{R}:\mathcal{L}^{\tau}(iu,in,\xi)\leq \delta\big\}\subseteq \big\{\xi\in\mathbb{R}:\L(iu,in,\xi)\leq C\delta\big\},
\end{equation*}
where $\L$ is the symbol associated to the kinetic form for \eqref{limit_para}. From what we have discussed in Section \ref{sec_diffusion}, we have for every $\eta\in C^{\infty}_{c}(\mathbb{R})$
\begin{equation*}
	\omega_{\L^{\tau}}^{\eta}(J;\delta)\leq \omega_{\L}^{\eta}(J;C\delta)\lesssim_{\eta}\Big(\frac{\delta}{J^{2}}\Big)^{\frac{1}{\theta_{1}-1}},
\end{equation*}
and moreover, since $\L_{\xi}^{\tau}=\L_{\xi}$, it follows that 
\begin{equation*}
	\sup_{u\in\mathbb{R},n\in\mathbb{Z},\abs{n}\sim J}\sup_{\xi\in\text{supp}(\eta)}\frac{\abs{\L^{\tau}_{\xi}}(iu,in,\xi)}{1+\abs{\xi}^{\theta_{1}-2}}=\sup_{u\in\mathbb{R},n\in\mathbb{Z},\abs{n}\sim J}\sup_{\xi\in\text{supp}(\eta)}\frac{\abs{\L_{\xi}}(iu,in,\xi)}{1+\abs{\xi}^{\theta_{1}-2}}\lesssim_{\eta}J^{2}.
\end{equation*}
Consequently, we obtain that $\L^{\tau}$ satisfies the nondegeneracy condition uniformly in $\tau\in(0,1)$.

 Now, we choose $\varphi(t,x)=\psi(x)\phi(x)$, where $\psi\in C^{\infty}_{c}(\O)$, $\phi=\phi^{\lambda}\in C^{1}([0,\infty))$ such that $0\leq \phi\leq 1$, $\phi\equiv1$ on $[0,T-\lambda]$, $\phi\equiv0$ on $[T,\infty)$ and $\abs{\partial_{t}\phi}\leq \lambda^{-1}$ for some $\lambda\in(0,1)$. This localization allows us to extend $\varphi\chi^{\tau}$ periodically in the space variable $x$ with a period $\Pi\supset \text{supp}(\psi)$ and then apply the averaging lemma in \cite{gess18}. The only difference from the kinetic equation treated in \cite{gess18} is 
the appearance of the extra term 
\begin{equation*}
	(b(\xi)+\tau)\big(\chi^{\tau}\Delta\varphi+2\nabla\chi^{\tau}\cdot\nabla\varphi\big)=	(b(\xi)+\tau)\big(\chi^{\tau}\Delta\psi+2\nabla\chi^{\tau}\cdot\nabla\psi\big)\phi,
\end{equation*}
which is created because of $\psi$, and this can be treated as the term $\chi\partial_{t}\phi$ in their argument. Therefore, by applying the averaging lemma (see \cite[Theorem 3.1]{gess18}), we obtain the following localized version of the averaging result.
 \begin{lem}
 	\label{averaging}
 	 Assume Hypotheses \ref{Hypothesis1}, \ref{Hypothesis2} and \ref{Hypothesis3}\ref{deg_condition}. Let $u^{\tau}$ be a weak solution for equation \eqref{kin_deg_approx_eq}, and
 	\begin{equation}\label{regularity}
 		s<\frac{1}{3(\theta_{1}^{2}-1)},\ \ \ \ \text{and}\ \ \ \ r<\frac{4\theta_{2}-3}{2\theta_{1}-1}.
 	\end{equation}
	Then for every $\psi\in C^{\infty}_{c}(\O)$, 
	\begin{equation}
		\norm{\psi u^{\tau}}_{L^{r}(\Omega\times[0,T];W^{s,r}(\O))}\leq C_{\psi}\Big(1+\norm{u_{0}}_{L^{2\theta_{1}-1}(\O)}^{2\theta_{1}-1}\Big),
	\end{equation}
	uniformly in $\tau\in(0,1)$. In particular, for any $V\subset\subset \O$
	\begin{equation}
		\norm{u^{\tau}}_{L^{r}(\Omega\times[0,T];W^{s,r}(V))}\leq C_{V}\Big(1+\norm{u_{0}}_{L^{2\theta_{1}-1}(\O)}^{2\theta_{1}-1}\Big),
	\end{equation}
	uniformly in $\tau\in(0,1)$.
\end{lem}

In what follows, we fix some $\lambda\in (0,1)$ and $r\in(1,2)$ satisfying condition \eqref{regularity}. Then, we study the tightness for the family of probability measures induced by $u^{\tau}$.

\begin{prop}
	Assume Hypotheses \ref{Hypothesis1}, \ref{Hypothesis2} and \ref{Hypothesis3}\ref{deg_condition}. Then the family of probability measures $\{\L(u^{\tau})\}_{\tau\in(0,1)}$ is tight in 
	\begin{equation*}
		L^{r}([0,T];L^{r}(\O))\bigcap L^{\infty}([0,T];W^{-\delta,r}(\O)),
	\end{equation*}
	 for every $\delta>0$.
\end{prop}

\begin{proof}
	Thanks to the uniform bound
	\begin{equation}\label{uniform_bound2}
		\sup_{\tau\in(0,1)}\ \E\sup_{t\in[0,T]}\norm{u^{\tau}(t)}_{L^{2\theta_{2}}(\O)}^{2\theta_{2}}\leq C_{T}\Big(1+\norm{u_{0}}_{L^{2\theta_{2}}(\O)}^{2\theta_{2}}\Big)
	\end{equation}
	by using the same argument as in the proof of Theorem \ref{weak_nondeg}, we have for every $\lambda\in(0,1/2)$
	\begin{equation}\label{uniform_bound3}
		\sup_{\tau\in(0,1)}\E\norm{u^{\tau}}_{C^{\lambda}([0,T];H^{-2})}<\infty.
	\end{equation}
	Let $\lambda\in(0,1/2)$ and let $\O_{1}\subset\subset\O_{2}\subset\subset\dots\subset \O$ be a sequence of non-empty smooth open sets such that $\cup_{i=1}^{\infty}\O_{i}=\O$. Given $R>0$, as in \cite[Proposition 5.3]{frid22}, we consider the set
	\begin{equation*}
		\begin{array}{ll}
			&\ds{K_{R}=\Big\{ u\in L^{\infty}([0,T];L^{2\theta_{2}}(\O))\bigcap L^{r}([0,T];W^{s,r}_{\text{loc}}(\O))\bigcap C^{\lambda}([0,T];W^{-2,r}(\O)): }\\
			\vs
			&\ds{\quad\quad\quad\quad\quad\quad\quad\quad\quad\quad \norm{u}_{L^{\infty}([0,T];L^{2\theta}(\O))}\leq R,\ \norm{u}_{C^{\lambda}([0,T];W^{-2,r}(\O))}\leq R, }\\
			\vs
			&\ds{\quad\quad\quad\quad\quad\quad\quad\quad\quad\quad \text{and}\ \norm{u}_{L^{r}([0,T];W^{s,r}(\O_{i}))}\leq 2^{i}(C_{\O_{i}}+1)R,\ \ i\in\mathbb{N}\Big\} },
		\end{array}
	\end{equation*}
	We will prove that $K_{R}$ is a relatively compact subset of
	\begin{equation*}
		L^{r}([0,T];L^{r}(\O))\bigcap L^{\infty}([0,T];W^{-\delta,r}(\O)),
	\end{equation*} 
	 for every $p>1$ and $\delta>0$. Indeed, since $K_{R}$ is bounded in $L^{\infty}([0,T];L^{r}(\O))$ and $C^{\lambda}([0,T];W^{-2,r}(\O))$, due to the compact embedding $L^{r}(\O)\hookrightarrow W^{-2,r}(\O)$, by \cite[Theorem 7]{simon1986} we have $K_{R}$ is relatively compact in $L^{\infty}([0,T];W^{-\delta,r}(\O))$, for every $\delta>0$. Moreover, for any $p>r$, let us define 
	\begin{equation*}
		\delta_{p}=\frac{sr}{p-r},\ \ \ \ \theta_{p}=\frac{s}{s+\delta_{p}}=\frac{p-r}{p}.
	\end{equation*}
	Then, if $O_{0}$ is any smooth subset of $\O$ with $\overline{\O_{0}}\subset \O$, by interpolation we have 
	\begin{equation*}
		\norm{u}_{L^{r}(\O_{0})}\leq \norm{u}_{W^{-\delta_{p},r}(\O_{0})}^{\theta_{p}}\norm{u}_{W^{s,r}(\O_{0})}^{1-\theta_{p}}.
	\end{equation*}    
	Thus, since $K_{R}$ is bounded in $W^{s,r}(\O_{0})$ and 
	\begin{equation*}
		\frac{\theta_{p}}{\infty}+\frac{1-\theta_{p}}{r}=\frac{1}{p},
	\end{equation*}
	by \cite[Theorem 7]{simon1986} we have $K_{R}$ is relatively compact in $L^{p}([0,T];L^{r}(\O_{0}))$. Applying this argument repeatedly for $\O_{1}$, $\O_{2}$, etc. in the place of $\O_{0}$, by a diagonal argument we can conclude that $K_{R}$ is relatively compact in $L^{p}([0,T];L^{r}_{\text{loc}}(\O))$. This implies, in particular, that if $\{f_{n}\}_{n\in\mathbb{N}}$ is any sequence in $K_{R}$, then we can find a subsequence that converges almost everywhere on $[0,T]\times \O$. Then, due to the uniform boundedness of $\{f_{n}\}_{n\in\mathbb{N}}$ in $L^{\infty}([0,T];L^{2\theta_{2}}(\O))$ and the fact that $2\theta_{2}>r$, by the Vitali convergence theorem it strongly converges in $L^{r}([0,T];L^{r}(\O))$. Hence, we conclude that $K_{R}$ is relatively compact in $L^{r}([0,T];L^{r}(\O))$.
	
	Now, if we denote $\mu_{u^{\tau}}:=\L(u^{\tau})$, $\tau\in(0,1)$, then thanks to Lemma \ref{averaging}, \eqref{uniform_bound2} and \eqref{uniform_bound3}, we obtain that there is a constant $C>0$ such that for any $\tau\in(0,1)$ 
	\begin{equation*}
		\begin{array}{ll}
			&\ds{\mu_{u^{\tau}}(K_{R}^{c})\leq \P\Big(\norm{u^{\tau}}_{L^{\infty}([0,T];L^{2\theta_{2}}(\O))}>R\Big)+\P\Big(\norm{u^{\tau}}_{C^{\lambda}([0,T];W^{-2,r}(\O))}>R\Big) }\\
			\vs
			&\ds{\quad\quad\quad\quad\quad +\sum_{i=1}^{\infty}\P\Big(\norm{u^{\tau}}_{L^{r}([0,T];W^{s,r}(\O_{i}))}>2^{n}(C_{\O_{i}}+1)R\Big)  }\\
			\vs
			&\ds{\quad\quad\leq \frac{1}{R}\Big(\E\norm{u^{\tau}}_{L^{\infty}([0,T];L^{2\theta_{2}}(\O))}+\E\norm{u^{\tau}}_{C^{\lambda}([0,T];W^{-2,r}(\O))}\Big) }\\
			\vs
			&\ds{\quad\quad\quad\quad\quad +\frac{1}{R}\sum_{i=1}^{\infty}\frac{1}{2^{i}(C_{\O_{i}}+1)}\E\norm{u^{\tau}}_{L^{r}([0,T];W^{s,r}(\O_{i}))}\leq \frac{C}{R} },
		\end{array}
	\end{equation*}
	and since $R>0$ is arbitrary, we complete the proof.
	
\end{proof}

\subsubsection{Proof of Theorem \ref{kin_deg}}

According to Corollary \ref{uniqueness_kinetic}, we know there is at most one kinetic solution for equation \eqref{limit_para} in $L^{2\theta_{2}}(\Omega;L^{\infty}([0,T];L^{2\theta_{2}}(\O)))$. As we have mentioned in the proof of Theorem \ref{weak_nondeg}, in order to prove the existence of a kinetic solution, it is sufficient to prove that there exists a probabilistically kinetic solution for equation \eqref{limit_para}
\begin{equation*}
	\big(\tilde{\Omega},\tilde{\F},\{\tilde{\F}\}_{t},\tilde{\P},\tilde{w},\tilde{u},\tilde{m}\big)
\end{equation*}  
such that $\tilde{u}\in L^{2\theta_{2}}(\tilde{\Omega};L^{\infty}([0,T];L^{2\theta_{2}}(\O)))$.

To this purpose, let $r\in(1,2)$ be any constant satisfying condition \eqref{regularity} and we define the space
\begin{equation*}
	\Xi_{\delta}(T)=\big(L^{r}([0,T];L^{r}(\O))\bigcap L^{\infty}([0,T];W^{-\delta,r}(\O))\big)\times C([0,T];\mathcal{U}_{0}),\ \ \ \ \delta>0,
\end{equation*}
where $\mathcal{U}_{0}$ is any Hilbert space such that the embedding $\mathcal{U}\hookrightarrow\mathcal{U}_{0}$ is Hilbert-Schmidt.

For every $\tau\in(0,1)$, let $u^{\tau}$ be the unique weak solution of equation \eqref{kin_deg_approx_eq}. Then for $\delta>0$, due to the tightness of $\{\L(u^{\tau},w)\}_{\tau\in(0,1)}$ in $\Xi_{\delta}(T)$, there exists a sequence $\tau_{k}\downarrow0$ such that $\L(u^{\tau_{k}},w)$ is weakly convergent in $\Xi_{\delta}(T)$. If we denote by $u$ any weak limit point for $\{u^{\tau_{k}}\}_{k\in\mathbb{N}}$, then by the Skorokhod Theorem there exists random variables $\mathcal{Z}_{k}=(\tilde{u}_{k},\tilde{w}_{k})$,  $\mathcal{Z}=(\tilde{u},\tilde{w})$, and a probability space $\big(\tilde{\Omega},\tilde{\F},\{\tilde{\F}\}_{t},\tilde{\P}\big)$ such that 
	\begin{equation*}
		\L(\mathcal{Z})=\L(u,w),\ \ \ \ \L(\mathcal{Z}_{n})=\L(u^{\tau_{k}},w),\ \ \ k\in\mathbb{N},
	\end{equation*}
	and 
	\begin{equation}\label{kin_almost_surely}
		\lim_{k\to\infty}\Big(\norm{\tilde{u}_{k}-\tilde{u}}_{L^{r}([0,T];L^{r}(\O))}+\norm{\tilde{u}_{k}-\tilde{u}}_{L^{\infty}([0,T];W^{-\delta,r}(\O))}+\norm{\tilde{w}_{k}-\tilde{w}}_{C([0,T];\mathcal{U}_{0})}\Big)=0,\ \ \ \tilde{\P}\text{-a.s.}.
	\end{equation}
Moreover, due to \eqref{kin_energy_Lp_approx}, we have 
\begin{equation}\label{kin_energy_Lp_approx_martingale}
	\begin{array}{ll}
		&\ds{\sup_{k\in\mathbb{N}}\Bigg(\tilde{\E}\sup_{t\in[0,T]}\norm{\tilde{u}_{k}(t)}_{L^{2\theta}(\O)}^{2\theta_{2}}+\tilde{\E}\int_{0}^{T}\int_{\O}\abs{\tilde{u}_{k}(s,x)}^{2(\theta_{2}-1)}b(\tilde{u}_{k}(s,x))\abs{\nabla \tilde{u}_{k}(s,x)}^{2}dxds}\\
		\vs
		&\ds{\quad\quad\quad\quad\quad +\tau_{k}\ \tilde{\E}\int_{0}^{T}\int_{\O}\abs{\tilde{u}_{k}(s,x)}^{2(\theta_{2}-1)}\abs{\nabla \tilde{u}_{k}(s,x)}^{2}dxds\Bigg)\leq C_{T}\Big(1+\norm{u_{0}}_{L^{2\theta_{2}}(\O)}^{2\theta_{2}}\Big)},
	\end{array}
\end{equation}
together with \eqref{kin_almost_surely}, this implies that
	\begin{equation}
	\tilde{\E}\ \esssup_{t\in[0,T]}\norm{\tilde{u}(t)}_{L^{2\theta_{2}}(\O)}^{2\theta_{2}}\leq C_{T}\Big(1+\norm{u_{0}}_{L^{2\theta_{2}}(\O)}^{2\theta_{2}}\Big),
\end{equation}
and we may extract a subsequence (not relabeled) of $\{\tilde{u}_{k}\}_{k\in\mathbb{N}}$ such that
	\begin{equation}\label{convergence_Lp_kin}
		\lim_{k\to\infty}\ \tilde{u}_{k}=\tilde{u},\ \ \ \text{in}\ \ \ L^{2}(\tilde{\Omega}\times[0,T];H).
	\end{equation}
Next, if we denote $\tilde{h}_{k}:=\mathbf{1}_{\tilde{u}_{n}>\xi}$, then the pair $(\tilde{h}_{k},\tilde{m}_{k})$ satisfies the kinetic formulation:
	\begin{equation}\label{kin_formulation_approx}
		d\tilde{h}_{k}-(b(\xi)+\tau_{k})\Delta\tilde{h}_{k}dt=\sum_{i}^{\infty}\delta_{\tilde{u}_{k}=\xi}\sigma_{i}(x,\xi)d\beta_{i}(t)+\partial_{\xi}\Big(\tilde{m}_{k}-\frac{1}{2}\delta_{\tilde{u}_{k}=\xi}\Sigma^{2}(x,\xi)\Big),
	\end{equation}
	where $\tilde{m}_{k}=\tilde{n}_{1,k}+\tilde{n}_{2,k}$, 
	\begin{equation*}
		d\tilde{n}_{1,k}(t,x,\xi)=b(\xi)\abs{\nabla\tilde{u}_{k}}^{2}d\delta_{\tilde{u}_{k}=\xi}dxdt,\ \ \ \ 		d\tilde{n}_{2,k}(t,x,\xi)=\tau_{k}\abs{\nabla\tilde{u}_{k}}^{2}d\delta_{\tilde{u}_{k}=\xi}dxdt.
	\end{equation*}
	Note that, in view of \eqref{kin_energy_Lp_approx} with $p=q=2$, we have
	\begin{equation*}
		\sup_{k\in\mathbb{N}}\E\big\lvert \tilde{m}_{k}([0,T]\times\O\times\mathbb{R})\big\rvert^{2}= \sup_{k\in\mathbb{N}}\E\Big\lvert \int_{0}^{T}\Inner{(b(\tilde{u}_{k}(t))+\tau_{k})\nabla \tilde{u}_{k}(t),\nabla\tilde{u}_{k}(t)}_{H}dt\Big\rvert^{2}\leq C
	\end{equation*}
	Hence, the sequence $\{\tilde{m}_{k}\}_{k\in\mathbb{N}}$ is bounded in $L^{2}_{w}(\hat{\Omega};\mathcal{M}_{b}([0,T]\times\O\times\mathbb{R}))$, and according to the Banach-Alaoglu theorem, it possesses a weak$^{\ast}$ convergent subsequence $\{\tilde{m}_{k}\}_{k\in\mathbb{N}}$ (not relabeled). Then by the same approach as in \cite[Proof of Theorem 6.4]{debussche16}, its weak$^{\ast}$ limit $\tilde{m}$ is a kinetic measure (see Definition \ref{def_kinetic_measure}), and moreover, thanks to \eqref{kin_energy_Lp_approx_martingale}, we can obtain that the pair $(\tilde{h}:=\mathbf{1}_{\tilde{u}>\xi},\tilde{m})$ satisfies the kinetic formulation for equation \eqref{limit_para} by passing the limit as $k\to\infty$ in \eqref{kin_formulation_approx}. Finally, if we show that 
	\begin{equation}\label{goal_kin}
		\text{div}\int_{0}^{\tilde{u}_{k}}\sqrt{b(\xi)}d\xi\rightharpoonup \text{div}\int_{0}^{\tilde{u}}\sqrt{b(\xi)}d\xi\ \ \ \text{in}\ \ \ L^{2}(\tilde{\Omega}\times[0,T]\times\O).
	\end{equation}
	 then by proceeding as in \cite[Proof of Theorem 6.4]{debussche16}, we may obtain that the chain rule formula \eqref{chain_rule_kin} holds and $\tilde{m}\geq \tilde{n}_{1}$, $\P$-a.s., where $\tilde{n}_{1}$ is defined as in \eqref{kinetic_measure}, with $u$ replaced by $\tilde{u}$. And this implies that  
	\begin{equation*}
		\big(\tilde{\Omega},\tilde{\F},\{\tilde{\F}\}_{t},\tilde{\P},\tilde{w},\tilde{u},\tilde{m}\big)
	\end{equation*}
	is a kinetic solution for equation \eqref{limit_para}. Concerning \eqref{goal_kin}, due to \eqref{kin_energy_Lp_approx_martingale} we have 
	\begin{equation*}
		\sup_{k\in\mathbb{N}}\tilde{\E}\int_{0}^{T}\int_{\O}\Big\lvert \text{div}\int_{0}^{\tilde{u}_{k}}\sqrt{b(\xi)}d\xi\Big\rvert^{2}dxdt=\sup_{k\in\mathbb{N}}\tilde{\E}\int_{0}^{T}\int_{\O}b(\tilde{u}_{k}(t,x))\abs{\nabla\tilde{u}_{k}(t,x)}^{2}dxdt<\infty,
	\end{equation*}
	and for all $\varphi\in C([0,T];C^{1}_{c}(\O))$ and $\psi\in L^{\infty}(\tilde{\Omega})$
	\begin{equation*}
		\begin{array}{ll}
			&\ds{\Bigg\lvert \tilde{\E}\psi\int_{0}^{T}\int_{\mathcal{O}}\Bigg[\text{div}\int_{0}^{\tilde{u}_{k}}\sqrt{b(\xi)}d\xi-\text{div}\int_{0}^{\tilde{u}}\sqrt{b(\xi)}d\xi\Bigg]\varphi(t,x)dxdt\Bigg\rvert }\\
			\vs
			&\ds{\leq C\norm{\psi}_{L^{\infty}(\Omega)}\norm{\varphi}_{C([0,T];C^{1}_{c}(\mathcal{O}))}\cdot\tilde{\E}\int_{0}^{T}\int_{\O}\abs{\tilde{u}_{k}(t,x)-\tilde{u}(t,x)}\Big(1+\abs{\tilde{u}_{k}(t,x)}^{\theta_{2}-1}+\abs{\tilde{u}(t,x)}^{\theta_{2}-1}\Big)dxdt }\\
			\vs
			&\ds{\leq C_{T}\norm{\psi}_{L^{\infty}(\Omega)}\norm{\varphi}_{C([0,T];C^{1}_{c}(\mathcal{O}))}\cdot \tilde{\E}\norm{\tilde{u}_{k}-\tilde{u}}_{L^{2}([0,T];H)}\Big[1+\tilde{\E}\esssup_{t\in[0,T]}\Big(\norm{\tilde{u}(t)}_{L^{2(\theta_{2}-1)}}^{\theta_{2}-1}+\norm{\tilde{u}_{k}(t)}_{L^{2(\theta_{2}-1)}}^{\theta_{2}-1}\Big)\Big] }\\
			\vs
			&\ds{\to0},
		\end{array}
	\end{equation*}
	which implies that \eqref{goal_kin} follows. Due to what we have discussed above, we complete the proof of the existence for a kineric solution $u$ of equation \eqref{limit_para}. Finally, from Lemma \ref{continuity_Lp} we already know that $u$ has $\P$-almost surely continuous trajectories in $L^{2\theta_{2}}(\O)$.

\section{Well-posedness of equation \eqref{limit_para} in $L^{1}(\O)$}

\label{sec_proof_main_thm}

In this section we will establish the well-posedness for \eqref{limit_para} in the $L^{1}$ setting; that is, when the initial data $u_{0}\in L^{1}(\O)$. In order to prove Theorem \ref{main_thm}, we need some preliminary results about generalized kinectic solutions of equation \eqref{limit_para}.

We start with a estimate which is a modification of \cite[Lemma 4.5]{gess18}.
\begin{lem}\label{kinetic_measure_bound}
	Assume Hypotheses \ref{Hypothesis1} and \ref{Hypothesis2}. Let $u_{0}\in L^{1}(\mathcal{O})$ and $u$ be a generalized kinetic solution of equation \eqref{limit_para} with the corresponding generalized kinetic measure $m$. Then for every $k\in\mathbb{N}$ and $p\geq1$,
	\begin{equation}
		\E\Big\lvert m\big([0,T]\times\mathcal{O}\times[-k,k]\big)\Big\rvert^{p}\leq C\Big(p,k,T,\norm{u_{0}}_{L^{1}(\mathcal{O})}^{p}\Big).
	\end{equation}
\end{lem}

\begin{proof}
	 If we denote $\chi:=\mathbf{1}_{u>\xi}-\mathbf{1}_{0>\xi}$, then the following equation
	\begin{equation}\label{transform}
		\partial_{t}\chi-b(\xi)\Delta\chi=\partial_{\xi}q+\sum_{i=1}^{\infty}\sigma_{i}(x,\xi)\delta_{u=\xi}\dot{\beta}_{k}(t)
	\end{equation}
	holds in the sense of distributions over $[0,T)\times\mathcal{O}\times\mathbb{R}$, where 
	\begin{equation*}
		q(t,x,\xi):=m(t,x,\xi)-\frac{1}{2}\delta_{u(t,x)=\xi}\Sigma^{2}(x,\xi).
	\end{equation*}
	In what follows, for $k\in\mathbb{N}$ we set
	\begin{equation*}
		\psi_{k}(\xi)=\mathbf{1}_{[-k,k]}(\xi), \ \ \ \ \ \Psi_{k}(\xi)=\int_{-k}^{\xi}\int_{-k}^{r}\psi_{k}(s)dsdr.
	\end{equation*}
	Let $\eta\in C^{1}_{c}([0,T))$ be nonnegative such that $\eta(0)=1$, $\eta'\leq 0$. First, for every $0\leq \phi\in C^{\infty}_{c}(\mathcal{O})$, after a preliminary step of regularization we may take $\varphi(t,x,\xi)=\eta(t)\phi(x)\Psi_{k}'(\xi)$ as test functions in \eqref{transform}, and we obtain 
	\begin{equation}\label{bound_test}
		\begin{array}{ll}
			&\ds{\int_{0}^{T}\int_{\mathcal{O}}\big[\Psi_{k}(u(t,x))-\Psi_{k}(0)\big]\phi(x)\abs{\eta'(t)}dxdt+\int_{[0,T]\times\mathcal{O}\times\mathbb{R}}\eta(t)\phi(x)\psi_{k}(\xi)dm(t,x,\xi) }\\
			\vs
			&\ds{=\int_{\mathcal{O}}\big[\Psi_{k}(u_{0}(x))-\Psi_{k}(0)\big]\phi(x)\eta(0)dx+\int_{0}^{T}\int_{\mathcal{O}}\int_{\mathbb{R}}b(\xi)\eta(t)\Psi_{k}'(\xi)\chi(t,x,\xi)\nabla^{2}\phi(x)d\xi dxdt }\\
			\vs
			&\ds{\quad\quad\quad\quad+ \frac{1}{2}\int_{0}^{T}\int_{\mathcal{O}}\phi(x)\eta(t)\psi_{k}(u(t,x))\Sigma^{2}(x,u(t,x))dxdt }\\
			\vs
			&\ds{\quad\quad\quad\quad\quad\quad\quad\quad +\sum_{i=1}^{\infty}\int_{0}^{T}\int_{\mathcal{O}}\phi(x)\eta(t)\Psi_{k}'(u(t,x))\sigma_{i}(x,u(t,x))dxd\beta_{i}(t)}.
		\end{array}
	\end{equation}
	Next, for every $B\in\mathfrak{B}$ and $0\leq \phi\in C^{\infty}_{c}(B)$, by approximation we now take $\varphi(t,x,\xi)=\eta(t)\phi(x)\zeta_{\delta}(x)\Psi_{k}'(\xi)$ as test functions. Then note that, by using the same technique as in the proof of Corollary \ref{uniqueness_further}, it can be easily verified that
	\begin{equation*}
		\lim_{\delta\to0}\int_{0}^{T}\int_{\mathcal{O}}\int_{\mathbb{R}}b(\xi)\eta(t)\Psi_{k}'(\xi)\chi(t,x,\xi)\nabla\phi(x)\cdot\nabla\zeta_{\delta}(x)d\xi dxdt=0,\ \ \ \mathbb{P}\text{-a.s.},
	\end{equation*}
	and recall that $\Delta\zeta_{\delta}\leq0$, if we taking the limit as $\delta\to0$, it follows that \eqref{bound_test} holds for $0\leq\phi\in C^{\infty}_{c}(B)$, with the symbol ``$=$" replaced by ``$\leq$". Therefore, by the standard argument of partition of unity as we applied in the proof of Proposition \ref{L1_contraction}, we have $\mathbb{P}$-a.s.
	\begin{equation*}
		\begin{array}{ll}
			&\ds{\int_{0}^{T}\int_{\mathcal{O}}\big[\Psi_{k}(u(t,x))-\Psi_{k}(0)\big]\abs{\eta'(t)}dxdt+\int_{[0,T]\times\mathcal{O}\times\mathbb{R}}\eta(t)\psi_{k}(\xi)dm(t,x,\xi) }\\
			\vs
			&\ds{\leq \int_{\mathcal{O}}\big[\Psi_{k}(u_{0}(x))-\Psi_{k}(0)\big]dx + \frac{1}{2}\int_{0}^{T}\int_{\mathcal{O}}\eta(t)\psi_{k}(u(t,x))\Sigma^{2}(x,u(t,x))dxdt }\\
			\vs
			&\ds{\quad\quad\quad +\sum_{i=1}^{\infty}\int_{0}^{T}\int_{\mathcal{O}}\eta(t)\Psi_{k}'(u(t,x))\sigma_{i}(x,u(t,x))dxd\beta_{i}(t) },
		\end{array}
	\end{equation*}
	and thus for every $p\geq1$,
	\begin{equation*}
		\begin{array}{ll}
			&\ds{\E\Bigg\lvert \int_{0}^{T}\int_{\mathcal{O}}\Psi_{k}(u(t,x))\abs{\eta'(t)}dxdt\Bigg\rvert^{p} +\E\Bigg\lvert\int_{[0,T]\times\mathcal{O}\times[-k,k]}\eta(t)dm(t,x,\xi)\Bigg\rvert^{p} }\\
			\vs
			&\ds{\lesssim \E\Big\lvert \int_{0}^{T}\int_{\mathcal{O}}\Psi_{k}(0)\abs{\eta'(t)}dxdt\Big\rvert^{p} +\E\Big\lvert\int_{\mathcal{O}}\big[\Psi_{k}(u_{0}(x))-\Psi_{k}(0)\big]dx \Big\rvert^{p}}\\
			\vs
			&\ds{ +\E\Bigg\lvert\int_{0}^{T}\int_{\mathcal{O}}\eta(t)\psi_{k}(u(t,x))\Sigma^{2}(x,u(t,x))dxdt\Bigg\rvert^{p} +\E\Bigg\lvert\sum_{i=1}^{\infty}\int_{0}^{T}\int_{\mathcal{O}}\eta(t)\Psi_{k}'(u(t,x))\sigma_{i}(x,u(t,x))dxd\beta_{i}(t)\Bigg\rvert^{p} }
		\end{array}
	\end{equation*}
	Finally, letting $\eta\to\mathbf{1}_{[0,t]}$, and thanks to  \eqref{sgfine2} and the fact that
	\begin{equation*}
		0\leq \Psi_{k}(u)\leq 2k(k+\abs{u}),\ \ \ \ 0\leq \Psi_{k}'(u)\leq 2k\mathbf{1}_{u>-k},\ \ \ u\in\mathbb{R},
	\end{equation*}
	by proceeding as in the proof of \cite[Lemma 4.5]{gess18}, we finish the proof.
\end{proof}

\bigskip

\begin{lem}\label{kinetic_measure_bound_2}
	Assume Hypotheses \ref{Hypothesis1} and \ref{Hypothesis2}. Let $u_{0}\in L^{1}(\mathcal{O})$ and $u$ be a generalized kinetic solution of equation \eqref{limit_para} with the corresponding generalized kinetic measure $m$. Then for every nonnegative $\psi\in C^{\infty}_{c}(\mathbb{R})$,
	\begin{equation}
		\E \int_{[0,T]\times\mathcal{O}\times\mathbb{R}}\psi(\xi)dm(t,x,\xi)\leq C\big(T,\text{supp}(\psi),\norm{\psi}_{\infty},  \norm{u_{0}}_{L^{1}(\mathcal{O})}\big).
	\end{equation}
\end{lem}

\begin{proof}
	Let $0\leq \psi\in C^{\infty}_{c}(\mathbb{R})$, and let $\eta\in C^{1}_{c}([0,T))$ be nonnegative such that $\eta(0)=1$, $\eta'\leq 0$. Take $\Psi\in C^{\infty}(\mathbb{R})$ such that $\Psi\geq0$ and $\Psi''=\psi$, then by proceeding as in the proof of Lemma \ref{kinetic_measure_bound}, we have
	\begin{equation*}
		\begin{array}{ll}
			&\ds{\E\int_{0}^{T}\int_{\mathcal{O}}\Psi(u(t,x))\abs{\eta'(t)}dxdt+\E\int_{[0,T]\times\mathcal{O}\times\mathbb{R}}\eta(t)\psi(\xi)dm(t,x,\xi) }\\
			\vs
			&\ds{\leq \E\int_{0}^{T}\int_{\mathcal{O}}\Psi(0)\abs{\eta'(t)}dxdt+ \E\int_{\mathcal{O}}\big[\Psi(u_{0}(x))-\Psi(0)\big]\eta(0)dx }\\
			\vs
			&\ds{\quad\quad\quad\quad+ \frac{1}{2}\E\int_{0}^{T}\int_{\mathcal{O}}\eta(t)\psi(u(t,x))\Sigma^{2}(x,u(t,x))dxdt },
		\end{array}
	\end{equation*}
	Since for every $r\in\mathbb{R}$, 
	\begin{equation*}
		\abs{\Psi'(r)}\leq \text{supp}\big(\psi)\cdot\norm{\psi}_{\infty},\ \ \ \ \abs{\Psi(r)}\leq c\ \text{supp}\big(\psi)\cdot\norm{\psi}_{\infty}(1+\abs{r}),
	\end{equation*}
	letting now $\eta\to\mathbf{1}_{[0,t]}$, it follows from \eqref{sgfine2} that 
	\begin{equation*}
		\E \int_{[0,T]\times\mathcal{O}\times\mathbb{R}}\psi(\xi)dm(t,x,\xi)\leq C\big(T,\text{supp}(\psi),\norm{\psi}_{\infty},  \norm{u_{0}}_{L^{1}(\mathcal{O})}\big).
	\end{equation*}
\end{proof}

By using the same technique as in the proof of Lemma \ref{kinetic_measure_bound} and following along the lines of the proof of \cite[Proposition 4.7]{gess18}, we obtain the decay of the generalized kinetic measure.
\begin{prop}\label{decay}
	Assume Hypotheses \ref{Hypothesis1} and \ref{Hypothesis2}. Let $u_{0}\in L^{1}(\mathcal{O})$ and $u$ be a generalized kinetic solution of equation \eqref{limit_para} with the corresponding generalized kinetic measure $m$. Then 
	\begin{equation}\label{decay1}
		\begin{array}{ll}
			&\ds{\esssup_{t\in[0,T]}\E\norm{\big(u(t)-2^{n}\big)^{+}}_{L^{1}(\mathcal{O})}+\esssup_{t\in[0,T]}\E\norm{\big(u(t)+2^{n}\big)^{-}}_{L^{1}(\mathcal{O})}+\frac{1}{2^{n}}\E m(A_{2^{n}})}\\
			\vs
			&\ds{\leq C\Big(T,\norm{u_{0}}_{L^{1}(\mathcal{O})}\Big)\delta(n) },\ \ \ \forall n\in\mathbb{N},
		\end{array}
	\end{equation}
	where 
	\begin{equation*}
		A_{2^{n}}=[0,T]\times\mathcal{O}\times\big\{\xi\in\mathbb{R}:2^{n}\leq \abs{\xi}\leq 2^{n+1}\big\}
	\end{equation*}
	and $\delta(n)$ depends only on the functions
	\begin{equation*}
		R\mapsto \norm{(u_{0}-R)^{+}}_{L^{1}(\mathcal{O})},\ \ \ \ R\mapsto \norm{(u_{0}+R)^{-}}_{L^{1}(\mathcal{O})},
	\end{equation*}
	and satisfies $\lim_{n\to\infty}\delta(n)=0$.\\
	
	Furthermore,
	\begin{equation}\label{decay2}
		\begin{array}{ll}
			&\ds{\E\esssup_{t\in[0,T]}\norm{\big(u(t)-2^{n}\big)^{+}}_{L^{1}(\mathcal{O})}+\E\esssup_{t\in[0,T]}\norm{\big(u(t)+2^{n}\big)^{-}}_{L^{1}(\mathcal{O})}}\\
			\vs
			&\ds{\leq C\Big(T,\norm{u_{0}}_{L^{1}(\mathcal{O})}\Big)\Big(\delta(n)+\delta^{\frac{1}{2}}(n)\Big) },
		\end{array}
	\end{equation}
	where $\delta(n)$ is as above, in addition possibly depending on the functions
	\begin{equation*}
		R\mapsto\E\int_{0}^{T}\int_{\mathcal{O}}\mathbf{1}_{\abs{u(t,x)}\geq R}\ dxdt.
	\end{equation*}
\end{prop}

By proceeding as in the proof of \cite[Corollary 4.8]{gess18}, we can deduce from \eqref{decay2} that generalized kinetic solutions have almost surely continuous trajectories in $L^{1}(\O)$.
\begin{cor}[Continuity in time]
	\label{continuity_L1}
	Assume Hypotheses \ref{Hypothesis1} and \ref{Hypothesis2}. Let $u_{0}\in L^{1}(\mathcal{O})$ and $u$ be a generalized kinetic solution of equation \eqref{limit_para}. Then there exists a representative of $u$ with $\P$-almost surely continuous trajectories in $L^{1}(\O)$. 
\end{cor}

Now we have all in hand to prove Theorem \ref{main_thm}.

\subsection{Proof of Theorem \ref{main_thm}}

First, we remark that the $L^{1}$-contraction property \eqref{L1_contraction_kinetic} of generalized kinetic solutions follows directly from Proposition \ref{L1_contraction} and Corollay \ref{continuity_L1}. This, implies that the uniqueness of generalized kinetic solution to equation \eqref{limit_para}. Hence, it only remains to prove the existence of a generalized kinetic solution to \eqref{limit_para}. Here we will only provide the proof in the case when Hypothesis \ref{Hypothesis3}\ref{deg_condition} holds, as the same argument can be applied under Hypothesis \ref{Hypothesis3}\ref{nondeg_condition}.

Fix $u_{0}\in L^{1}(\mathcal{O})$. Let $\{u_{0,\kappa}\}_{\kappa>0}\subset L^{2\theta_{2}}(\O)$ be a sequence satisfying that $u_{0,\kappa}\to u_{0}$ in $L^{1}(\mathcal{O})$ as $\kappa\to0$. Under Hypotheses \ref{Hypothesis1}, \ref{Hypothesis2} and \ref{Hypothesis3}\ref{deg_condition}, from Theorem \ref{kin_deg} we know that for every $\kappa>0$ there exists a unique kinetic solution $u_{\kappa}$ for equation \eqref{limit_para} with the initial data $u_{0,\kappa}$. Moreover, due to \eqref{L1_contraction_kinetic} we obtain that $\{u_{\kappa}\}$ is Cauchy in $L^{\infty}([0,T];L^{1}(\Omega;L^{1}(\O)))$, so there exists some $u\in L^{\infty}([0,T];L^{1}(\Omega;L^{1}(\mathcal{O})))$ such that
\begin{equation}\label{cauchy_L1}
	\lim_{\kappa\to0}\sup_{t\in[0,T]}\E\norm{u_{\kappa}(t)-u(t)}_{L^{1}(\mathcal{O})}=0.
\end{equation}
Thanks to \eqref{energy_kinetic_Lp} and the uniform bound of $\{u_{0,\kappa}\}$ in $L^{1}(\O)$, we have for every $p\geq1$
\begin{equation*}
	\sup_{\kappa>0}\ \E\esssup_{t\in[0,T]}\norm{u_{\kappa}(t)}_{L^{1}(\O)}^{p}<C_{T,p}\Big(1+\norm{u_{0}}_{L^{1}(\O)}^{p}\Big),
\end{equation*}
so by lower-semicontinuity of the norm we have 
\begin{equation}
	\E\esssup_{t\in[0,T]}\norm{u(t)}_{L^{1}(\O)}^{p}<C_{T,p}\Big(1+\norm{u_{0}}_{L^{1}(\O)}^{p}\Big).
\end{equation}
Moreover, for every $p>1$ and every $k>0$, by Lemma \ref{kinetic_measure_bound} we have 
\begin{equation}
	\sup_{\kappa}\E\big\lvert m^{\kappa}([0,T]\times\mathcal{O}\times[-k,k])\big\rvert^{p}\leq C(k,p),
\end{equation}
which implies the sequence $(m^{\kappa})$ is bounded in $L^{p}(\Omega;\mathcal{M}(B_{k}))$, where $B_{k}:=[0,T]\times\mathcal{O}\times[-k,k]$. By proceeding as in the proof of \cite[Theorem 4.9]{gess18}, we can extract a subsequence and a random Borel measure $m$ on $[0,T]\times\mathcal{O}\times\mathbb{R}$ such that $m^{\kappa}\rightharpoonup^{\ast}m$ weakly$^{\ast}$ in $L^{p}(\Omega;\mathcal{M}(B_{k}))$ for every $k\in\mathbb{N}$, and show that $m$ is a generalized kinetic measure (see Definitions \ref{def_kinetic_measure}). We shall now verify that $(u,m)$ is a generalized kinetic solution of equation \eqref{limit_para}. Following the same approach as in the proof of \cite[Theorem 4.9]{gess18}, we only need to verify that conditions \eqref{div_L2} and \eqref{chain_rule} in Definitions \ref{def_gen_kinetic_solution} hold.

Indeed, we first take arbitrary $0\leq \phi\in C^{\infty}_{c}(\mathbb{R})$, then by Lemma \ref{kinetic_measure_bound_2} we have
\begin{equation}\label{div_uniform_bound}
	\begin{array}{ll}
		&\ds{\sup_{\kappa>0}\ \E\int_{0}^{T}\int_{\mathcal{O}}\Big\lvert \text{div}\int_{0}^{u_{\kappa}}\sqrt{b(\xi)\phi(\xi)}d\xi\Big\rvert^{2}dxdt } \\
		\vs
		&\ds{= \sup_{\kappa>0}\ \E\int_{[0,T]\times\mathcal{O}\times\mathbb{R}}\phi(\xi)dm^{\kappa}(t,x,\xi) }\\
		\vs
		&\ds{\leq  C\big(T,\text{supp}(\phi),\norm{\phi}_{\infty},\norm{u_{0}}_{L^{1}(\mathcal{O})}\big) }.
	\end{array}
\end{equation}
Moreover, due to \eqref{cauchy_L1} and the fact that $\sqrt{\phi}\in C_{c}(\mathbb{R})$, we conclude using integration by parts that, for all $\varphi\in C([0,T];C^{1}_{c}(\mathcal{O}))$, $\psi\in L^{\infty}(\Omega)$,
\begin{equation*}
	\begin{array}{ll}
		&\ds{\Bigg\lvert \E\psi\int_{0}^{T}\int_{\mathcal{O}}\Bigg[\text{div}\int_{0}^{u_{\kappa}}\sqrt{b(\xi)\phi(\xi)}d\xi-\text{div}\int_{0}^{u}\sqrt{b(\xi)\phi(\xi)}d\xi\Bigg]\varphi(t,x)dxdt\Bigg\rvert }\\
		\vs
		&\ds{\leq CT\big(\theta,\phi\big)\norm{\psi}_{L^{\infty}(\Omega)}\norm{\varphi}_{C([0,T];C^{1}_{c}(\mathcal{O}))}\cdot\sup_{t\in[0,T]}\E\norm{u_{\kappa}(t)-u(t)}_{L^{1}(\O)}\to0 },
	\end{array}
\end{equation*}
which implies that 
\begin{equation}\label{weak_convergence_L1}
	\text{div}\int_{0}^{u_{\kappa}}\sqrt{b(\xi)\phi(\xi)}d\xi\rightharpoonup \text{div}\int_{0}^{u}\sqrt{b(\xi)\phi(\xi)}d\xi,\ \ \ \ \text{in}\ \ L^{2}(\Omega\times[0,T]\times\O).
\end{equation}
 Now, let $0\leq\phi\in C_{c}(\mathbb{R})$, we choose a sequence $(\phi_{\epsilon})_{\epsilon>0}$, $0\leq \phi_{\epsilon}\in C^{\infty}_{c}(\mathbb{R})$ with $\text{supp}(\phi_{\epsilon})\subseteq\text{supp}(\phi)$, such that $\phi_{\epsilon}\to\phi$ in $C(\mathbb{R})$ as $\epsilon\to0$. For every $\kappa>0$, it is easy to check that
\begin{equation*}
	\text{div}\int_{0}^{u_{\kappa}}\sqrt{b(\xi)\phi_{\epsilon}(\xi)}d\xi\rightharpoonup \text{div}\int_{0}^{u_{\kappa}}\sqrt{b(\xi)\phi(\xi)}d\xi,\ \ \ \ \text{in}\ \ L^{2}(\Omega\times[0,T]\times\O),
\end{equation*}  
and it follows from \eqref{div_uniform_bound} that 
\begin{equation*}
	\begin{array}{ll}
		&\ds{\E\int_{0}^{T}\int_{\mathcal{O}}\Big\lvert \text{div}\int_{0}^{u_{\kappa}}\sqrt{b(\xi)\phi(\xi)}d\xi\Big\rvert^{2}dxdt \leq  \sup_{\epsilon\in(0,1)}\ \E\int_{0}^{T}\int_{\mathcal{O}}\Big\lvert \text{div}\int_{0}^{u_{\kappa}}\sqrt{b(\xi)\phi_{\epsilon}(\xi)}d\xi\Big\rvert^{2}dxdt }\\
		\vs
		&\ds{\quad\quad\quad\quad\quad\quad\quad\quad\quad\quad\quad\quad\quad\quad\quad\quad\quad\quad\leq C\big(T,\text{supp}(\phi),\norm{\phi}_{\infty},\norm{u_{0}}_{L^{1}(\mathcal{O})}\big),\ \ \ \forall\tau>0 }.
	\end{array}
\end{equation*}
Then applying the same argument as before, we conclude that as $\kappa\to0$
\begin{equation}\label{weak_convergence_L2}
	\text{div}\int_{0}^{u_{\kappa}}\sqrt{b(\xi)\phi(\xi)}d\xi\rightharpoonup \text{div}\int_{0}^{u}\sqrt{b(\xi)\phi(\xi)}d\xi,\ \ \ \ \text{in}\ \ L^{2}(\Omega\times[0,T]\times\O).
\end{equation}
This completes the proof of \eqref{div_L2}. Concerning the chain rule formula \eqref{chain_rule}, since $u_{\kappa}$ is a kinetic solution, we know \eqref{chain_rule} holds for each $u_{\kappa}$, and so it holds for $u$ as a consequence of \eqref{weak_convergence_L2}.

\section{Invariant measure}
\label{sec_inv}

According to Theorem \ref{main_thm}, we have that for every $T>0$ and $u_{0}\in L^{1}(\O)$ there exists a unique generalized kinetic solution $u\in L^{1}(\Omega;L^{\infty}([0,T];L^{1}(\O)))$ which has almost surely continuous contrajectory in $L^{1}(\O)$. This allows us to define $P_{t}$ the transition semigroup associated to equation \eqref{limit_para} by 
\begin{equation*}
	P_{t}\varphi(z):=\E\varphi(u^{z}(t)),\ \ \ t\geq0,\ \ \ z\in L^{1}(\mathcal{O}),
\end{equation*}
for every $\varphi\in C_{b}(L^{1}(\mathcal{O}))$, where $u^{z}$ is the unique generalized kinetic solution with initial data $z$.

In this section, we will study invariant measures for $P_{t}$ in the case when the diffusion $b$ in \eqref{limit_para} is non-degenerate. Moreover, we assume that condition \eqref{polynomial_sigma} holds.

\subsection{Existence of an invariant measure of $P_{t}$ in $L^{1}(\mathcal{O})$}

Our result on the existence of an invariant measure for $P_{t}$ is given in the following theorem.

\begin{thm}\label{inv_existence}
	Assume Hypotheses \ref{Hypothesis1}, \ref{Hypothesis3}\ref{nondeg_condition} and condition \eqref{polynomial_sigma} hold. Then $P_{t}$, $t\geq0$, admits at least one invariant measure $\nu$ in $L^{1}(\mathcal{O})$. Moreover, $\text{supp}(\nu)\subset H^{1}$.
\end{thm}

\begin{proof}
First we remark that the $L^{1}$-contraction property \eqref{L1_contraction_kinetic} implies, in particular, that transition semigroup $P_{t}$, $t\geq0$ is Feller in $L^{1}(\mathcal{O})$.

 Now let $u_{0}\in L^{2\theta}(\O)$ and  $u\in L^{2\theta}(\Omega;L^{\infty}([0,T];L^{2\theta}(\O)))\cap L^{2}(\Omega;L^{2}([0,T];H^{1}))$ be the unique weak solution of equation \eqref{limit_para}. As a consequence of the It{\^o} formula, we have
\begin{equation*}
	\begin{array}{ll}
		\ds{d\norm{u(t)}_{H}^{2} \leq\Big(-2b_{0}\norm{u(t)}_{H^{1}}^{2}dt+\norm{\sigma(u(t))}_{L_{2}(\mathcal{U},H)}^{2}\Big)dt+ 2\inner{u(t),\sigma(u(t)) dw^{Q}(t)}_{H} }
	\end{array}
\end{equation*}
Due to condition \eqref{polynomial_sigma}, we have for every small $\delta>0$
\begin{equation}\label{noise_term_estimate}
	\norm{\sigma(u(t))}_{L_{2}(\mathcal{U},H)}^{2}ds\leq (\lambda^{2}+\delta) \norm{u(t)}_{H}^{2}+c(\delta).
\end{equation}
Since $\lambda^{2}<2\alpha_{1}b_{0}$, if we take the expectation, then there exist some constants $\tilde{\lambda}>0$ and $c>0$ such that for every $t\geq0$
\begin{equation*}
	\frac{d}{dt}\ \E\norm{u(t)}_{H}^{2}+\tilde{\lambda}\E\norm{u(t)}_{H^{1}}^{2}\leq c,
\end{equation*}
which implies that for every $t\geq0$ 
\begin{equation}\label{energy_estimate_inv}
	\E\norm{u(t)}_{H}^{2}\leq c\Big(1+e^{-\tilde{\lambda}t}\norm{u_{0}}_{H}^{2}\Big)\ \ \ \ \ \text{and}\ \  \ \ \ \E\int_{0}^{t}\norm{u(s)}_{H^{1}}^{2}ds\leq c\Big(t+\norm{u_{0}}_{H}^{2}\Big).
\end{equation}
Then, if for every $R>0$ and $t>0$ we denote 
\begin{equation*}
	B_{R}:=\big\{x\in L^{1}(\O):\norm{x}_{H^{1}}\leq R\big\},\ \ \ \ \ R_{t}:=\frac{1}{t}\int_{0}^{t}P_{s}^{\ast}\delta_{0}\ ds,
\end{equation*}
then thanks to \eqref{energy_estimate_inv}, for every $R>0$ and $t>0$ we have 
\begin{equation}\label{proof_support}
	R_{t}(B_{R}^{c})=\frac{1}{t}\int_{0}^{t}\P\big(\norm{u^{0}(s)}_{H^{1}}>R\big)ds\leq \frac{c}{R^{2}}
\end{equation} 
Since the embedding $H^{1}\hookrightarrow L^{1}(\mathcal{O})$ is compact, this implies that the family of probability measures $\{R_{t}\}_{t>0}$ is tight in $L^{1}(\mathcal{O})$. By the Krylov-Bogoliubov theorem, we conclude that the semigroup $P_{t}$ admits an invariant measure $\nu$ in $L^{1}(\mathcal{O})$. Moreover, as a direct consequence of \eqref{proof_support}, $\text{supp}(\nu)\subset H^{1}$.

\end{proof}

\subsection{Uniqueness of the invariant measure of $P_{t}$ in $L^{1}(\mathcal{O})$}

In the study on the uniqueness of invariant measure, we assume that the noise in equation \eqref{limit_para} is additive; that is, $\sigma_{i}(x)=\sigma_{i}(x,r)$, $\forall r\in\mathbb{R}$. In this case, Hypothesis \ref{Hypothesis1} implies that $\sigma_{i}\in H^{1}\cap L^{\infty}(\O)$, with 
\begin{equation}\label{sigma_additive_condition}
	\sum_{i=1}^{\infty}\norm{\sigma_{i}}_{L^{\infty}(\O)}^{2}<+\infty,\ \ \ \ \  \text{Tr}_{H}\Big[\big((-A)^{\frac{1}{2}}\sigma \big)\big((-A)^{\frac{1}{2}}\sigma \big)^{\ast}\Big]=\sum_{i=1}^{\infty}\norm{\nabla\sigma_{i}}_{H}^{2}\leq c\sum_{i=1}^{\infty}\lambda_{i}^{2}<+\infty.
\end{equation} 
 We need to remark that in this case, from Remark \ref{L1_almost_surely_weak_solution} and our constuction of generalized kinetic solutions (see Section \ref{sec_proof_main_thm}), one can see that $L^{1}$-contraction \eqref{L1_contraction_kinetic} holds for generalized kinetic solutions in the almost sure sense. Namely, if $u_{1},u_{2}$ are generalized kinetic solutions to \eqref{limit_para} with initial data $u_{1,0},u_{2,0}\in L^{1}(\O)$, respectively, then it holds that for every $t\in[0,T]$
\begin{equation}\label{L1_contraction_almost_surely}
	\norm{(u_{1}(t)-u_{2}(t))^{+}}_{L^{1}(\O)}\leq \norm{(u_{1,0}-u_{2,0})^{+}}_{L^{1}(\O)},\ \ \ \P\text{-a.s.}.
\end{equation}
As a particular case of Hypothesis \ref{Hypothesis3}\ref{nondeg_condition}, we will assume that the diffusion $b$ is bounded; that is, $\theta=1$ in \eqref{nondeg_growth}. In other words, we assume that $b\in C^{1}(\mathbb{R})$ and there exist some constants $b_{0},b_{1}>0$ such that 
\begin{equation*}
	b_{0}\leq b(r)\leq b_{1},\ \ \ \ r\in\mathbb{R}.
\end{equation*} 

We start with the following irreducibility condition uniformly with respect to the initial condition on bounded sets of $H$.
\begin{prop}[Irreducibility]\label{irreducible}
	Assume Hypotheses \ref{Hypothesis1} and \ref{Hypothesis3}\ref{nondeg_condition}. In addition, we assume the noise in \eqref{limit_para} is additive and the diffusion $b$ is bounded. Then for every $M,\epsilon>0$, there exists $t_{\ast}=t_{\ast}(M,\epsilon)>0$ such that
	\begin{equation}
		\inf_{\norm{z}_{H}\leq M}\P\big(\norm{u^{z}(t)}_{H}<\epsilon\big)>0,\ \ \ t\geq t_{\ast},
	\end{equation}
	where $u^{z}$ is the unique weak solution of equation \eqref{limit_para} with the initial data $z\in H$.
\end{prop}

\begin{proof}
	If we define
	\begin{equation*}
		W_{A}(t):=\int_{0}^{t}e^{(t-s)A} \sigma dw(s),\ \ \ \forall t\geq0,
	\end{equation*}
	where $A:=\Delta$ is the Laplacian introduced in Section \ref{sec_assumption}, then it is well-known that
	\begin{equation*}
		W_{A}\in L^{2}(\Omega;C([0,T];H))\bigcap L^{2}([0,T];H^{1})
	\end{equation*}
	 is the unqiue weak solution of problem
	\begin{equation*}
		\le\{\begin{array}{l}
			\ds{\partial_{t}\xi(t,x)=\Delta\xi(t,x)+\sigma\partial_{t}w(t,x), }\\[10pt]
			\ds{\xi(0,x)=0,\ \ \ \xi(t,\cdot)|_{\partial\mathcal{O}}=0. }
		\end{array}\r.
	\end{equation*}
  Moreover, thanks to \eqref{sigma_additive_condition}, it is not hard to show that the trajectories of $W_{A}$ are in $C^{\delta}([0,T];H^{1})$ for arbitrary $\delta\in[0,1/2)$ (see \cite[Section 5.4]{daprato2014}). Now, if we denote $\eta^{z}(t):=u^{z}(t)-W_{A}(t)$, then we know $\eta^{z}\in C([0,T];H)\bigcap L^{2}([0,T];H^{1})$, $\mathbb{P}$-a.s., solves the following deterministic problem 
		\begin{equation*}
		\le\{\begin{array}{l}
			\ds{\partial_{t}\eta(t,x)=\text{div}\big(b(u^{z}(t,x))\nabla\eta(t,x)\big)+\text{div}\big[(b(u^{z}(t,x))-I)\nabla W_{A}(t,x)\big], }\\[10pt]
			\ds{\eta(0,x)=z(x),\ \ \ \eta(t,\cdot)|_{\partial\mathcal{O}}=0. }
		\end{array}\r.
	\end{equation*}
	 This implies that for every $t>0$,
	\begin{equation*}
		\begin{array}{ll}
			\ds{\frac{1}{2}\frac{d}{dt}\norm{\eta^{z}(t)}_{H}^{2} \leq -b_{0}\norm{\eta^{z}(t)}_{H^{1}}^{2}-\Inner{(b(u^{z}(t))-I)\nabla W_{A}(t),\nabla\eta^{z}(t)}_{H} }
		\end{array}
	\end{equation*}
	Since we are assuming $b$ is bounded, so that for any small $\delta>0$
	\begin{equation*}
		\big\lvert \Inner{(b(u^{z}(t))-I)\nabla W_{A}(t),\nabla\eta^{z}(t)}_{H}\big\rvert\leq \delta\norm{\eta^{z}(t)}_{H^{1}}^{2}+c(\delta)\norm{W_{A}(t)}_{H^{1}}^{2}.
	\end{equation*}
	Hence, there exists for some constants $c>0$ and $\overline{\lambda}>0$ such that
	\begin{equation*}
		\norm{\eta^{z}(t)}_{H}^{2}\leq c\Big(e^{-\overline{\lambda}t}\norm{z}_{H}^{2}+\sup_{s\in[0,t]}\norm{W_{A}(s)}_{H^{1}}^{2}\Big),
	\end{equation*}
	and so that
	\begin{equation}\label{irred_key}
		\norm{u^{z}(t)}_{H}^{2}\leq 2\Big(\norm{\eta^{z}(t)}_{H}^{2}+\norm{W_{A}(t)}_{H}^{2}\Big)\leq C\Big(e^{-\overline{\lambda}t}\norm{z}_{H}^{2}+\sup_{s\in[0,t]}\norm{W_{A}(s)}_{H^{1}}^{2}\Big).
	\end{equation}
	This implies that for every $M,\epsilon>0$, there exists $t_{\ast}=t_{\ast}(M,\epsilon)>0$ such that
	\begin{equation*}
		\inf_{\norm{z}_{H}\leq M}\P \big(\norm{u^{z}(t)}_{H}<\epsilon\big)\geq \P\Big(\sup_{s\in[0,t]}\norm{W_{A}(s)}_{H^{1}}<C_{M}\epsilon\Big),\ \ \ t\geq t_{\ast},
	\end{equation*}
	for some constant $C_{M}>0$ depending on $M$. Finally, for all $t\geq t_{\ast}$, since $W_{A}(\cdot)$ is full in $C([0,t];H^{1})$, we have
	\begin{equation*}
		\inf_{\norm{z}_{H}\leq M}\P \big(\norm{u^{z}(t)}_{H}<\epsilon\big)\geq \P\Big(\sup_{s\in[0,t]}\norm{W_{A}(s)}_{H^{1}}<C_{M}\epsilon\Big)>0.
	\end{equation*}
\end{proof}

Thanks to \eqref{irred_key}, we can easily obtain the following corollary, which shows that the weak solution is small if the noise is small.
\begin{cor}\label{control_ball}
	 For every $M>0$ and every $\epsilon>0$, there exists $T=T(\epsilon,M)>0$ and $\delta=\delta(\epsilon)>0$ such that
	\begin{equation*}
		\frac{1}{T}\int_{0}^{T}\norm{u^{z}(t)}_{H}^{2}dt<\epsilon,\ \ \P\text{-a.s.},
	\end{equation*}
	given that
	\begin{equation*}
		\norm{z}_{H}\leq M\ \ \ \text{and}\ \ \ \sup_{t\in[0,T]}\norm{W_{A}(t)}_{H^{1}}\leq \delta,\ \ \P\text{-a.s.}.
	\end{equation*}
\end{cor}

Next, we will show that any weak solution enters a ball in $H$ of some fixed radius in finite time, which is a modification of the result in \cite[Section 4.1]{debussche15}.

\begin{lem}[Finite time to enter a ball]\label{finite_time}
	Assume Hypotheses \ref{Hypothesis1} and \ref{Hypothesis3}\ref{nondeg_condition}. In addition, we assume the noise in \eqref{limit_para} is additive and the diffusion $b$ is bounded. Then there exists a constant $K_{0}>0$ such that for every $u_{1,0},u_{2,0}\in  H$ and $T>0$, if we define the following stopping times:
	\begin{equation*}
		\tau_{l}=\inf\Big\{t\geq\tau_{l-1}+T:\norm{u_{1}(t)}_{H}^{2}+\norm{u_{2}(t)}_{H}^{2}\leq K_{0} \Big\},\ \ \ \tau_{0}=0,
	\end{equation*}
	where $u_{1},u_{2}$ are weak solutions to equation \eqref{limit_para} with initial data $u_{1,0},u_{2,0}$, respectively, then we have 
	\begin{equation*}
		\P(\tau_{l}<+\infty)=1,\ \ \ l\in\mathbb{N}.
	\end{equation*}
\end{lem}

\begin{proof}
	
	From the proof of Theorem \ref{inv_existence}, as a consequence of the It{\^o} formula and \eqref{sigma_additive_condition}, there exists some sonstant $c_{0}>0$ such that for every $t,T\geq0$,
	\begin{equation*}
		\begin{array}{ll}
			&\ds{\E\Bigg(\int_{t}^{t+T}\Big(\norm{u_{1}(s)}_{H}^{2}+\norm{u_{2}(s)}_{H}^{2}\Big)ds\Big|\mathcal{F}_{t}\Bigg) }\ds{\leq c_{0}\Big(\norm{u_{1}(t)}_{H}^{2}+\norm{u_{2}(t)}_{H}^{2}+T\Big)}\\
			\vs
			&\ds{\leq c_{0}\Bigg(\norm{u_{1,0}}_{H}^{2}+\norm{u_{2,0}}_{H}^{2}+\int_{0}^{t}\Inner{u_{1}(s),\sigma dw(s)}_{H}+\int_{0}^{t}\Inner{u_{2}(s),\sigma dw(s)}_{H}+T\Bigg) }
		\end{array}
	\end{equation*}
	We define the sequences $(t_{k})_{k\in\mathbb{N}}$ and $(r_{k})_{k\in\mathbb{N}}$ by
	\begin{equation*}
		t_{0}=0,\ \ \ t_{k+1}=t_{k}+r_{k},\ \ \ k\geq1,
	\end{equation*}
	where the sequence $(r_{k})_{k\in\mathbb{N}}$ will be determined later. Then if we set $K_{0}:=4c_{0}$, we have for each $k\in\mathbb{N}$ 
	\begin{equation*}
		\begin{array}{ll}
			&\ds{\P\Bigg(\inf_{s\in[t_{k},t_{k+1}]}\Big(\norm{u_{1}(s)}_{H}^{2}+\norm{u_{2}(s)}_{H}^{2}\Big)\geq K_{0}\Big|\mathcal{F}_{t_{k}} \Bigg) }\ds{\leq\P\Bigg(\frac{1}{r_{k}}\int_{t_{k}}^{t_{k}+r_{k}}\Big(\norm{u_{1}(s)}_{H}^{2}+\norm{u_{2}(s)}_{H}^{2}\Big)ds\geq K_{0}\Big|\mathcal{F}_{t_{k}}\Bigg) }\\
			\vs
			&\ds{\leq \frac{1}{4r_{k}}\Big(\norm{u_{1,0}}_{H}^{2}+\norm{u_{2,0}}_{H}^{2}+1\Big) }+\frac{1}{4r_{k}}\Big(\int_{0}^{t_{k}}\Inner{u_{1}(s),\sigma dw(s)}_{H}+\int_{0}^{t_{k}}\Inner{u_{2}(s),\sigma dw(s)}_{H}\Big)\\
			\vs
			&\ds{\leq \frac{1}{2}+\frac{1}{4r_{k}}\Big(\int_{0}^{t_{k}}\Inner{u_{1}(s),\sigma dw(s)}_{H}+\int_{0}^{t_{k}}\Inner{u_{2}(s),\sigma dw(s)}_{H}\Big) },
		\end{array}
	\end{equation*}
	where we have chosen $(r_{k})_{k\in\mathbb{N}}$ such that
	\begin{equation}\label{r_condition1}
		\frac{1}{2r_{k}}\Big(\norm{u_{1,0}}_{H}^{2}+\norm{u_{2,0}}_{H}^{2}+1\Big)\leq 1,\ \ \ k\in\mathbb{N}.
	\end{equation}
	Next, let us consider events
	\begin{equation*}
		A_{k}=\Big\{\inf_{s\in[t_{l},t_{l+1}]}\Big(\norm{u_{1}(s)}_{H}^{2}+\norm{u_{2}(s)}_{H}^{2}\Big)\geq K_{0},\ \ \ l=0,\dots,k-1\Big\},
	\end{equation*}
	 then since it follows from \eqref{energy_estimate_inv} that 
	\begin{equation*}
			\E\Bigg(\int_{0}^{t_{k}}\Inner{u_{i}(s),\sigma dw(s)}_{H}\Bigg)^{2}\leq c\text{Tr}_{H}\sigma^{2}\int_{0}^{t_{k}}\E\norm{u_{i}(s)}_{H}^{2}ds\leq ct_{k}\big(1+\norm{u_{i,0}}_{H}^{2}\big),\ \ \ i=1,2,
	\end{equation*}
	by the Young inequality, there exists some constant $c_{1}>0$ such that
	\begin{equation*}
		\begin{array}{ll}
			&\ds{\P(A_{k+1})\leq \frac{1}{2}\P(A_{k})+\frac{1}{4r_{k}}\E\Bigg(\mathbf{1}_{A_{k}}\int_{0}^{t_{k}}\Inner{u_{1}(s),\sigma dw(s)}_{H}\Bigg)+\frac{1}{4r_{k}}\E\Bigg(\mathbf{1}_{A_{k}}\int_{0}^{t_{k}}\Inner{u_{2}(s),\sigma dw(s)}_{H}\Bigg)  }\\
			\vs
			&\ds{\quad\quad\quad\quad\leq \frac{3}{4}\P(A_{k})+\frac{c_{1}}{r_{k}^{2}}\Big(1+\norm{u_{1,0}}_{H}^{2}+\norm{u_{2,0}}_{H}^{2}\Big)t_{k} },
		\end{array}
	\end{equation*}
	Hence, if we choose $(r_{k})_{k\in\mathbb{N}}$ be satisfying both \eqref{r_condition1} and that
	\begin{equation}\label{r_condition2}
		\frac{Ct_{k}}{r_{k}^{2}}\Big(1+\norm{u_{1,0}}_{H}^{2}+\norm{u_{2,0}}_{H}^{2}\Big)\leq \Big(\frac{3}{4}\Big)^{k},\ \ \ k\in\mathbb{N},
	\end{equation}
	then we obtain 
	\begin{equation*}
		\P(A_{k+1})\leq \frac{3}{4}\P(A_{k})+\Big(\frac{3}{4}\Big)^{k},\ \ \ k\in\mathbb{N},
	\end{equation*}
	and this gives us that
	\begin{equation*}
		\P(A_{k})\leq k\Big(\frac{3}{4}\Big)^{k},\ \ \ k\geq0.
	\end{equation*}
	Hence, by the Borel-Cantelli Lemma, it follows that 
	\begin{equation*}
		K_{0}:=\inf\Big\{k\geq0:\inf_{s\in[t_{k},t_{k+1}]}\Big(\norm{u_{1}(s)}_{H}^{2}+\norm{u_{2}(s)}_{H}^{2}\Big)\leq K_{0}\Big\}
	\end{equation*}
	is finite almost surely. This implies that the stopping time
	\begin{equation*}
		\tau:=\inf\Big\{t\geq0:\norm{u_{1}(s)}_{H}^{2}+\norm{u_{2}(s)}_{H}^{2}\leq K_{0}\Big\}
	\end{equation*}
	is also finite $\P$-almost surely since $\tau\leq t_{k_{0}+1}$. Therefore, for every $T>0$, the stopping times $\tau_{l}$, $l\in\mathbb{N}$, are also almost surely finite. 
\end{proof}

Finally, we are ready to state our conclusion on the uniqueness.

\begin{thm}\label{uniqueness_inv}
	Assume Hypotheses \ref{Hypothesis1} and \ref{Hypothesis3}\ref{nondeg_condition}. In addition, we assume the noise in \eqref{limit_para} is additive and the diffusion $b$ is bounded. Then, if $u_{1},u_{2}$ are generalized kinetic solutons of equation \eqref{limit_para} with initial data $u_{1,0},u_{2,0}\in L^{1}(\O)$, respectively, then 
	\begin{equation}
		\lim_{t\to+\infty}\norm{u_{1}(t)-u_{2}(t)}_{L^{1}(\mathcal{O})}=0,\ \ \P\text{-a.s.}.
	\end{equation} 
\end{thm}

\begin{proof} Let $u_{1},u_{2}$ be generalized kinetic solutions for equation \eqref{limit_para} with initial data $u_{1,0},u_{2,0}\in L^{1}(\O)$, respectively. Fix any $\epsilon>0$, we choose $\hat{u}_{1,0},\hat{u}_{2,0}\in H$ such that
	\begin{equation*}
		\norm{\hat{u}_{i,0}-u_{i,0}}_{L^{1}(\mathcal{O})}\leq \epsilon/4,\ \ \ i=1,2,
	\end{equation*}   
	where we denoted by $\hat{u}_{i}$ the unique weak solution of equation \eqref{limit_para} with initial data $\hat{u}_{i,0}$, for $i=1,2$, relatively.
	
	Let $K_{0}>0$ be the constant introduced in Lemma \ref{finite_time}. Then by Corollary \ref{control_ball} and thanks to the $L^{1}$-contraction property \eqref{L1_contraction_almost_surely}, we can choose $T>0$ and $\delta>0$ depending both only on $\epsilon$ such that, if we define stopping times $(\tau_{l})_{l\in\mathbb{N}}$ by
		\begin{equation*}
		\tau_{l}=\inf\Big\{t\geq\tau_{l-1}+T:\norm{\hat{u}_{1}(t)}_{H}^{2}+\norm{\hat{u}_{2}(t)}_{H}^{2}\leq K_{0} \Big\},\ \ \ \tau_{0}=0,
	\end{equation*}
	then $\tau_{l}$ is almost surely finite from Lemma \ref{finite_time}, and 
	\begin{equation*}
		\begin{array}{ll}
			\ds{\P\Bigg(\frac{1}{T}\int_{\tau_{l}}^{\tau_{l}+T}\norm{u_{1}(t)-u_{2}(t)}_{L^{1}(\mathcal{O})}dt\leq\epsilon\Big|\mathcal{F}_{\tau_{l}}\Bigg) }
			&\ds{\geq\P\Bigg(\frac{1}{T}\int_{\tau_{l}}^{\tau_{l}+T}\norm{\hat{u}_{1}(t)-\hat{u}_{2}(t)}_{L^{1}(\mathcal{O})}dt\leq \epsilon/2\Big|\mathcal{F}_{\tau_{l}}\Bigg) }\\
			\vs
			&\ds{\geq\P\Bigg(\frac{1}{T}\int_{\tau_{l}}^{\tau_{l}+T}\norm{\hat{u}_{1}(t)-\hat{u}_{2}(t)}_{H}dt\leq C\epsilon\Big|\mathcal{F}_{\tau_{l}}\Bigg) }\\
			\vs
			&\ds{\geq\P\Bigg(\sup_{t\in[\tau_{l},\tau_{l}+T]}\norm{W_{A}(t)-W_{A}(\tau_{l})}_{H^{1}}\leq \delta\Big|\mathcal{F}_{\tau_{l}}\Bigg) },
		\end{array}
	\end{equation*}
 where $W_{A}(t)$ is the random process defined in Proposition \ref{irreducible}. By the strong Markov property, the term on the right-hand side $$\P\Bigg(\sup_{t\in[\tau_{l},\tau_{l}+T]}\norm{W_{A}(t)-W_{A}(\tau_{l})}_{H^{1}}\leq \delta\Big|\mathcal{F}_{\tau_{l}}\Bigg)$$
is non-random and independent on $l$, which depends only on $\epsilon$. Moreover, it is clearly positive. We denote it by some scalar $\lambda=\lambda(\epsilon)$. Now for every $l_{0}\in\mathbb{N}$ and $k\in\mathbb{N}$, we have
\begin{equation*}
	\P\Bigg(\frac{1}{T}\int_{\tau_{l}}^{\tau_{l}+T}\norm{u_{1}(s)-u_{2}(s)}_{L^{1}(\mathcal{O})}ds\geq\epsilon,\ \ \text{for }l=l_{0},\dots,l_{0}+k \Bigg)\leq (1-\lambda)^{k},
\end{equation*}
and this implies that
\begin{equation}\label{uniqueness_key}
	\begin{array}{ll}
		&\ds{\P\Bigg(\lim_{l\to\infty}\frac{1}{T}\int_{\tau_{l}}^{\tau_{l}+T}\norm{u_{1}(s)-u_{2}(s)}_{L^{1}(\mathcal{O})}ds\geq\epsilon \Bigg)}\\
		\vs
		&\ds{=\P\Bigg(\exists l_{0}\in\mathbb{N}:\frac{1}{T}\int_{\tau_{l}}^{\tau_{l}+T}\norm{u_{1}(s)-u_{2}(s)}_{L^{1}(\mathcal{O})}ds\geq\epsilon,\ \ \text{for }l\geq l_{0}\Bigg)=0 }
	\end{array}
\end{equation}
Note that \eqref{L1_contraction_almost_surely} yields the mapping $t\mapsto\norm{\rho^{1}(t)-\rho^{2}(t)}_{L^{1}(\mathcal{O})}$ is $\P$-almost surely non-increasing, so the limit in the left-hand side of \eqref{uniqueness_key} exists. Therefore, \eqref{uniqueness_key} deduces that
\begin{equation*}
	\P\Big(\lim_{t\to\infty}\norm{u_{1}(t)-u_{2}(t)}_{L^{1}(\mathcal{O})}\geq\epsilon\Big)=0.
\end{equation*}
Due to the arbitrariness of $\epsilon>0$, we complete the proof.
\end{proof}

As a consequence of Theorem \ref{uniqueness_inv}, the uniqueness of invariant measure for $P_{t}$ in $L^{1}(\O)$ holds.
\begin{cor}\label{inv_uniqueness}
	Assume Hypotheses \ref{Hypothesis1} and \ref{Hypothesis3}\ref{nondeg_condition}. In addition, we assume the noise in \eqref{limit_para} is additive and the diffusion $b$ is bounded. Then there is at most one invariant measure for $P_{t}$ associated with equation \eqref{limit_para} in $L^{1}(\mathcal{O})$. 
\end{cor}

Finally, combining Theorem \ref{inv_existence} together with Corollary \ref{inv_uniqueness}, we obtain Theorem \ref{ergodic_thm}.
 


\appendix

\section{Proof of Lemma \ref{contraction}}
\label{sec_appendix}

Our proof is a modification of \cite[Proof of Proposition 13]{debussche14}, where A. Debussche and J. Vovelle proved a version of the doubling of variables for kinetic solutions of a scalar conservation law with a stochastic forcing on the torus $\mathbb{T}^{d}$.

Let $u_{1},u_{2}$ be generalized kinetic solutions to \eqref{limit_para}. We denote $h_{1}=\mathbf{1}_{u_{1}>\xi}$, $h_{2}=\mathbf{1}_{u_{2}>\xi}$ with the corresponding Young measures $\nu^{1}=\delta_{u_{1}}$, $\nu^{2}=\delta_{u_{2}}$ and  the corresponding kinetic measures $m_{1},m_{2}$, respectively. Let $\varphi_{1}\in C^{\infty}_{c}(\mathcal{O}_{x}\times\mathbb{R}_{\xi})$  and $\varphi_{2}\in C^{\infty}_{c}(\mathcal{O}_{y}\times\mathbb{R}_{\zeta})$. Then, thanks to \eqref{kinetic_test}, we have $\P$-a.s., 
\begin{equation*}
	\dInner{h_{1}^{+}(t),\varphi_{1}}=\dInner{m_{1}^{\ast},\partial_{\xi}\varphi_{1}}[0,t]+H_{1}(t),\ \ \ \ H_{1}(t):=\sum_{k=1}^{\infty}\int_{0}^{t}\int_{\mathcal{O}}\int_{\mathbb{R}}\varphi_{1}\sigma_{k}(x,\xi)d\nu_{x,s}^{1}(\xi)dxd\beta_{k}(s),
\end{equation*}
where
\begin{equation*}
	\begin{array}{ll}
		&\ds{\dInner{m_{1}^{\ast},\partial_{\xi}\varphi_{1}}([0,t])=\dInner{h_{1,0},\varphi_{1}}\delta_{0}([0,t])+\int_{0}^{t}\dInner{h_{1},b\cdot \nabla_{x}^{2}\varphi_{1}}ds }\\
		\vs
		&\ds{\quad\quad\quad\quad +\frac{1}{2}\int_{0}^{t}\int_{\mathcal{O}}\int_{\mathbb{R}}\partial_{\xi}\varphi_{1}\Sigma^{2}(x,\xi)d\nu_{x,s}^{1}(\xi)dxds-\dInner{m_{1},\partial_{\xi}\varphi_{1}}([0,t]) }.
	\end{array}
\end{equation*}
Similarly, we have $\P$-a.s.,
\begin{equation*}
	\dInner{\overline{h}_{2}^{+}(t),\varphi_{2}}=\dInner{m_{2}^{\ast},\partial_{\zeta}\varphi_{2}}[0,t]+\overline{H}_{2}(t),\ \ \ \ \overline{H}_{2}(t):=-\sum_{k=1}^{\infty}\int_{0}^{t}\int_{\mathcal{O}}\int_{\mathbb{R}}\varphi_{2}\sigma_{k}(y,\zeta)d\nu_{y,s}^{2}(\zeta)dyd\beta_{k}(s),
\end{equation*}
where
\begin{equation*}
	\begin{array}{ll}
		&\ds{\dInner{\overline{m}_{2}^{\ast},\partial_{\zeta}\varphi_{2}}([0,t])=\dInner{\overline{h}_{2,0},\varphi_{2}}\delta_{0}([0,t])+\int_{0}^{t}\dInner{\overline{h}_{2},b\cdot \nabla_{y}^{2}\varphi_{2}}ds }\\
		\vs
		&\ds{\quad\quad\quad\quad -\frac{1}{2}\int_{0}^{t}\int_{\mathcal{O}}\int_{\mathbb{R}}\partial_{\zeta}\varphi_{2}\Sigma^{2}(y,\zeta)d\nu_{y,s}^{2}(\zeta)dxds+\dInner{m_{2},\partial_{\zeta}\varphi_{2}}([0,t]) }.
	\end{array}
\end{equation*}
Next, using the It{\^o} formula for the product $H_{1}(t)\overline{H}_{2}(t)$ and taking integration by parts, we have
\begin{equation*}
	\begin{array}{ll}
		&\ds{\int_{\mathcal{O}^{2}}\int_{\mathbb{R}^{2}}\Big(h_{1}^{+}(t)\overline{h}_{2}^{+}(t)-h_{1,0}\overline{h}_{2,0}\Big)\varphi_{1}(x,\xi)\varphi_{2}(y,\zeta) d\xi d\zeta dxdy }\\
		\vs
		&\ds{=\int_{0}^{t}\dInner{h_{1}^{+}(s),\varphi_{1}}d\dInner{\overline{m}^{\ast}_{2},\partial_{\zeta}\varphi_{2}}(s)+\int_{0}^{t}\dInner{\overline{h}_{2}^{+}(s),\varphi_{2}}d\dInner{m^{\ast}_{1},\partial_{\xi}\varphi_{1}}(s) }\\
		\vs
		&\ds{\quad+\int_{0}^{t}\dInner{h_{1}^{+}(s),\varphi_{1}}d\overline{H}_{2}(s)+\int_{0}^{t}\dInner{\overline{h}_{2}^{+}(s),\varphi_{2}}dH_{1}(s) }\\
		\vs
		&\ds{\quad -\sum_{k=1}^{\infty}\int_{0}^{t}\int_{\mathcal{O}^{2}}\int_{\mathbb{R}^{2}}\varphi_{1}(x,\xi)\varphi_{2}(y,\zeta)\sigma_{k}(x,\xi)\sigma_{k}(y,\zeta)d\nu_{x,s}^{1}(\xi)d\nu_{y,s}^{2}(\zeta)dxdyds}
	\end{array}
\end{equation*}
Now, let $\varphi_{3}\in C^{\infty}_{c}(\mathbb{R}_{\eta})$ and $\alpha(x,\xi,y,\zeta,\eta)=\varphi_{1}(x,\xi)\varphi_{2}(y,\zeta)\varphi_{3}(\eta)$. Then, it follows that 
\begin{equation}\label{double_ana1}
	\int\Big(h_{1}^{+}(t)\overline{h}_{2}^{+}(t)-h_{1,0}\overline{h}_{2,0}\Big)\alpha =:\sum_{i=1}^{5}I_{i},\ \ \ \P\text{-a.s.},
\end{equation}
where the integral is taken with respect to $(x,y,\xi,\zeta,\eta)\in \O^{2}\times\mathbb{R}^{3}$ by convention as in Section \ref{L1},
\begin{equation*}
	I_{1}:=\int_{0}^{t}\int h_{1}(s)\overline{h}_{2}(s)\Big(b(\xi)\cdot \nabla_{x}^{2}+b(\zeta)\cdot \nabla_{y}^{2}\Big)\alpha ,
\end{equation*}

\begin{equation*}
	\begin{array}{ll}
		\ds{I_{2}}
		&\ds{=-\frac{1}{2}\int_{0}^{t}\int h_{1}(s)\partial_{\zeta}\alpha\ \Sigma^{2}(y,\zeta)d\nu_{y,s}^{2}(\zeta)+\frac{1}{2}\int_{0}^{t}\int\overline{h}_{2}(s)\partial_{\xi}\alpha\ \Sigma^{2}(x,\xi)d\nu_{x,s}^{1}(\xi) }\\
		\vs
		&\ds{\quad -\sum_{k=1}^{\infty}\int_{0}^{t}\int\alpha\  \sigma_{k}(x,\xi)\sigma_{k}(y,\zeta)d\nu_{x,s}^{1}(\xi)d\nu_{y,s}^{2}(\zeta)},
	\end{array}
\end{equation*}

\begin{equation*}
	I_{3}=\int_{0}^{t}\int h_{1}^{-}(s)\partial_{\zeta}\alpha dm_{2}(y,s,\zeta),\ \ \ \  I_{4}=-\int_{0}^{t}\int\overline{h}_{2}^{+}(s)\partial_{\xi}\alpha dm_{1}(x,s,\xi),
\end{equation*}
and
\begin{equation*}
		I_{5}=-\sum_{k=1}^{\infty}\int_{0}^{t}\int h_{1}(s)\alpha \sigma_{k}(y,\zeta)d\nu_{y,s}^{2}(\zeta)d\beta_{k}(s)+\sum_{k=1}^{\infty}\int_{0}^{t}\int\overline{h}_{2}(s)\alpha \sigma_{k}(x,\xi)d\nu_{x,s}^{1}(\xi)d\beta_{k}(s).
\end{equation*}
By a standard density argument, we have \eqref{double_ana1} remains true for all test functions $\alpha\in C^{\infty}_{0}\big(\mathcal{O}_{x}\times\mathbb{R}_{\xi}\times\mathcal{O}_{y}\times\mathbb{R}_{\zeta}\times \mathbb{R_{\eta}}\big)$. Then taking $\alpha(x,\xi,y,\zeta,\eta)=\phi(x,y)K(\eta)\psi(\eta-\xi)\psi(\eta-\zeta)$, where $0\leq\phi\in C^{\infty}_{c}(\O^{2})$ and $\psi,K$ are defined in \eqref{mollifier}, from integration by parts we obtain that
\begin{equation}\label{I_1}
	I_{1}=\int_{0}^{t}\int h_{1}(s)\overline{h}_{2}(s)\Big(b(\xi)\nabla_{x}^{2}+b(\zeta)\nabla_{y}^{2}\Big)\phi(x,y)K(\eta)\psi(\eta-\xi)\psi(\eta-\zeta),
\end{equation}

\begin{equation}\label{I_2}
	\begin{array}{ll}
		\ds{I_{2}}
		&\ds{=\frac{1}{2}\sum_{k=1}^{\infty}\int_{0}^{t}\int\phi(x,y)K(\eta)\psi(\eta-\xi)\psi(\eta-\zeta)\big\lvert \sigma_{k}(x,\xi)-\sigma_{k}(y,\zeta)\big\rvert^{2}d\nu_{x,s}^{1}(\xi)d\nu_{y,s}^{2}(\zeta) }\\
		\vs
		&\ds{\quad-\frac{1}{2}\int_{0}^{t}\int h_{1}(s)\phi(x,y)K'(\eta)\psi(\eta-\xi)\psi(\eta-\zeta)\Sigma^{2}(y,\zeta)d\nu_{y,s}^{2}(\zeta) }\\
		\vs
		&\ds{\quad+\frac{1}{2}\int_{0}^{t}\int\overline{h}_{2}(s)\phi(x,y)K'(\eta)\psi(\eta-\xi)\psi(\eta-\zeta)\Sigma^{2}(x,\xi)d\nu_{x,s}^{1}(\xi) },
	\end{array}
\end{equation}
and
\begin{equation}\label{I_5}
	\begin{array}{ll}
		\ds{I_{5}}
		&\ds{=\sum_{k=1}^{\infty}\int_{0}^{t}\int\overline{h}_{2}(s)\phi(x,y)K(\eta)\psi(\eta-\xi)\psi(\eta-\zeta)\Big(\sigma_{k}(x,\xi)-\sigma_{k}(y,\zeta)\Big)d\nu_{x,s}^{1}(\xi)d\beta_{k}(s) }\\
		\vs
		&\ds{\quad+\sum_{k=1}^{\infty}\int_{0}^{t}\int h_{1}(s)\overline{h}_{2}(s)\phi(x,y)K'(\eta)\psi(\eta-\xi)\psi(\eta-\zeta)\sigma_{k}(y,\zeta)d\beta_{k}(s) }.
	\end{array}
\end{equation}
Concerning about terms $I_{3}$ and $I_{4}$, from Definition \ref{def_gen_kinetic_solution} it follows that 
\begin{equation*}
	\begin{array}{ll}
		\ds{I_{3}}
		&\ds{\leq-\int_{0}^{t}\int\phi(x,y)K(\eta)\psi(\eta-\xi)d\nu_{x,s}^{1}(\xi) dn_{2}^{\psi(\eta-\cdot)}(y,s)  }\\
		\vs
		&\ds{\quad+\int_{0}^{t}\int h_{1}^{-}(s)\phi(x,y)K'(\eta)\psi(\eta-\xi)\psi(\eta-\zeta) dm_{2}(y,s,\zeta) }\\
		\vs
		&\ds{\leq-\int_{0}^{t}\int_{\mathcal{O}^{2}}\int_{\mathbb{R}}\phi(x,y)K(\eta)\psi(\eta-u_{1})\cdot\Big\lvert \nabla_{y}\int_{0}^{u_{2}}\sqrt{\psi(\eta-\zeta)b(\zeta)}d\zeta\Big\rvert^{2} }\\
		\vs
		&\ds{\quad+\int_{0}^{t}\int h_{1}^{-}(s)\phi(x,y)K'(\eta)\psi(\eta-\xi)\psi(\eta-\zeta) dm_{2}(y,s,\zeta) }
	\end{array}
\end{equation*}
Similarly, we obtain that  
\begin{equation*}
	\begin{array}{ll}
		\ds{I_{4}}
		&\ds{\leq-\int_{0}^{t}\int_{\mathcal{O}^{2}}\int_{\mathbb{R}}\phi(x,y)K(\eta)\psi(\eta-u_{2})\cdot\Big\lvert \nabla_{x}\int_{0}^{u_{1}}\sqrt{\psi(\eta-\xi)b(\xi)}d\xi\Big\rvert^{2} }\\
		\vs
		&\ds{\quad-\int_{0}^{t}\int\overline{h}_{2}^{+}(s)\phi(x,y)K'(\eta)\psi(\eta-\xi)\psi(\eta-\zeta) dm_{1}(x,s,\xi) }
	\end{array}
\end{equation*}
Hence, by making a use of the chain rule formulation \eqref{chain_rule} in Definition \ref{def_gen_kinetic_solution}, we have
\begin{equation}\label{I_3_I_4}
	\begin{array}{ll}
		\ds{I_{3}+I_{4}}
		&\ds{\leq 2\int_{0}^{t}\int_{\mathcal{O}^{2}}\int_{\mathbb{R}}\phi(x,y)K(\eta)\Big( \nabla_{x}\int_{0}^{u_{1}}\psi(\eta-\xi)\sqrt{b(\xi)}d\xi\Big)\cdot \Big(\nabla_{y}\int_{0}^{u_{2}}\psi(\eta-\zeta)\sqrt{b(\zeta)}d\zeta\Big) }\\
		\vs
		&\ds{\quad-\int_{0}^{t}\int\overline{h}_{2}^{+}(s)\phi(x,y)K'(\eta)\psi(\eta-\xi)\psi(\eta-\zeta) dm_{1}(x,s,\xi) }\\
		\vs
		&\ds{\quad+\int_{0}^{t}\int h_{1}^{-}(s)\phi(x,y)K'(\eta)\psi(\eta-\xi)\psi(\eta-\zeta) dm_{2}(y,s,\zeta) }\\
		\vs
		&\ds{= 2\int_{0}^{t}\int h_{1}(s)\overline{h}_{2}(s)\sqrt{b(\xi)b(\zeta)}\nabla_{x,y}^{2}\phi(x,y)K(\eta)\psi(\eta-\xi)\psi(\eta-\zeta) }\\
		\vs
		&\ds{\quad-\int_{0}^{t}\int \overline{h}_{2}^{+}(s)\phi(x,y)K'(\eta)\psi(\eta-\xi)\psi(\eta-\zeta) dm_{1}(x,s,\xi) }\\
		\vs
		&\ds{\quad+\int_{0}^{t}\int h_{1}^{-}(s)\phi(x,y)K'(\eta)\psi(\eta-\xi)\psi(\eta-\zeta) dm_{2}(y,s,\zeta) },
	\end{array}
\end{equation}
Finally, combining \eqref{I_1}-\eqref{I_3_I_4} and taking the expectation on both sides of \eqref{double_ana1}, we complete the proof of Lemma \ref{contraction}.

\end{document}